%% file: DOT-Arxiv.tex
\documentclass[10pt,a4paper]{article}
\usepackage[margin=3cm]{geometry}
\usepackage{amsfonts,amsmath,bm,dsfont,mathrsfs}
\usepackage{amssymb,amsthm}
\usepackage{enumitem,cases}
\usepackage[dvipsnames]{xcolor}
 
\usepackage{hyperref}[6.83]
\hypersetup{
  colorlinks = true,
  allcolors = blue,
  urlcolor = red,
}
\usepackage{tikz}
\usetikzlibrary{positioning}
\usepackage{graphicx}
\usepackage{subcaption,caption}
\usepackage[ruled,vlined]{algorithm2e}
\usepackage{pdflscape,afterpage,capt-of}
\usepackage{booktabs, rotating, multirow, array}

\definecolor{blue}{RGB}{57, 57, 235}
\definecolor{red}{RGB}{165, 0, 33}
\definecolor{green}{RGB}{57, 235, 57}
\usepackage{xcolor}
\usepackage{acronym}
\usepackage{framed}
\usepackage{mathrsfs}
\usepackage{bm}

\usepackage{cleveref}
\makeatletter
\newcommand{\refcheckize}[1]{%
  \expandafter\let\csname @@\string#1\endcsname#1%
  \expandafter\DeclareRobustCommand\csname relax\string#1\endcsname[1]{%
    \csname @@\string#1\endcsname{##1}\wrtusdrf{##1}}%
  \expandafter\let\expandafter#1\csname relax\string#1\endcsname
}
\makeatother

\newtheorem{example}{Example}
\newtheorem{theorem}{Theorem}
\newtheorem{proposition}{Proposition}
\newtheorem{remark}{Remark}
\numberwithin{lemma}{section}
\numberwithin{example}{section}
\numberwithin{theorem}{section}
\numberwithin{proposition}{section}
\numberwithin{remark}{section}

\definecolor{refblue}{RGB}{26,13,171}
\definecolor{newblue}{RGB}{13,13,217}
\definecolor{newred}{RGB}{221,13,13}

\let\cedilla\c
\def\b{\bm b}
\def\c{\bm c}
\def\e{\bm e}
\def\u{\bm u}

\def\v{\bm v}
\def\d{\bm d}
\def\q{\bm q}

\def\i{\bm i}
\def\p{\bm p}

\def\x{\bm x}
\def\y{\bm y}
\def\A{\mathcal A}
\def\B{\mathcal B}

\def\L{\mathcal L}
\def\I{\mathcal I}

\def\R{\mathds R}
\def\X{\mathcal X}

\def\P{\mathcal P}

\def\T{\mathcal T}

\def\z{\bm z}

\newcommand{\set}[2]{\left\{ #1 \mid  #2 \right\}}
\newcommand{\eqdef}{:=}
\newcommand{\G}{\mathcal{G}}
\newcommand{\F}{\mathcal{F}}
\DeclareMathOperator{\argmin}{arg \, min}

\newcommand{\dd}{\mathop{}\!{\rm d}}

\title{\Large \bf An efficient second-order cone programming approach for dynamic optimal transport on staggered grid discretization\thanks{This work was supported by
the National Key R \& D Program of China (No. 2021YFA001300),
the National Natural Science Foundation of China (No. 12271150),
the Hunan Provincial Natural Science Foundation of China (No. 2023JJ10001), and 
the Science and Technology Innovation Program of Hunan Province (No. 2022RC1190).}
\thanks{The code (MATLAB/C++) of the software is available at GitHub (\url{https://github.com/chlhnu/DOT-SOCP}).}
}

\author{
Liang Chen\thanks{School of Mathematics, Hunan University, Changsha, 410082, China
(\url{chl@hnu.edu.cn}).} 
\quad  
Youyicun Lin\thanks{School of Mathematics, Hunan University, Changsha, 410082, China
    (\url{linyouyicun@hnu.edu.cn}).}
\quad 
Yuxuan Zhou\thanks{Department of Mathematics, Southern University of Science and Technology, Shenzhen, 518055, China            (\url{zhouyx8@mail.sustech.edu.cn}).}}

\date{May 8, 2025}

\begin{document}
\maketitle

\begin{abstract}
This paper proposes an efficient numerical optimization method based on second-order cone programming (SOCP) to solve dynamic optimal transport (DOT) problems with quadratic costs on staggered grid discretizations.
By properly reformulating the discretized DOT problem into an equivalent linear SOCP, we develop a highly efficient implementation of an inexact decomposition-based proximal augmented Lagrangian method to solve it. 
The proposed approach is provided as an open-source software package to facilitate reproducibility and further research. 
Numerical experiments on a diverse range of DOT problems demonstrate that our software significantly outperforms several state-of-the-art solvers in terms of both accuracy and computational efficiency. Furthermore, it exhibits robust performance when handling measures that are not strictly positive or are in irregular domains with obstacles.

\bigskip
\noindent
{\bf Keywords:}
Dynamic optimal transport,
Second-order cone programming,
Staggered grid discretization,
Proximal augmented Lagrangian method,
MATLAB software package

\medskip
\noindent 
{\bf MSCcodes:}
90C25, 90C06, 90-04, 65K05, 49Q22
\end{abstract}

	\section{Introduction} 
Initiated by Monge \cite{monge} in 1781, optimal transport (OT) has evolved into a well-established mathematical area, and significantly advanced through Kantorovich's relaxation \cite{Kantorovich42}. 
The (Monge-Kantorovich) OT problem seeks the minimum cost to transport the mass between two probability measures \cite{villani}.
In recent years, it has found extensive applications across various fields, including signal processing \cite{flamary16,tart16}, economics \cite{galichon16,wrd25}, and machine learning \cite{peyre}.

In the seminal work of Benamou and Brenier \cite{bb00}, OT was studied within a continuum mechanics framework.
According to \cite[Proposition 1.1]{bb00}, given two probability measures $\rho_0$ and $\rho_1$ over a spatial domain $X\subseteq{\mathds R}^d$, the $L^2$ Kantorovich-Wasserstein distance, i.e., the optimal value of the OT problem under a quadratic distance cost, is the same as the optimal value of the following \emph{dynamic optimal transport} (DOT) problem 
\begin{equation}
	\label{eq:opt-ene}
	\min_{\rho,\bm{v}}
	\left\{
	\int_0^1\int_{X} \frac{1}{2}\|\bm{v}(t,\cdot)\|^2 \dd\rho(t,\cdot) \dd t
	\ \Big\vert\ 
	\begin{array}{ll}
		\partial_t\rho+\mathrm{div}_{\x}(\rho\bm{v})=0,  \\
		\rho(0,\cdot)= \rho_0,  \rho(1,\cdot)= \rho_1 
	\end{array}
	\right\},
\end{equation}
where $\rho(t,\cdot)$ denotes a time-varying probability measure on $X$, and $\bm{v}(t,\cdot)$ is a Borel velocity field on $X$ such that $\int_0^1\int_{X} \frac{1}{2}\|\bm{v}(t,\cdot)\| \dd\rho(t,\cdot) \dd t<+\infty$.
The operator $\mathrm{div}_{\x}$ denotes the divergence with respect to the spatial variable $\x$. 
Here, the continuity equation $\partial_t\rho+\mathrm{div}_{\x}(\rho\bm{v})=0$ in \eqref{eq:opt-ene} holds in the weak sense.
The DOT problem \eqref{eq:opt-ene} aims to minimize the time-space integral of the kinetic energy, subject to the constraint of the continuity equation (including no-flux boundary conditions $\rho(t,\cdot)\bm{v}(t,\cdot)\cdot \bm{n}_{x} = 0$ on $\partial X$) and prescribed temporal boundary conditions. 
From a geometric perspective, solving \eqref{eq:opt-ene} is equivalent to computing a geodesic in the Wasserstein space with the optimal solution $\rho(t,\cdot)$ being the shortest curve from $\rho_0$ to $\rho_1$.
Moreover, DOT also facilitates the extension of OT to related research areas such as mean-field games \cite{bb15,bonnans23}.

As highlighted in \cite{jose22}, a key advantage of \eqref{eq:opt-ene} over the classical Monge-Kantorovich OT problem is that it leads to a discretized optimization problem with $N_t\times N^D_{\x}$ variables, where $N_{\x}$ and $N_t$ represent the number of spatial and temporal grid points, respectively, and $D$ is the spatial dimension.
In contrast, the classical (entropy-regularized) linear programming approach results in an optimization problem with $N_{\x}^{2D}$ variables. 
However, the DOT problem \eqref{eq:opt-ene} is nonlinear and nonconvex, posing significant challenges for numerical computation. 
Fortunately, this issue was alleviated in \cite{bb00} by introducing the vector measure $\bm{m}$ to take the place of $\rho\bm{v}$, transforming the continuity equation into the linear constraint $\partial_t\rho+\mathrm{div}_{\x}(\bm{m})=0$. 
Consequently, the DOT problem \eqref{eq:opt-ene} can be reformulated as the convex optimization problem
\begin{equation}
	\label{eq:opt-ene-bb}
	\min_{\rho,\bm{m}}\left\{\mathfrak{B}(\rho,\bm{m})  \mid \partial_t\rho+\mathrm{div}_{\x}(\bm{m})=0,\
	\rho(0,\cdot)= \rho_0, \ \rho(1,\cdot)= \rho_1
	\right\},
\end{equation}
where $\mathfrak{B}(\rho,\bm{m})$ is the Benamou-Brenier functional, which is a convex function (see \Cref{convexdot} for the detailed definition and derivation).  

To solve \eqref{eq:opt-ene-bb}, several numerical methods have been developed. 
The original work of Benamou and Brenier \cite{bb00} 
focused on the dual problem of \eqref{eq:opt-ene-bb} and applied the alternating direction method of multipliers (ADMM), which was commonly known as ALG2 \cite{glowinski75,gabay76}. 
Their approach used a centered finite-difference grid in the spatial domain.
Later, by treating the algorithm in \cite{bb00} as the Douglas-Rachford splitting method, Papadakis et al. \cite{papa14} proposed solving \eqref{eq:opt-ene-bb} via other proximal splitting schemes and further took advantage of the staggered grid discretization. 
Compared to centered grids, staggered grids help avoid the odd-even decoupling problem that occurs with centered grids, enhancing the stability of discretized models \cite[Section 8]{cfd}.
More recently, Yu et al. \cite{yjj24} applied the accelerated proximal gradient method \cite{fista}, also known as the fast iterative soft threshold algorithm (FISTA), combined with staggered grids in the time-space domain and the multilevel initialization technique of \cite{ljl21}. 
The corresponding numerical results suggest that the multilevel-based FISTA is several times faster than the other approaches with which it was compared. 
Although ADMM and proximal splitting schemes are efficient for discrete DOT problems of moderate size, their numerical performance struggles with large discrete grids, e.g., a $64\times 256\times 256$ staggered grid in a $2$-dimensional spatial space. 
The multilevel-based FISTA provides more scalability for large-scale discrete problems but exhibits instability when one of the probability densities $\rho_0$ and $\rho_1$ is not strictly positive.
Consequently, it remains a challenging problem to solve large-scale discrete DOT problems efficiently.

Generally, the computational efficiency of first-order algorithms for discretized DOT problems on large-scale grids is primarily hindered by two challenges: the significant time consumption of each iteration and the excessive number of iterations required to achieve a desired level of accuracy.
Discretizing the DOT problem on staggered grids yields a convex optimization problem involving millions, or even tens of millions, of variables. Consequently, each iteration becomes computationally prohibitive unless specialized techniques are employed to mitigate the cost. 
It was observed in \cite[Section 4]{papa14} that proximal splitting algorithms typically necessitate solving three distinct subproblems: a Poisson equation on the centered grid, a series of cubic equations to find real roots, and a linear system involving the interpolation matrix.
While the Poisson equation can be solved efficiently via the fast Fourier transform (FFT), the latter two subproblems remain computationally demanding despite their structural simplicity. 
Specifically, existing literature often employs Newton's method to resolve the cubic equations for efficiency \cite{bb00,papa14,bb15}; however, as noted in \cite{papa14}, this can lead to numerical oscillations. 
To bypass these two types of subproblems to improve per-iteration speed, FISTA was used in the software of \cite{yjj24}, which performs better than many other algorithms when densities have a positive lower bound. 
Nevertheless, it generally suffers from requiring more iterations to reach the target precision, especially when the positive lower bound is not available. 
As demonstrated by our numerical experiments in \Cref{subsec:num}, existing software packages encounter significant difficulties in maintaining both high efficiency and precision.

The primary objective of this paper is to overcome these obstacles in applying first-order algorithms to DOT problems on staggered grid discretization. 
To achieve this, we propose a novel numerical optimization approach based on \emph{second-order cone programming} (SOCP) \cite{monteiro00,soc,chensocpreview2025} to efficiently solve these problems. 
Specifically, by analyzing the structure of the dual problem of the discretized DOT problem, we separate the coupled variables and reformulate the problem as a linear SOCP.
This reformulation eliminates the interpolation matrices, thus avoiding the linear systems induced by interpolation.
For the linear SOCP reformulation, we establish the existence of solutions to its Karush-Kuhn-Tucker (KKT) system without requiring additional assumptions, which facilitates the convergence analysis of primal-dual-type algorithms. 
Moreover, when ADMM-type algorithms are implemented for this reformulated problem, the task of solving cubic equations is converted into computing projections on second-order cones,  which significantly improves efficiency (see \Cref{tab:compareProj} for a comparison of solving cubic equations with computing their projection counterparts).
To further take advantage of the SOCP reformulation, we propose to solve it using an inexact decomposition-based proximal \emph{augmented Lagrangian method} (ALM), which is the culmination of a series of works \cite{hestenes,powell,rock-alm,lxd16,cl21} to solve constrained convex optimization problems. 
Principally, it resembles the classic ADMM.
However, a larger dual step length in the interval $(0,2)$ is allowed, akin to the ALM, as guaranteed by the convergence analysis, to achieve over-relaxation of the multiplier updating to extract more efficiency.
In the numerical implementation, we dynamically adjust the penalty parameters so that the proposed algorithm does not rely on parameter tuning and remains efficient even when a high-quality initial solution is unavailable.
Furthermore, inspired by recent work \cite{ljl21,yjj24}, we employ multilevel methods by successively refining the grids. 
Finally, we conducted extensive numerical experiments on the proposed approach to verify its efficiency and robustness, and made it an open-source software package in both MATLAB (with C++ MEX subroutines) and Python. 
Our numerical results demonstrate that the proposed approach is generally at least $10$ times faster than other software packages to achieve a relative error of $10^{-4}$ on the residual of the KKT system. 
Moreover, it performs effectively and efficiently for DOT problems without strictly positive probability measures, even in irregular domains, or Dirac measures.  

The remaining parts of this paper are organized as follows. 
In \Cref{sec:pre}, we present the notation and preliminaries, including a brief overview of the DOT problem and its discretization on staggered grids. 
In \Cref{sec:socp}, we reformulate the discretized DOT problem as a linear SOCP in which the coupled variables are properly separated into a more tractable formulation. 
Moreover, we show that such a reformulation is equivalent to the original discretized DOT problem without requiring any additional assumptions. 
In \Cref{sec:alg}, we introduce an inexact proximal ALM for solving the reformulated SOCP, where each step can be solved with extremely low computational complexity. 
In \Cref{sec:numerical}, we discuss computational details and conduct numerical experiments to illustrate the efficiency and robustness of the proposed approach.
We conclude this paper in \Cref{sec:conclu}.

\section{Notation and preliminaries}
\label{sec:pre}
We use $\mathds{R}^n$ to represent the $n$-dimensional real Euclidean space, and use $\mathds{R}_+^n$ ($\mathds{R}_-^n$) to represent the non-negative (non-positive) orthant in $\mathds{R}^n$.
We denote by $\|\x\|$ the Euclidean norm of $\x$.
For a vector $\x\in \mathds{R}^n$, 
$|\x|^2$ denotes the vector obtained by squaring each component of $\x$; 
that is, $\left(|\x|^2\right)_i = (x_i)^2$. 
For two vectors $\x$ and $\y$, we often use $(\x;\y)$ to represent $(\x^\top,\y^\top)^\top$ for convenience. 
We use $\bm{1}_n$ and $\bm{0}_n$ to denote the vectors in $\mathds{R}^n$ whose entries are all ones and all zeros, respectively.

For a vector $\y\in \mathds{R}^{1+n}$, 
we can use $ y_0$ to denote the first component of ${\y}$ and define
$\bar {\y}:=(y_1,\ldots,y_{n})^\top$, which is the subvector of $\y$ without its first component.
Then, for simplicity, we can denote $\y=( y_0;\bar\y)$. 
Let $\mathds{K}_{\mathrm{soc}}:=\{\y\in\mathds{R}^{1+n}\mid  y_0\geq \|\bar{\bm\y}\|\}$
be the second-order cone in $\mathds{R}^{1+n}$, 
which is a closed convex cone. 
According to \cite[Proposition 3.3]{fukushima}, 
the projection $\operatorname{\Pi}_{\mathds{K}_{\mathrm{soc}}}(\y)$ of $\y=( y_0;\bar\y)\in \mathds{R}^{1+n}$  onto the second-order cone $\mathds{K}_{\mathrm{soc}}$ is given by
\begin{equation}
	\label{app-proj}
	\operatorname{\Pi}_{\mathds{K}_{\mathrm{soc}}}(\y) = \begin{cases}
		\y & \mbox{if } \|\bar\y\|\leq  y_0,\\
		0 &  \mbox{if } \|\bar\y\|\leq - y_0,\\
		\frac{1}{2}( y_0+\|\bar\y\|)\left(1;\frac{\bar\y}{\|\bar\y\|}\right) & \mbox{otherwise}.
	\end{cases}
\end{equation}

Let $\Omega:=[0,1]\times X \subset \mathds{R}\times\mathds{R}^D$ be a compact convex time-space subset. 
Denote $\mathbb{M}(\Omega)$ (respectively, $\mathbb{M}(\Omega;\mathds{R}^D)$) as the space of finite signed measures (respectively, finite signed vector measures) on $\Omega$. The subset $\mathbb{M}_+(\Omega)\subset \mathbb{M}(\Omega)$ represents the set of non-negative finite measures. $C(\Omega)$ is the space of continuous functions over $\Omega$, and $C(\Omega;\mathds{R}^D)$ represents the space of continuous vector fields over $\Omega$. Given that $\Omega$ is compact, $\mathbb{M}(\Omega)$ and $\mathbb{M}(\Omega;\mathds{R}^D)$ are the dual spaces of $C(\Omega)$ and $C(\Omega;\mathds{R}^D)$, respectively. 
For convenience, we abbreviate $\mathbb{M}(\Omega;\mathds{R}^D)$ and $C(\Omega;\mathds{R}^D)$ as $\mathbb{M}(\Omega)^D$ and $C(\Omega)^D$. 
The operator $\partial_t$ represents the partial derivative with respect to time,
$\nabla_{\x}$  denotes the gradient concerning the spatial variables,
and $\mathrm{div}_{\x}$ refers to the divergence concerning the spatial variables.
For convenience, we denote $\nabla_{t,\x} := (\partial_t,\nabla_{\x})$.

\subsection{Convex programming reformulation of the DOT problem}
\label{convexdot}
We briefly introduce the tractable convex reformulation of the DOT problem \eqref{eq:opt-ene} proposed in \cite{bb00} following the 
presentation in \cite[Section 5]{Santamot}.
The Benamou-Brenier functional $\mathfrak{B}\colon \mathbb{M}(\Omega)\times\mathbb{M}(\Omega)^D \rightarrow [0,+\infty]$ is defined by
$$
\mathfrak{B}(\rho,\bm{m}):=\sup_{a, \bm{b}}\Big\{\int_\Omega a(t,\x)\dd\rho +\int_\Omega \bm{b}(t,\x)\cdot\dd\bm{m}\mid(a,\bm{b})\in  \P\Big\},
$$
where the closed convex set $\P$ is defined by
\begin{equation}
	\label{defcp}
	\P:=\Big\{(a,\bm{b}) \in  C(\Omega)\times C(\Omega)^D\mid
	a(t, \x)+\frac{\|\bm{b}(t, \x)\|^2}{2} \leqslant 0\quad 
	\forall(t, \x) \in  \Omega
	\Big\}. 
\end{equation}
By definition, $\mathfrak{B}(\rho, \bm{m})$ is convex and lower semicontinuous (for the weak convergence). 
According to \cite[Proposition 5.18]{Santamot}, 
$\mathfrak{B}(\rho, \bm{m})<+\infty$ implies $\rho\in \mathbb{M}_+(\Omega)$ and $\bm{m}\ll \rho$, i.e., $\bm{m}$ is absolutely continuous with respect to $\rho$ and, in this case, there exists a density $\bm{v}$ with respect to $\rho$ such that $\bm{m}=\rho\bm{v}$ and $\mathfrak{B}(\rho, \bm{m})=\int_\Omega \frac{1}{2}\|\bm{v}\|^2\dd\rho$.
Consequently, the DOT problem \eqref{eq:opt-ene} can be equivalently recast as the convex optimization problem \eqref{eq:opt-ene-bb}.

Introducing the continuously differentiable test function $\phi\in  C^1(\Omega)$ as the dual variable, 
the dual problem of \eqref{eq:opt-ene-bb} is given (see \cite{hugo20} and \cite[Section 6]{Santamot} for details) by
\begin{equation}\label{eq:opt-dual}
	\max\limits_{\phi,\bm{p}}
	\big    \{
	G(\phi)-\delta_{\P}(\bm{p})
	\mid 
	\nabla_{t,\x}\phi-\bm{p}={\bm{0}}
	\big\}, 
\end{equation}
where 
$G(\phi):=\int_X\phi(1,\cdot)\dd\rho_1
-\int_X\phi(0,\cdot)\dd\rho_0$ is obtained from the boundary conditions of \eqref{eq:opt-ene-bb},
$\bm{p}:=(a,\bm{b})\in  C(\Omega)\times C(\Omega)^D$, and 
$\delta_{\P}(\bm{p})$ is the indicator function of the closed convex set $\P$ defined by \eqref{defcp}.

\subsection{Discretized DOT problem on staggered grid}
\label{subsec:dis}

To discretize the problem \eqref{eq:opt-dual}, we adopt the staggered finite difference framework, following the schemes utilized in \cite{papa14,yjj24}.
Without loss of generality, we set $X=[0,1]^D$ for notational simplicity. 
The domain $\Omega = [0,1]\times X$ can then be uniformly partitioned by dividing $\Omega$ into $n_d$ segments in each dimension, resulting in a step size of $h_d = 1/n_d$ for $d = 0, \ldots, D$. 
In the $d$-th dimension, the index sets $J_d^c$ for the centered grids, and $J_d^s$ for the staggered grids are given by
$$
\begin{cases}
	J_d^c \eqdef \{j\mid j \mbox{ is an integer, }0\leq j\leq n_d\}.
	\\
	J_d^s \eqdef \set{(j-1/2)}{j \mbox{ is an integer, }1\leq j\leq n_d}.
\end{cases}
$$
Then, the index sets of the centered grids for discretizing $X$ and $\Omega$ are the Cartesian products
$$
\G^c_X \eqdef J^c_1 \times \cdots \times J^c_D
\quad
\mbox{and}
\quad 
\G^c \eqdef J_0^c \times\G^c_X, 
$$
respectively. To achieve a more accurate and convenient discrete format of partial derivatives concerning the $d$-th component, consider the index set of the staggered grid given by
$$
\G^s_d \eqdef J_0^c \times \cdots \times J_{d-1}^c \times J_{d}^s \times J_{d+1}^c\cdots \times J_D^c. 
$$
For a vector whose components are assigned to the mesh points in the staggered grid, 
we use the index $\i = (i_0,\ldots,i_D)$ with $i_d\in J_d^c\cup J_d^s$  to specify the positions of these components.  Accordingly, the unknown function $\phi$ is represented on the centered grid by the vector
$$\bm{\varphi}:=(\varphi_{\i})_{\i \in  \G^c}\in \mathds{R}^{|\mathcal{G}^c|}:=\mathds{R}^{(n_0+1)\times\cdots\times(n_D+1)}.$$
Define $\mathds{R}^{|\G^s_d|}:=\mathds{R}^{{(n_0+1)}\times\cdots\times{(n_{d-1}+1)}\times{n_d}\times{(n_{d+1}+1)}\times\cdots\times{(n_D+1)}}$. Let $\e_d\in\mathds{R}^{D+1}$, $d=0,\ldots, D$ be the standard basis vectors such that the $(d+1)$-th component of $\e_d$ is $1$, and all the other components are $0$. 
Then, one can define the discrete gradient operator
\begin{equation}
	\label{Amap}
	\A \bm{\varphi} := \left(\mathcal{A}_0\bm{\varphi};\mathcal{A}_1\bm{\varphi};\ldots; \mathcal{A}_D\bm{\varphi}\right)\in  \mathds{R}^{|\mathcal{G}_0^s|}\times\mathds{R}^{|\mathcal{G}_1^s|}\times\cdots\times \mathds{R}^{|\mathcal{G}_D^s|}, 
\end{equation}
where
$(\mathcal{A}_d\bm{\varphi})_{\i}:= 
\frac{1}{h_d}(\varphi_{\i+\frac{\e_d}{2}}
-\varphi_{\i-\frac{\e_d}{2}})\,$ $\forall \i\in \G^s_d$. 
By definition, one has 
\begin{equation}
	\label{eq:kerA}
	\mathrm{ker}(\A) = \{ \gamma\bm{1}_{|\G^c|}\in \mathds{R}^{|\G^c|}\mid\ \gamma\in \mathds{R}\}.
\end{equation}
To adequately represent the set $\mathcal{P}$ in \eqref{defcp}, we utilize the interpolation transformations to align variables onto a unified grid.
Specifically, we define the linear interpolations
$\mathcal{L}_{\T}\colon \mathds{R}^{|\mathcal{G}^c|}\rightarrow  \mathds{R}^{|\mathcal{G}^s_0|}$ 
and
$\mathcal{L}_{\X}^d \colon \mathds{R}^{|\mathcal{G}^c|}\rightarrow  \mathds{R}^{|\mathcal{G}^s_d|}$ for $d=1,\ldots,D$
defined by 
\begin{equation}
	\begin{cases}
		\left(\mathcal{L}_{\T}\y\right)_{\i} := \frac{y_{\i-\frac{\e_0}{2}}+\, y_{\i+\frac{\e_0}{2}}}{2}  & \forall\, \i\in\G^s_0, 
		\\[1mm]
		\left( \mathcal{L}^d_{\X}\y\right)_{\i} := \frac{y_{\i-\frac{\e_d}{2}}+\, y_{\i+\frac{\e_d}{2}}}{2} & \forall\, \i\in \G^s_d,
		\label{eq:ltlx}
	\end{cases}
\end{equation}
which compute the averages between adjacent elements corresponding to the temporal and spatial indices, respectively. 
Moreover, the adjoints of these operators are denoted by $ \mathcal{L}_{\T}^*$ and $ (\mathcal{L}_{\X}^d)^*$, respectively. 
Then, for any given vector $\v=(\v_1;\ldots;\v_D)\in  \mathds{R}^{|\mathcal{G}_1^s|}\times\cdots\times \mathds{R}^{| \mathcal{G}_D^s|}$, we define 
$\mathcal{L}_{\X}^*\v:=(\mathcal{L}_{\X}^1)^*\v_1+\ldots+(\mathcal{L}_{\X}^D)^*\v_D$ as the adjoint linear operator of  $\mathcal{L}_{\X}$. 
For sampling the two measures $\rho_0$ and $\rho_1$ in \eqref{eq:opt-ene-bb},
we define the hypercube centered at each $\bm{ \bar i}=(i_1,\ldots,i_D)\in \G^c_X$ by 
$$
X_{\bm{\bar i}}:=\big\{(y_1,\ldots,y_D)\in X\mid |i_dh_d-y_d|\leq \frac{h_d}{2},  d=1,\ldots,D
\big\}.
$$ 
Then we define the vector $\bm{c}\in \mathds{R}^{|\G^c|}$ by 
\begin{equation}
	\label{eq:defc}
	\c_{\i}:=\begin{cases}
		\rho_0 (X_{\bm{\bar{i}}}) & \text{ if } i_0 = 0,\\
		-\rho_1(X_{\bm{\bar{i}}})  & \text{ if } i_0=n_0, \\
		0                  & \mbox{ otherwise}
	\end{cases}
	\qquad \forall\, \i=(i_0,\bm{\bar \i})\in\G^c.
\end{equation}
Based on the above definitions, the discretized version of problem \eqref{eq:opt-dual} under the staggered grid is given as follows:
\begin{equation}
	\label{eq:opt-disdualori}
	\min\limits_{\bm{\varphi} , \q}
	\big\{
	\langle \c,\bm{\varphi} \rangle 
	\mid 
	\A \bm{\varphi} - \q=\bm{0},\  \bm{q}\in  \mathds{P}
	\big\}, 
\end{equation}
where the variable 
$\q:=(\q_0;\ldots ;\q_D)\in  \mathds{R}^{|\mathcal{G}_0^s|}\times\cdots\times \mathds{R}^{| \mathcal{G}_D^s|}$ 
is the discretization of $\bm{p}$, 
and the set $\mathds{P}$ represents the discretized constraints from $\P$ given by
\begin{equation}
	\label{defp}
	\begin{array}{ll}
		\mathds{P}:= 
		\Big\{
		\q = (\q_0; \ldots; \q_D) \mid  
		\q_d \in  \mathds{R}^{|\G^s_d|},
		\\[1mm]
		\qquad 
		~\qquad 
		(\q_0)_{\i}+\frac{1}{2}\left( \mathcal{L}_{\T}\mathcal{L}_{\X}^*(|\q_1|^2;\ldots ;|\q_D|^2)\right)_{\i}\leq 0\
		\forall\, \i\in \G^s_0
		\Big\}.
	\end{array}    
\end{equation}

We make the following remark regarding the interpolation operators $\mathcal{L}_{\T}$ and $\mathcal{L}_{\X}^*$ in the discretization. 
\begin{remark}
	In the context of finite element and finite volume methods, interpolation operators (referred to as ``weighted means'') arise naturally during discretization \cite{hugo18,nata21}, typically introducing computational redundancies. 
	In finite difference schemes, while centered discretization can avoid such interpolations, it may introduce oscillations into the optimization problem due to the constraints being solely first-order and a lack of mutual constraints between adjacent grid points.
	In contrast, staggered grids with interpolation operators help mitigate the odd-even decoupling commonly found in regular centered grids, thereby enhancing the stability of discretized models \cite[Section 8]{cfd}. 
\end{remark}

\section{The SOCP reformulation}
\label{sec:socp}
In this section, we propose an equivalent SOCP reformulation of the discretized dual DOT problem \eqref{eq:opt-disdualori}. 
We begin by conducting a theoretical study to establish intrinsic properties that can be exploited in algorithm design. 

\subsection{Properties of the discretized DOT problem}
By introducing the dual variable 
${\bm{\Lambda}=(\bm{\Lambda}_0; \bar{\bm\Lambda};\bm{\Lambda}_q)}
\in  
\mathds{R}^{|\mathcal{G}_0^s|}\times \mathds{R}^{|\mathcal{G}_1^s|+\cdots + | \mathcal{G}_D^s|}\times \mathds{R}^{|\mathcal{G}_0^s|}$, 
the Lagrangian function of \eqref{eq:opt-disdualori} is defined by
\begin{equation*}
	\begin{array}{ll}
		\displaystyle
		\mathfrak{L}(\bm{\varphi},\bm{q};\bm{\Lambda}):= &\langle \c,\bm{\varphi}\rangle
		+\langle \mathcal{A}\bm{\varphi} -\bm{q},(\bm{\Lambda}_0;\bar{\bm\Lambda})\rangle 
		\\[1mm]
		\displaystyle
		&\qquad-\langle \bm{q}_0+\frac{1}{2}\mathcal{L}_{\T}\mathcal{L}_{\X}^*(|\bm{q}_1|^2;\ldots  ;|\bm{q}_D|^2),\bm{\Lambda}_q\rangle,    
	\end{array}
\end{equation*}
where we write $\q \equiv (\q_0;\bar\q)$ for convenience. 
Therefore, by direct calculation, one has 
$$
\begin{array}{ll}
	\inf\limits_{\bm{\varphi},\bm{q}}
	\mathfrak{L}(\bm{\varphi},\bm{q};\bm{\Lambda})
	\\[2mm]
	= 
	\inf\limits_{\bm{\varphi},\bm{q}}\,
	\langle \c+ \mathcal{A}^* (\bm{\Lambda}_0;\bar{\bm\Lambda}) ,\bm{\varphi}\rangle
	-\langle \bm{q}_0, \bm{\Lambda}_0 +\bm{\Lambda}_q\rangle
	-\langle \bar{\bm q},  \bar{\bm\Lambda}\rangle 
	-\frac{1}{2} \sum\limits_{d=1}^D\langle |\q_d|^2, \mathcal{L}_{\X}^d\mathcal{L}_{\T}^*\bm{\Lambda}_q\rangle
	\\
	=
	\begin{cases}
		\inf\limits_{\bar{\bm q}}
		\frac{1}{2} \sum\limits_{d=1}^D\langle |\q_d|^2, \mathcal{L}_{\X}^d\mathcal{L}_{\T}^*\bm{\Lambda}_0\rangle 
		-\langle \bar{\bm q},  \bar{\bm\Lambda}\rangle 
		&
		\mbox{if } 
		\left\{\begin{array}{ll}
			\c+ \mathcal{A}^* (\bm{\Lambda}_0;\bar{\bm\Lambda})=0,\\ 
			\bm{\Lambda}_0 +\bm{\Lambda}_q=0,
			\mathcal{L}_{\X}\mathcal{L}_{\T}^*\bm{\Lambda}_q\le 0,    
		\end{array}\right.
		\\[2mm]
		-\infty &\mbox{otherwise,}
	\end{cases}
	\\[7mm]
	=
	\begin{cases}
		-\frac{1}{2}
		\sum\limits_{d=1}^D
		\sum\limits_{\i\in\overline \G^s_d}
		\frac{(\bar{\bm\Lambda}_d)_{\i}^2}{(\mathcal{L}^d_{\X}\mathcal{L}_{\T}^*\bm{\Lambda}_0)_{\i}}
		&
		\mbox{if } 
		\left\{\begin{array}{ll}
			\c+ \mathcal{A}^* (\bm{\Lambda}_0;\bar{\bm\Lambda})=0,\\ 
			\bm{\Lambda}_0 +\bm{\Lambda}_q=0,
			\mathcal{L}_{\X}\mathcal{L}_{\T}^*\bm{\Lambda}_0\ge 0,
			\\
			\mbox{and }
			(\bar{\bm\Lambda}_d)_{\i}=0 \mbox{ if } 
			(\mathcal{L}^d_{\X}\mathcal{L}_{\T}^*\bm{\Lambda}_0)_{\i}=0, 
		\end{array}\right.
		\\[2mm]
		-\infty &\mbox{ otherwise},
	\end{cases}
\end{array}
$$
where 
$\overline \G^s_d:=\{ \i\in \G^s_d\mid (\mathcal{L}^d_{\X}\mathcal{L}_{\T}^*\bm{\Lambda}_0)_{\i}>0\}$. 
Then the dual problem to \eqref{eq:opt-disdualori} is given by
\begin{equation}
	\label{eq:kkt-dualy}
	\max\limits_{\bm{\Lambda}_0,\bar{\bm\Lambda}}
	\Big\{-\frac{1}{2}
	\sum\limits_{d=1}^D
	\sum\limits_{\i\in\overline \G^s_d}
	\frac{(\bar{\bm\Lambda}_d)_{\i}^2}{(\mathcal{L}^d_{\X}\mathcal{L}_{\T}^*\bm{\Lambda}_0)_{\i}}
	\ \Big\vert \  
	\c+ \mathcal{A}^* (\bm{\Lambda}_0;\bar{\bm\Lambda})=0,  
	\ 
	(\bm{\Lambda}_0,\bar{\bm\Lambda})\in \mathcal{C}
	\Big\},
\end{equation}
where $\mathcal{C}:=\Big\{(\bm{\Lambda}_0,\bar{\bm\Lambda})\mid 
\bm{\Lambda}_0\ge 0, 
\mbox{and } 
(\bar{\bm\Lambda}_d)_{\i}=0 \mbox{ if } 
(\mathcal{L}^d_{\X}\mathcal{L}_{\T}^*\bm{\Lambda}_0)_{\i}=0, \ \i\in\G^s_d
\Big\}$.
Meanwhile, the KKT system of \eqref{eq:opt-disdualori} is given by
\begin{equation}
	\begin{cases}
		\label{eq:kkt-dual}
		\mathcal{A}\bm{\varphi}=\bm{q}, \ 
		\mathcal{A}^*(\bm{\Lambda}_0;\bar{\bm\Lambda})+\c=0,\ 
		\bm{\Lambda}_0 +\bm{\Lambda}_q=0,  
		\\
		0 \leq \bm{\Lambda}_0 \perp \bm{q}_0+\frac{1}{2}\mathcal{L}_{\T}\mathcal{L}_{\X}^*(|\bm{q}_1|^2; \ldots ;|\bm{q}_D|^2) \leq 0,\\
		\bar{\bm\Lambda}_d = \left(\mathcal{L}_{\X}^d \mathcal{L}^*_{\T}\bm{\Lambda}_0\right) \odot \bm{q}_d, \ d=1,\ldots, D,
	\end{cases}
\end{equation} 
where $\odot$ denotes the Hadamard (element-wise) product.
Based on the duality framework above, we have the following result on the existence of optimal solutions and Lagrange multipliers to \eqref{eq:opt-disdualori}. 
In particular, this result does not require any additional assumptions.
\begin{proposition}
	\label{prop:range}
	The solution set to the KKT system \eqref{eq:kkt-dual} is nonempty, i.e., 
	the solution sets to both \eqref{eq:opt-disdualori} and 
	\eqref{eq:kkt-dualy} are nonempty.
\end{proposition}
\begin{proof}
	One can symbolically write the discrete divergence operator as $\A^* \equiv 
	(\A_0^*, \bar\A^* )$ with 
	$\bar\A^* \equiv (\A_1 ^*, \ldots, \A^*_D)$. 
	Recall that $h_d=1/n_d$ represents the step size in the $d$-th dimension. 
	Define the vector $\bm{\xi}_0:=(h_0h_1\cdots h_D) \bm{1}_{|\G^s_0|} \in  \mathds{R}^{|\G^s_0|}$. 
	By the definition of $\A^*_0$, one has
	\begin{equation*}
		(-\A_0^*\bm{\xi}_0)_{\i} = \begin{cases}
			h_1\cdots h_D & \mbox{if } i_0 = 0,\\
			-h_1\cdots h_D & \mbox{if } i_0 = n_0,\\
			0             & \mbox{ otherwise} 
		\end{cases}\qquad \forall\, \i\in\G^c.
	\end{equation*} 
	Consequently, one has from the definition of $\bm{c}$ in \eqref{eq:defc} that  
	\begin{equation}
		\label{czero}
		\sum_{i_1,\ldots,i_D} (-\bm{c}-\A_0^*\bm{\xi}_0)_{0,i_1,\ldots,i_D} = \sum_{i_1,\ldots,i_D} (-\bm{c}-\A_0^*\bm{\xi}_0)_{n_0,i_1,\ldots,i_D}=0.
	\end{equation}
	According to \eqref{Amap} one can define the mapping $\bar\A$ by
	$\bar \A \bm{\varphi} := \left(\mathcal{A}_1\bm{\varphi};\ldots; \mathcal{A}_D\bm{\varphi}\right)\in  \mathds{R}^{|\mathcal{G}_1^s|}\times\cdots\times \mathds{R}^{| \mathcal{G}_D^s|}$. 
	Note that the kernel of $\bar\A$ is spanned by the set of $(n_0+1)$ vectors
	\begin{equation*}
		\left\{
		\begin{array}{r}
			\begin{pmatrix}
				\bm{1}_{|\G^c_X|};
				\bm{0}_{|\G^c_X|};
				\bm{0}_{|\G^c_X|};
				\cdots
				\bm{0}_{|\G^c_X|}
			\end{pmatrix}, 
			\begin{pmatrix}
				\bm{0}_{|\G^c_X|};
				\bm{1}_{|\G^c_X|};
				\bm{0}_{|\G^c_X|};
				\cdots
				\bm{0}_{|\G^c_X|}
			\end{pmatrix},
			\qquad \qquad~
			\\[2mm]
			\ldots,\begin{pmatrix}
				\bm{0}_{|\G^c_X|};
				\cdots
				\bm{0}_{|\G^c_X|};
				\bm{0}_{|\G^c_X|};
				\bm{1}_{|\G^c_X|}
			\end{pmatrix}
		\end{array}
		\right\}.
	\end{equation*}
	Thus, one can see from \eqref{czero} that 
	$\langle -\bm{c}-\A_0^*\bm{\xi}_0,\x\rangle = 0 $ for any $\x\in \mathrm{ker}(\bar{\A})$. 
	Consequently, it follows that $-\bm{c}-\A_0^*\bm{\xi}_0 \in  \mathrm{rge}(\bar\A^*)$ (the range space of $\bar\A^*$), and there exists a vector $\bar{\bm\xi}$ such that
	$$\A^*\bm{\xi}+\c =0 \quad \mbox{with}\quad  \bm{\xi} := (\bm{\xi}_0;\bar{\bm\xi}).$$
	Moreover, since $\bm{\xi}_0 > 0$,
	one has from the definitions of $\L_\T$ and $\L_\X^d$ in \eqref{eq:ltlx} that 
	$\mathcal{L}^d_{\X}\mathcal{L}^*_{\T}(\bm{\xi}_0)>0$ for all $d=1,\ldots,D$.  
	This implies that $\bm{\xi}$ is a strictly feasible point of \eqref{eq:kkt-dualy}. 
	Note that the objective function in \eqref{eq:kkt-dualy} is always non-positive.
	Consequently, from \cite[Theorem 28.2]{rock-1}, we know that a Kuhn-Tucker vector exists for \eqref{eq:kkt-dualy}.

	On the other hand, for any given index $i_0$, define the vector $\bm{\varphi}^*\in  \mathds{R}^{|\G^c|}$ by
	\begin{equation*}
		(\varphi^*)_{i_0,i_1,\ldots,i_D}:=-h_0(i_0+1)\zeta \quad \forall\, (i_1,\ldots,i_D)\in \G^c_X,
	\end{equation*}
	where $\zeta>0$ is a constant. This definition implies that, as index $i_0$ increases, the value of $(\varphi^*)_{i_0,i_1,\ldots,i_D}$ decreases. 
	Therefore, one has 
	$$
	\q^*:=\A\bm{\varphi}^* = (\q_0^*;\bm{0};\ldots;\bm{0}),
	$$
	because $\bm{\varphi}^*$ does not vary with respect to $i_1,\ldots,i_D$. 
	Moreover, each component of the vector $\q_0^*$ is $-\zeta<0$, indicating that  $\q^*$ lies in the interior of the constraint set $\mathds{P}$ defined by \eqref{defp}. 
	Thus, $(\bm{\varphi}^*,\q^*)$ is a strictly feasible point to 
	\eqref{eq:opt-disdualori}. 
	Moreover, the optimal value of \eqref{eq:opt-disdualori} is finite since \eqref{eq:kkt-dualy} is feasible. 
	Consequently,  from \cite[Theorem~28.2]{rock-1}, we know that a Kuhn-Tucker vector of \eqref{eq:opt-disdualori} exists. 
	Then, by \cite[Theorems 30.4 and 30.5]{rock-1}, the solution set to the KKT system \eqref{eq:kkt-dual} is nonempty, and any solution to the KKT system constitutes a primal-dual solution pair to problems \eqref{eq:opt-disdualori} and \eqref{eq:kkt-dualy}.
	This completes the proof. 
\end{proof}

\subsection{Equivalent reformulation of the constraint set}
Next, we reformulate the constraint set $\mathds{P}$ in \eqref{eq:opt-disdualori} and \eqref{defp} into a computationally more tractable form. 
Define the linear operator $\F: \mathds{R}^{|\mathcal{G}_0^s|}\times\cdots\times \mathds{R}^{| \mathcal{G}_D^s|}\rightarrow \mathds{R}^{|\G^s_0|}\times\cdots\times\mathds{R}^{|\G^s_0|}
\equiv \mathds{R}^{(4D+1)|\G^s_0|}$
as follows:
\begin{equation}
	\label{eq.operatorF}
	\mathcal{F}: \q  \mapsto  (\mathcal{F}_0 \q_0; \mathcal{F}^1_1 \q_1 ;\ldots; \mathcal{F}^1_4 \q_1;\ldots;\mathcal{F}^d_1 \q_d;\ldots;\mathcal{F}^d_4 \q_d;\ldots;\mathcal{F}^D_1 \q_D;\ldots;\mathcal{F}^D_4 \q_D),
\end{equation}
where $\q=(\q_0;\ldots ;\q_D)$, $\mathcal{F}_0 \q_0=\q_0$, 
and for all $\i\in\G^s_0$,
\begin{equation}
	\label{fdetails}
	\begin{array}{ll}
		(\mathcal{F}^d_1 \q_d)_{\i} 
		= \begin{cases}
			(\q_d)_{\i-\frac{\e_0+\e_d}{2}}  &\hspace{-4pt} \mbox{if }i_d \geq 1, 
			\\
			0 &\hspace{-4pt}\mbox{if } i_d = 0, 
		\end{cases}
		\
		(\mathcal{F}^d_{2}  \q_d)_{\i} 
		= \begin{cases}
			(\q_d)_{\i-\frac{\e_0-\e_d}{2}}  &\hspace{-4pt} \mbox{if }i_d \leq n_d-1, \\
			0 &\hspace{-4pt}\mbox{if } i_d = n_d, 
		\end{cases}
		\\[4mm]
		(\mathcal{F}^d_{3} \q_d)_{\i} 
		= \begin{cases}
			(\q_d)_{\i+\frac{\e_0-\e_d}{2}}  &\hspace{-4pt}\mbox{if }i_d \geq 1, \\
			0 &\hspace{-4pt}\mbox{if } i_d = 0, 
		\end{cases}
		\
		(\mathcal{F}^d_4 \q_d)_{\i} 
		= \begin{cases}
			(\q_d)_{\i+\frac{\e_0+\e_d}{2}}  &\hspace{-4pt} \mbox{if }i_d \leq n_d-1, \\
			0 &\hspace{-4pt}\mbox{if } i_d = n_d. 
		\end{cases}
	\end{array}
\end{equation}
Note that the operator $\mathcal{F}$ decouples the linear operator $\mathcal{L}_{\T}\mathcal{L}_{\X}^*$ in the sense that  
\begin{equation}
	\label{eq:f1vf2v}
	\mathcal{L}_{\T}\mathcal{L}_{\X}^*(\q_1;\ldots;\q_D)
	= \frac{1}{4} \sum_{d=1}^D\sum_{v=1}^{4} \mathcal{F}^d_{v} \q_d \quad  \forall\, (\q_1;\ldots;\q_D) \in  \mathds{R}^{|\mathcal{G}_1^s|}\times\cdots\times \mathds{R}^{| \mathcal{G}_D^s|}.
\end{equation}
By using the linear operator $\F$ defined in \eqref{eq.operatorF}, one can see from \eqref{defp} and \eqref{eq:f1vf2v} that 
$\q\in \mathds{P} \iff \F\q \in \widetilde{\mathds{P}}$, where
\begin{equation*}  
	\begin{array}{r}
		\widetilde{\mathds{P}}:=
		\Big\{\widetilde{\q}=(\widetilde{\q}_0; \widetilde{\q}_{1,1}; \ldots;\widetilde{\q}_{1,4}; 
		\ldots;
		\widetilde{\q}_{D,1};\ldots; \widetilde{\q}_{D,4})\in  \mathds{R}^{(4D+1)|\G^s_0|} \mid
		\qquad~
		\\[1mm]
		\left(\widetilde{\q}_0\right)_{\i} + \frac{1}{8} \sum_{d=1}^D\sum_{v=1}^{4} (|\widetilde{\q}_{d,v}|^2)_{\i} \leq 0\
		\forall\, \i \in \G^s_0\Big\}.
	\end{array}
\end{equation*}
Thus, one can equivalently reformulate \eqref{eq:opt-disdualori} as the following problem
\begin{equation}\label{eq:disDualReformulation}
	\min\limits_{\bm{\varphi},\q}
	\big\{
	\langle \c, \bm{\varphi}\rangle \mid
	\mathcal{A}\bm{\varphi}=\q, \ \mathcal{F}\q\in  \widetilde{\mathds{P}}
	\big\}.
\end{equation}
It is crucial to note that the operator $\mathcal{F}^*\mathcal{F}$ is diagonal (see Proposition \ref{prop:fbbf} below). 
Moreover,  it is routine and easy to verify that a strictly feasible point of \eqref{eq:opt-disdualori} is a strictly feasible point of \eqref{eq:disDualReformulation}.

\subsection{Reformulation to SOCP}
As discussed in \cite{papa14}, one challenge in applying traditional primal-dual algorithms to solving \eqref{eq:opt-disdualori} is the requirement to solve a series of cubic equations to obtain their real roots. 
To avoid this, we adopt an SOCP reformulation of the discretized problem \eqref{eq:disDualReformulation} to transform the projection onto the set $\widetilde{\mathds{P}}$ into a parallel projection onto second-order cones.  
For the given second-order cone $\mathds{K}_{\mathrm{soc}}\subset \mathds{R}^{4D+2}$ and any $\x \equiv(x_0,x_1,\ldots,x_{4D})^{\top}\in \R^{4D+1}$, we have 
\begin{equation}
	\label{eq:socobservation}
	\begin{array}{lll}
		x_0 + \frac{1}{8} \sum\limits_{v=1}^{4D} x_{v}^2 \leq 0\  
		& \iff &  (1+x_0)^2+ \frac{1}{2} \sum\limits_{v=1}^{4D} x_{v}^2\leq (1-x_0)^2,
		\\
		&\iff&  \bar{\B} \x + \bar{\d} \in  \mathds{K}_{\mathrm{soc}},
	\end{array}
\end{equation}
where $\bar{\B} \bm{x} := (-x_0, \frac{\sqrt{2}}{2} x_1, \ldots, \frac{\sqrt{2}}{2} x_{4D}, x_0)^\top$ and $\bar{\d} := (1, 0, \ldots, 0, 1)^\top$.
Therefore, for $\widetilde{\q}\equiv (\widetilde{\q}_0;\widetilde{\q}_1;\ldots\\;\widetilde{\q}_{4D})\in \mathds{R}^{(4D+1)|\G^s_0|}$, we define the linear operator
$\B\colon  \mathds{R}^{(4D+1) |\G^s_0|}\rightarrow\mathds{R}^{(4D+2) |\G^s_0|}$
by
\begin{equation}
	\label{eq:bdd}
	\begin{array}{l}
		\mathcal{B}\widetilde{\q}:= 
		\big(
		-\widetilde{\q}_0;\frac{\sqrt{2}}{2}\widetilde{\q}_1;\ldots; \frac{\sqrt{2}}{2}\widetilde{\q}_{4D};  \widetilde{\q}_0
		\big),    
	\end{array}
\end{equation}
Then, one can see that for any $\q\in \mathds{R}^{|\mathcal{G}_0^s|}\times\cdots\times \mathds{R}^{| \mathcal{G}_D^s|}$, 
$$\F \q\in \widetilde{\mathds{P}}\ \iff \  \B\F \q+\d\in  \mathds{Q},$$ 
where $\d:= (\bm{1}_{|\G^s_0|}; \bm{0}_{|\G^s_0|};\ldots; \bm{0}_{|\G^s_0|}; \bm{1}_{|\G^s_0|})\in  \mathds{R}^{(4D+2) |\G^s_0|}$, and $\mathds{Q}\subset \mathds{R}^{(4D+2)|\G^s_0|}$ is the Cartesian product of second-order cones defined by
\begin{equation*}
	\mathds{Q} := 
	\big\{(\bm{\beta}_1;\ldots; \bm{\beta}_{4D+2}) \mid \bm{\beta}_d\in\mathds{R}^{|\G^s_0|}, 
	((\beta_1)_{\i};\ldots;(\beta_{4D+2})_{\i})\in \mathds{K}_{\mathrm{soc}}\ \forall\, \i\in \G^s_0 
	\big\}.
\end{equation*}
Consequently, we equivalently recast \eqref{eq:disDualReformulation} as
\begin{equation}
	\label{eq:opt-dualsoc}
	\min\limits_{\bm{\varphi},\q}
	\big\{
	\langle \c, \bm{\varphi}\rangle 
	\mid
	\mathcal{A}\bm{\varphi}=\q,\ 
	\B\mathcal{F}\q+\d\in  \mathds{Q}
	\}.
\end{equation}
According to \eqref{eq:socobservation}, one can see that a strictly feasible point of \eqref{eq:disDualReformulation} is also strictly feasible for \eqref{eq:opt-dualsoc}. 

At first glance, the reformulation \eqref{eq:opt-dualsoc} involves seemingly more complex constraints than the original discretized formulations \eqref{eq:opt-disdualori} or \eqref{eq:disDualReformulation}. 
While this might appear to be a disadvantageous strategy, the computational efficiency gained from the projection onto $\mathds{Q}$ far outweighs the perceived complexity.
Specifically, the projection $\operatorname{\Pi}_{\widetilde{\mathds{P}}}(\cdot)$ necessitates solving a series of cubic equations. In contrast, the projection $\operatorname{\Pi}_{\mathds{Q}}(\cdot)$ onto the second-order cones can be explicitly computed in closed form \eqref{app-proj} with a linear complexity. 
To empirically support this design, we conducted a preliminary numerical experiment in MATLAB comparing the efficiency of $\operatorname{\Pi}_{\mathds{Q}}(\cdot)$ and $\operatorname{\Pi}_{\widetilde{\mathds{P}}}(\cdot)$ across various grid sizes $(n_0, n_1, n_2)$ with $D=2$ within the same computational environment as \Cref{sec:numerical}. 
Notably, our implementation of $\operatorname{\Pi}_{\mathds{Q}}(\cdot)$ utilizes a C++ MEX interface to accelerate the computation\footnote{We observed that the overhead of merely constructing the cubic equations—even before solving them—already exceeded the total time required for the second-order cone projections. This observation underscores the efficiency of the SOCP reformulation and obviates the need for a C++ implementation of $\operatorname{\Pi}_{\widetilde{\mathds{P}}}(\cdot)$.}.
We count the average CPU time of $20$ randomly generated $\x\in \mathds{R}^{n_0\times (n_1+1)\times (n_2+1)}$.
The numerical results are summarized in \Cref{tab:compareProj}.
Given that these projections are the most computationally intensive components of primal-dual type algorithms, these preliminary results provide strong justification for the algorithmic framework developed in the following section. 
Furthermore, as an alternative to avoid $\operatorname{\Pi}_{\mathds{P}}(\cdot)$, the accelerated proximal gradient method was employed in \cite{yjj24}. 
While it bypasses both the cubic equations and the linear systems associated with staggered grid interpolation, its overall efficiency remains sensitive to the Lipschitz constant of the objective function, as further analyzed in our numerical evaluations in Section \ref{sec:numerical}.

\begin{table}
	\caption{Comparison of $\operatorname{\Pi}_{\widetilde{\mathds{P}}}(\x)$ with $\operatorname{\Pi}_{\mathds{Q}}(\x)$ in computational time (seconds)}
	\label{tab:compareProj}    % Give a unique label
	\begin{tabular}{llll}
		\hline\noalign{\smallskip}
		$(n_0,n_1,n_2)$ &  $\operatorname{\Pi}_{\widetilde{\mathds{P}}}(\x)$ & $\operatorname{\Pi}_{\mathds{Q}}(\x)$ & $\operatorname{\Pi}_{\mathds{Q}}(\x)$ in C++ \\
		\noalign{\smallskip}\hline\noalign{\smallskip}
		$(32, 128, 128)$  &  0.1733 (s) & 0.0264 (s) & 0.0098 (s) \\
		$(64, 256, 256)$  &  1.4948 (s) & 0.2481 (s) & 0.0786 (s) \\
		$(128, 512, 512)$ &  21.0523 (s) & 3.3177 (s) & 0.7664 (s) \\
		\hline
	\end{tabular}
\end{table}

The SOCP reformulation \eqref{eq:opt-dualsoc} belongs to the class of linear conic programming problems, which have been extensively studied in recent decades. Classical numerical methods, such as the interior point method, are well-suited for this problem \cite{sdpt3}. 
Recent studies, such as \cite{ll21}, can also be applied. 
Here, in order to tackle problems with extremely large data size, we propose to take a customized approach by examining the particular properties of \eqref{eq:opt-dualsoc}, which are beneficial for the algorithmic design in the next section.  
The most crucial observation is the following result.

\begin{proposition}\label{prop:fbbf}
	The composite linear operator $\F^*\B^*\B\F$ is diagonal, where $\F$ and $\B$ are defined in \eqref{eq.operatorF} and \eqref{eq:bdd}, respectively.
\end{proposition}
\begin{proof}
	Let $\q = (\q_0; \q_1;\ldots; \q_D)$ and $ \hat{\q} = (\hat{\q}_0;\hat{\q}_1;\ldots;\hat{\q}_D) $ be vectors in $\mathds{R}^{|\mathcal{G}_0^s|}\times\cdots\times \mathds{R}^{| \mathcal{G}_D^s|}$, and $\u = (\u_0; \u_1;\ldots; \u_{4D+1}) \in  \mathds{R}^{(4D+2) |\G^s_0|}$.
	From \eqref{eq.operatorF} and \eqref{eq:bdd}, the linear operator $\B \F$ and its adjoint $\F^*\B^*$ are given by
	\begin{equation}
		\label{eq.operBF}
		\begin{array}{ll}
			\B \F \q  = \Big(-\F_0 \q_0; \frac{\sqrt{2}}{2} \F_{1}^1 \q_1; \ldots, \frac{\sqrt{2}}{2} \F_{4}^1 \q_1; \ldots;\frac{\sqrt{2}}{2} \F_{1}^D \q_D; \ldots; \frac{\sqrt{2}}{2} \F_{4}^D \q_D; \F_0 \q_0\Big)
		\end{array}
	\end{equation}
	and
	\begin{equation}
		\label{eq.operFB}
		\begin{array}{ll}
			\F^* \B^* \u = \Big( \F_{0}^* (\u_{4D+1} - \u_{0}); \frac{\sqrt{2}}{2} \sum\limits_{v=1}^{4} (\F_{v}^1)^* \u_v ;\ldots; \frac{\sqrt{2}}{2} \sum\limits_{v=1}^{4} (\F_{v}^D)^* \u_{4D-4+v} \Big). 
		\end{array}
	\end{equation} 
	Combining \eqref{eq.operBF} and \eqref{eq.operFB}, we can obtain that 
	\begin{equation*}
		\begin{array}{ll}
			\langle\F^* \B^* \B \F \q, \hat{\q}\rangle 
			= \langle {2 \F_0^* \F_0 \q_0}, {\hat{\q}_0}\rangle 
			+ \sum\limits_{d=1}^D \big\langle{\frac{1}{2} \sum\limits_{v=1}^{4} (\F_{v}^d)^* \F_{v}^d \q_d},{\hat{\q}_d}\big\rangle.
		\end{array}
	\end{equation*}
	Since $\hat{\q}$ is arbitrarily chosen, it holds that 
	\begin{equation}
		\label{FFF}
		\begin{array}{ll}
			\F^* \B^* \B \F \q = \Big(2 \F^*_0 \F_0 \q_0; \frac{1}{2} \sum\limits_{v=1}^{4} (\F_{v}^1)^* \F_{v}^1 \q_1;\ldots; \frac{1}{2} \sum\limits_{v=1}^{4} (\F_{v}^D)^* \F_{v}^D \q_D\Big).
		\end{array}
	\end{equation}
	Note that $\F_0$ is the identity operator, so $\F^*_0 \F_0 \q_0 = \q_0$. Next, observe that each operator $\mathcal{F}_{v}^d$ extracts specific elements from $\q_d$, ensuring that each row of $\mathcal{F}_{v}^d$ contains at most one nonzero element that is equal to $1$. 
	Then, for $v=1$ and any $d=1,\ldots, D$, one has 
	\begin{equation*}
		\setlength{\extrarowheight}{3pt}
		\begin{array}{ll}
			\big\langle\left(\F_{v}^d\right)^* \F_{v}^d \q_d,\hat{\q}_d\big\rangle
			= \big\langle{\F_{1}^d \q_d}, 
			{\F_{1}^d \hat{\q}_d}\big\rangle
			\\[1mm]
			\displaystyle
			= \sum\limits_{\i \in  \G^s_0, i_d > 0} (q_d)_{\i-\frac{\e_0+\e_d}{2}} (\hat{q}_d)_{\i-\frac{\e_0+\e_d}{2}} 
			= \sum\limits_{\i \in  \G^s_d, i_0 < n_0} (q_d)_{\i} (\hat{q}_d)_{\i},
		\end{array}
	\end{equation*}
	which indicates that, for all $\i\in\G^s_d$, 
	\begin{equation*}
		\big( (\F_{1}^d )^* \F_{1}^d \q_d\big)_{\i} =
		\begin{cases}
			(q_d)_{\i} & \mbox{if}\; i_0 < n_0, \\
			0               & \mbox{if}\; i_0 = n_0.
		\end{cases} 
	\end{equation*}
	Consequently, the matrix $(\F_{1}^d)^* \F_{1}^d$ is diagonal. 
	Note that one can repeat the above procedure to compute $(\F_{v}^d)^* \F_{v}^d\q_d$ explicitly for the cases that $v = 2, 3, 4$ as follows:
	$$
	\left(\F_{2}^d\right)^* \F_{2}^d \q_d=\left(\F_{1}^d\right)^* \F_{1}^d \q_d,
	$$
	and
	$$
	\big((\F_{3}^d)^* \F_{3}^d \q_d\big)_{\i}
	=\big((\F_{4}^d)^* \F_{4}^d \q_d\big)_{\i}=
	\begin{cases}
		(q_d)_{\i} & \mbox{if}\; i_0 >0, \\
		0               & \mbox{if}\; i_0 = 0.
	\end{cases}
	$$
	Therefore, for any index $\i$ that represents a mesh point in $\G^s_d$, one has that 
	\begin{equation*}
		\big(\sum_{v=1}^{4}(\F_{v}^d)^* \F_{v}^d\q_d\big)_{\i} = 
		\begin{cases}
			2 (q_d)_{\i} & \mbox{if}\; i_0 = 0, \\
			4 (q_d)_{\i} & \mbox{if}\; 0 < i_0 < n_0, \\
			2 (q_d)_{\i} & \mbox{if}\; i_0 = n_0,
		\end{cases}
	\end{equation*}
	which, together with \eqref{FFF}, implies that $\F^*\B^*\B\F$ is diagonal. 
	This completes the proof.
\end{proof}

\begin{remark}
	A prior study \cite[Section 3.6]{papa14} discussed a (rotated) SOCP reformulation of problem \eqref{eq:kkt-dualy} but without decoupling the constraints. 
	Implementing first-order primal-dual methods can make such a reformulation lead to more difficult subproblems. 
	It is essentially different from our approach, in which the linear operator $\F$ plays an essential role in separating variables $\q$, enabling us to directly solve the second-order cone projection and efficiently apply existing first-order algorithms, as will be detailed in \Cref{sec:alg}.
\end{remark}

\subsection{Optimality condition and existence of solutions}
%\label{sec:optcond}
Finally, we discuss the optimality condition and the existence of solutions to the SOCP reformulation \eqref{eq:opt-dualsoc}. 
By introducing the auxiliary variables $\z:=\B\mathcal{F}\q+\d$, it can be equivalently reformulated as 
\begin{equation}
	\label{eq:opt-dualsocalm}
	\min\limits_{\bm{\varphi},\z,\q}\ 
	\big\{
	\langle \c, \bm{\varphi}\rangle + \delta_{\mathds{Q}}(\z)
	\mid
	\mathcal{M}(\bm\varphi;\bm z) +\mathcal{N}\q= (\bm 0;\d)
	\big\},
\end{equation}
where 
$
\mathcal{M} 
(\bm\varphi;\bm z)
:=
(\A \bm\varphi;\z)
$
and
$\mathcal{N} \q := (-\q;-\B\F\q)$. 
Define for \eqref{eq:opt-dualsocalm} the %augmented 
Lagrangian function
$$
\mathfrak{L}(\bm{\varphi},\z,\q;\bm{\alpha},\bm{\beta}):
=
\langle \c, \bm{\varphi}\rangle + \delta_{\mathds{Q}}(\z)+\langle\mathcal{A}\bm{\varphi}-\q,\bm{\alpha}\rangle+\langle \z-\B\mathcal{F}\q-\d,\bm{\beta} \rangle
$$
where $\bm{\alpha}$ and $\bm{\beta}$ are the dual variables (multipliers). 
Then, the KKT system to \eqref{eq:opt-dualsocalm} is given by
\begin{equation}
	\label{eq:kkt}
	\begin{cases}
		\mathcal{M}(\bm\varphi;\bm z)+\mathcal{N}\q=(\bm0 ;\d),
		\\
		(\A^*\bm\alpha;\bm\alpha) + (\bm 0;\F^*\B^* \bm\beta)  =  (-\bm c;\bm 0),
		\\
		\mathds{Q}\ni \z \perp \bm{\beta} \in  \mathds{Q}.
	\end{cases}
\end{equation} 
Next, we show that the solution set to the KKT system \eqref{eq:kkt} is nonempty, and one can recover from it a solution of the KKT system \eqref{eq:kkt-dual} of \eqref{eq:opt-disdualori}. 
We should reemphasize that we have not imposed any assumptions either in reformulating the discrete DOT problem \eqref{eq:opt-disdualori} as the SOCP problem \eqref{eq:opt-dualsocalm} and in establishing the proposition below. 
Consequently, the SOCP problem \eqref{eq:opt-dualsocalm} is an exact reformulation of the discretized DOT problem \eqref{eq:opt-disdualori}, in the sense that no further approximation is introduced at the discrete level. 
We should also mention that this exactness refers only to the finite-dimensional reformulation and does not address the convergence of the staggered-grid discretization to the continuous problem \eqref{eq:opt-ene-bb} or its dual \eqref{eq:opt-dual}.

\begin{proposition}
	\label{prop:sockkt}
	The solution set to the KKT system \eqref{eq:kkt} is nonempty.
	Moreover, for $(\bm{\varphi},\z,\q,\bm{\alpha}\\, \bm{\beta})$ being a solution to \eqref{eq:kkt}, the subvector $(\bm{\varphi},\q, \bm{\alpha}_0, \\\bar{\bm{\alpha}},-\bm{\alpha}_0)$ is a solution to the KKT system \eqref{eq:kkt-dual}, 
	where we decompose $\bm{\alpha}
	=(\bm{\alpha}_0;\bar{\bm{\alpha}})\in \mathds{R}^{|\mathcal{G}_0^s|}\times\mathds{R}^{|\mathcal{G}_1^s|+ \cdots+| \mathcal{G}_D^s|}$.
\end{proposition}

\begin{proof}
	Note that the problems \eqref{eq:opt-disdualori} and \eqref{eq:opt-dualsocalm} share the same optimal value. 
	Then, by  \Cref{prop:range}, the optimal value of \eqref{eq:opt-dualsocalm} is finite and attainable. 
	Moreover, it is easy to see that a strictly feasible point $(\bm{\varphi}, \bm{q})$ of \eqref{eq:opt-disdualori} ensures that $(\bm{\varphi}, \z,\q)$ is strictly feasible to \eqref{eq:opt-dualsocalm} with $\z=\B\mathcal{F}\q+\d$. 
	It follows from \cite[Theorem 28.2]{rock-1} that the KKT system \eqref{eq:kkt} admits a nonempty solution set.

	Let $(\bm{\varphi},\z,\q,\bm{\alpha},\bm{\beta})$ be a solution to \eqref{eq:kkt}. 
	Note that $\A\bm{\varphi}=\q$ and $\A^*\bm{\alpha}+\c=0$. 
	Moreover, since  $\bm{\beta}, \z\in \mathds{R}^{(4D+2)|\mathcal{G}^s_0|}$, one can write them as 
	$$
	\bm{\beta}=(\bm{\beta}_0;\ldots; \bm{\beta}_{4D+1})
	\quad\mbox{and}\quad
	\z=(\z_0;\ldots; \z_{4D+1}), 
	$$
	in which each component is a vector in $\mathds{R}^{|\mathcal{G}^s_0|}$. 
	Note that by the definitions of $\F$ in \eqref{eq.operatorF} and $\B$ in \eqref{eq:bdd} one can get 
	$$
	\mathcal{F}^*\B^* \bm{\beta}= \Big(\bm{\beta}_{4D+1}-\bm{\beta}_0;\frac{\sqrt{2}}{2}\sum_{v=1}^4\left(\mathcal{F}^1_{v}\right)^*\bm{\beta}_{v};\ldots;\frac{\sqrt{2}}{2}\sum_{v=1}^4\left(\mathcal{F}^D_{v}\right)^*\bm{\beta}_{4D-4+v}\Big).
	$$
	Note that $\bm{\beta} \in  \mathds{Q}$, i.e., 
	$$
	[\bm{\beta}]_{\i} := ( (\beta_0)_{\i}; \bar{\bm{\beta}}_{\i})\in\mathds{K}_{\mathrm{soc}}
	\quad\mbox{with}\quad 
	\bar{\bm{\beta}}_{\i}:=((\beta_1)_{\i};\ldots;(\beta_{4D+1})_{\i})\in
	\mathds{R}^{4D+1}.
	$$
	Similarly, one has from $\z \in  \mathds{Q}$ that 
	$$
	[\z]_{\i} := ( (z_0)_{\i}; \bar \z_{\i})\in\mathds{K}_{\mathrm{soc}}
	\quad\mbox{with}\quad 
	\bar{\z}_{\i}:=((z_1)_{\i};\ldots;(z_{4D+1})_{\i})\in
	\mathds{R}^{4D+1}.
	$$
	Note that  $\left(\bm{\beta}_{4D+1}-\bm{\beta}_0\right)_{\i} \leq 0$ for all $\i\in\mathcal{G}^s_0$. 
	By using $ \bm{\alpha}=-\mathcal{F}^*\B^* \bm{\beta} $ one can get $(\alpha_0)_{\i} \geq 0$ for all $\i\in\mathcal{G}^s_0$.

	Next, we show that the complementarity condition in \eqref{eq:kkt-dual} holds with $\bm{\Lambda}_0=\bm{\alpha}_0$.  
	Since $\z=\B\F\q+\d$, the first and last elements of $[\z]_{\i}$ are given by $1-(q_0)_{\i}$ and $1+(q_0)_{\i}$, 
	implying that $[\z]_{\i}$ is a nonzero vector. 
	If $(\alpha_0)_{\i}>0$ for some $\i\in\mathcal{G}^s_0$, one has 
	$(\bm{\beta}_{4D+1}-\bm{\beta}_0)_{\i}<0$, 
	indicating that $[\bm{\beta}]_{\i}$ is a nonzero vector. 
	Since $\mathds{Q}\ni \z \perp \bm{\beta} \in  \mathds{Q}$ 
	and $[\z]_{\i}$ is a nonzero vector, 
	$[\z]_{\i}$ must lie on the boundary of the second-order cone $\mathds{K}_{\mathrm{soc}}$.  
	Then by  \eqref{eq:socobservation} one has
	$$
	(q_0)_{\i}+\frac{1}{2}\left(\mathcal{L}_{\T}\mathcal{L}_{\X}^*(|\q_1|^2;\ldots;|\q_D|^2)\right)_{\i}=0. 
	$$
	
	Finally, we show that the last line in \eqref{eq:kkt-dual} holds with $\bm{\Lambda}_0=\bm{\alpha}_0$ and $\bar{\bm\Lambda}=\bar{\bm{\alpha}}$.  
	Recall that 
	\begin{equation*}
		\begin{array}{ll}
			\bar{\bm{\alpha}}
			=-\Big(\frac{\sqrt{2}}{2} \sum_{v=1}^4\left(\mathcal{F}^1_{v}\right)^*\bm{\beta}_{v};
			\ldots;
			\frac{\sqrt{2}}{2}
			\sum_{v=1}^4\left(\mathcal{F}^D_{v}\right)^*\bm{\beta}_{4D-4+v}
			\Big).
		\end{array}
	\end{equation*}
	Similar to the previous discussion for  $[\z]_{\i}$, if $(\alpha_0)_{\i} > 0$ for some $\i\in\mathcal{G}^s_0$, the vector $[\bm{\beta}]_{\i}$ is nonzero. Consequently, 
	by \cite[Lemma 15]{soc}, there exists a constant
	\begin{equation*}
		\eta_{\i}: = \frac{(z_0)_{\i}}{(\beta_0)_{\i}}>0\text{ such that }[\z]_{\i}=\eta_{\i}\big((\beta_0)_{\i}; -\bar{\bm{\beta}}_{\i}\big).
	\end{equation*}
	Since $\z=\B\mathcal{F}\q+\d$, one has
	$$
	1-(q_0)_{\i}=(z_0)_{\i}=\eta_{\i}(\beta_0)_{\i}
	\quad\mbox{and}\quad 
	1+(q_0)_{\i}=(z_{4D+1})_{\i}=-\eta_{\i}(\beta_{4D+1})_{\i},
	$$
	which implies
	$(\alpha_0)_{\i}=(\beta_{0})_{\i}-(\beta_{4D+1})_{\i} = 2/\eta_{\i}$. 
	On the other hand, 
	if $(\alpha_0)_{\i}=0$ for some $\i\in\mathcal{G}^s_0$, then $(\bm{\beta}_{4D+1}-\bm{\beta}_0)_{\i}=0$.
	In this case, if $(\beta_0)_{\i}\neq  0$, there exist a $\eta_{\i}>0$ such that 
	$$
	1-(q_0)_{\i}=(z_0)_{\i}=\eta_{\i}(\beta_0)_{\i}\quad\mbox{and}\quad 
	1+(q_0)_{\i}=(z_{4D+1})_{\i}=-\eta_{\i}(\beta_{0})_{\i},
	$$
	which leads to a contradiction.
	Thus, it holds that 
	$(\beta_0)_{\i}=(\beta_{4D+1})_{\i}=0$,  
	implying that $[\bm{\beta}]_{\i}$ is a zero vector. 
	Define the vector $\widetilde{\bm{\eta}}\in \mathds{R}^{|\mathcal{G}^s_0|}$ by
	\begin{equation*}
		\widetilde{\eta}_{\i}:=\begin{cases}
			\frac{1}{\eta_{\i}} &\text{if } (\alpha_0)_{\i} > 0,\\
			0 &\text{if }{(\alpha_0)}_{\i} = 0
		\end{cases}
		\quad \forall\, \i\in\mathcal{G}^s_0,
	\end{equation*}
	so that $\bm{\alpha}_0 = 2 \widetilde{\bm{\eta}}$ and $-\bm{\beta}_n = \widetilde{\bm{\eta}}\odot \z_n, n=1,\ldots,4D+1$. 
	One can see from the proof for Proposition \ref{prop:fbbf} that  
	$$
	\begin{array}{ll}
		\B \F \q  = 
		\\
		\quad\Big(-\F_0 \q_0; \frac{\sqrt{2}}{2} \F_{1}^1 \q_1; \ldots, \frac{\sqrt{2}}{2} \F_{4}^1 \q_1;\ldots; \frac{\sqrt{2}}{2} \F_{1}^D \q_D; \ldots; \frac{\sqrt{2}}{2} \F_{4}^D \q_D; \F_0 \q_0
		\Big). 
	\end{array}
	$$
	Therefore, the $d$-th block-component of $\bar{\bm{\alpha}}$ can be given by
	\begin{equation*}
		\begin{array}{ll}
			-\frac{\sqrt{2}}{2}\Big(\sum\limits_{v=1}^4\left(\mathcal{F}^d_{v}\right)^*\bm{\beta}_{4d-4+v}\Big)  &=  \frac{\sqrt{2}}{2}\Big(\sum\limits_{v=1}^4\left(\mathcal{F}^d_{v}\right)^*\left(\widetilde{\bm{\eta}}\odot \z_{4d-4+v}\right)\Big)  \\[1mm]
			&= \frac{1}{2}\sum\limits_{v=1}^4 \left(\mathcal{F}^d_{v}\right)^* \left(\widetilde{\bm{\eta}}\odot \left(\mathcal{F}_{v}^d\q_d\right)\right).
		\end{array}
	\end{equation*}
	Moreover,  from \eqref{eq:f1vf2v} one has that $\mathcal{L}_{\X}^d\mathcal{L}_{\T}^* \bm{\alpha}_0 = \frac{1}{4}\sum^4_{v=1}\left(\F_{v}^d\right)^*\bm{\alpha}_0$.
	Then by  $\bm{\alpha}_0 = 2 \widetilde{\bm{\eta}}$ one has
	$$
	(\mathcal{L}_{\X}^d\mathcal{L}_{\T}^* \bm{\alpha}_0)
	\odot \q_d 
	= \frac{1}{2}\Big(\sum^4_{v=1}\left(\F_{v}^d\right)^*\widetilde{\bm{\eta}}\Big)\odot \q_d,
	\quad d=1,\ldots, D.
	$$
	Consequently, the proposition is proved if one can show that
	\begin{equation}
		\label{keyrelation}
		\left(\F_{v}^d\right)^*(\widetilde{\bm{\eta}}\odot (\mathcal{F}_{v}^d\q_d))
		=\left((\F_{v}^d)^*\widetilde{\bm{\eta}}\right) \odot \q_d \quad  
		\forall v=1,\ldots, 4. 
	\end{equation}
	For the case that $v=1$, from \eqref{fdetails} one has the operator $\left(\F_{1}^d\right)^*\colon \mathds{R}^{|\mathcal{G}^s_0|}\rightarrow\mathds{R}^{|\mathcal{G}^s_d|}$ is given by
	\begin{equation*}
		\left((\mathcal{F}^d_1)^*\bm{p}\right)_{\i-\frac{\e_0+\e_d}{2}} :=\begin{cases}
			p_{\i}  &\text{if } i_d\geq 1 \text{ and }i_0+\frac{1}{2} <n_0,\\
			0 &\mbox{otherwise} 
		\end{cases}
		\quad \forall\, \i\in\G^s_0.
	\end{equation*}
	Therefore, it holds that for all $\i\in\G^s_0,$
	$$\big(\left((\mathcal{F}^d_1)^*\widetilde{\bm{\eta}}\right) \odot \q_d\big)_{\i-\frac{\e_0+\e_d}{2}} =\begin{cases}
		\widetilde{\eta}_{\i} (q_d)_{\i-\frac{\e_0+\e_d}{2}}   &\text{if } i_d\geq 1 \text{ and }i_0+\frac{1}{2} <n_0,\\
		0 &\mbox{otherwise.}  
	\end{cases}
	$$
	Moreover, since if $i_d\geq 1$, we have $(\widetilde{\bm{\eta}}\odot (\mathcal{F}_{1}^d\q_d))_{\i} =  \widetilde{\eta}_{\i} (q_d)_{\i-\frac{\e_0+\e_d}{2}}.$
	Thus, for all $\i\in\G^s_0,$
	$$\big((\mathcal{F}^d_1)^*(\widetilde{\bm{\eta}}\odot (\mathcal{F}_{1}^d\q_d))\big)_{\i-\frac{\e_0+\e_d}{2}} = \begin{cases}
		\widetilde{\eta}_{\i} (q_d)_{\i-\frac{\e_0+\e_d}{2}}   &\text{if } i_d\geq 1 \text{ and }i_0+\frac{1}{2} <n_0,\\
		0 &\mbox{otherwise.}  
	\end{cases}$$  
	Consequently, \eqref{keyrelation} holds for $v=1$. 
	Moreover, it is easy to repeat the above procedure to show that \eqref{keyrelation} holds for the rest of the cases, and this completes the proof.     
\end{proof}

According to \Cref{prop:sockkt}, every solution  $(\bm{\varphi},\z,\q,\bm{\alpha},\bm{\beta})$ to the KKT system \eqref{eq:kkt} yields a solution $(\bm{\varphi},\q,\bm\Lambda_0,\bar{\bm\Lambda},\bm\Lambda_q)$ to the KKT system \eqref{eq:kkt-dual}, where $(\bm\Lambda_0,\bar{\bm\Lambda},\bm\Lambda_q):=(\bm\alpha_0,\bar{\bm\alpha},-\bm{\alpha}_0)$.
Moreover, $\bm z$ is the auxiliary variable introduced in \eqref{eq:opt-dualsocalm} when the conic constraint $\B\F\q+\d \in \mathds{Q}$ in \eqref{eq:opt-dualsoc} is rewritten as $\z = \B\F\q + \d$ and $\z \in\mathds{Q},$ and $\bm \beta$ is the multiplier associated with the equality constraint $\z - \B\F\q - \d = 0$.
The next section focuses on the algorithmic aspects of solving the KKT system \eqref{eq:kkt}. In contrast, the practical implementation in \Cref{sec:numerical} employs stopping criteria based on the KKT system \eqref{eq:kkt-dual}, which provides a more appropriate metric for assessing the solution quality of the original discretized problem \eqref{eq:opt-disdualori}.

\section{An inexact proximal ALM for solving SOCP reformulation}
\label{sec:alg}
This section proposes an inexact decomposition-based proximal ALM to solve the equivalent SOCP reformulation \eqref{eq:opt-dualsocalm} of the discrete DOT problem.
Recall that the augmented Lagrangian function of  \eqref{eq:opt-dualsocalm} is defined by
\begin{equation}
	\label{alf}
	\begin{array}{ll}
		\displaystyle
		\mathfrak{L}_{\sigma}(\bm{\varphi},\z,\q;\bm{\alpha},\bm{\beta})&:= \langle \c, \bm{\varphi}\rangle + \delta_{\mathds{Q}}(\z)+\langle\mathcal{A}\bm{\varphi}-\q,\bm{\alpha}\rangle+\langle \z-\B\mathcal{F}\q-\d,\bm{\beta} \rangle \\[1mm]
		& \displaystyle
		\quad+ \frac{\sigma}{2}\|\mathcal{A}\bm{\varphi}-\q\|^2 + \frac{\sigma}{2}\|\z-\B\mathcal{F}\q-\d\|^2. 
	\end{array}
\end{equation}
From \eqref{eq:kerA} we know that the discrete gradient operator $\A$ has a nonempty kernel space, so the linear operator $\A^*\A$ is singular. Consequently, minimizing the augmented Lagrangian function with respect to $\bm{\varphi}$ can lead to an unbounded set of solutions, which is generally not suitable for establishing the convergence of ALM-type or ADMM-type algorithms to a solution of the KKT system. 
To address this issue, we consider the following problem 
\begin{equation}
	\label{eq:opt-dualsocalm-ran}
	\min\limits_{\bm{\varphi},\z,\q}
	\big\{ 
	\langle \c, \bm{\varphi}\rangle + \delta_{\mathds{Q}}(\z)
	\mid
	\mathcal{M}(\bm\varphi;\z) +\mathcal{N}\q= (\bm 0; \d), 
	\
	\bm{\varphi}\in \mathrm{rge}(\A^*)
	\big\},
\end{equation}
i.e., restricting $\bm{\varphi}\in \mathrm{rge}(\A^*)$ in \eqref{eq:opt-dualsocalm}. 
Accordingly, the KKT system of \eqref{eq:opt-dualsocalm-ran} is given by
\begin{equation}\label{eq:kkt2}
	\begin{cases}
		\mathcal{M}(\bm\varphi;\bm z)+\mathcal{N}\q=(\bm0 ;\d),\  \bm{\varphi}\in\mathrm{rge}(\A^*),
		\\ 
		(\A^*\bm\alpha;\bm\alpha) + (\bm 0;\F^*\B^* \bm\beta)  =  (-\bm c;\bm 0),
		\\
		\mathds{Q}\ni \z \perp \bm{\beta} \in  \mathds{Q}.
	\end{cases}
\end{equation} 

\begin{remark}
	\label{remarkkkt}
	One can see that any solution of the KKT system \eqref{eq:kkt2} is also a solution of \eqref{eq:kkt}. 
	Moreover, for a given solution $(\bm{\varphi},\z,\q,\bm{\alpha},\bm{\beta})$ of the KKT system \eqref{eq:kkt}, it is easy to see that $(\operatorname{\Pi}_{\mathrm{rge}(\A^*)}(\bm{\varphi}),\z,\q,\bm{\alpha},\bm{\beta})$ is also a solution of \eqref{eq:kkt2}. 
\end{remark}

The inexact proximal ALM to solve \eqref{eq:opt-dualsocalm-ran} is presented as \Cref{alg:socinpalm}. 
\begin{algorithm}[ht]
	\caption{An inexact proximal ALM for solving the SOCP \eqref{eq:opt-dualsocalm-ran}}
	\label{alg:socinpalm}
	\KwIn{
		a penalty parameter $\sigma > 0$, 
		a dual step length $\tau \in (0, 2)$, and an initial point $(\bm{q}^{(0)}, \bm{\alpha}^{(0)}, \bm{\beta}^{(0)})$ in $ \mathds{R}^{|\mathcal{G}_0^s|+ \cdots + |\G_D^s|}\times\mathds{R}^{|\mathcal{G}_0^s|+ \cdots + |\G_D^s|}\times\mathds{Q}$.}
	\KwOut{the infinite sequence $\{(\bm{\varphi}^{(k)}, \z^{(k)}, \bm{q}^{(k)}, \bm{\alpha}^{(k)}, \bm{\beta}^{(k)})\}$.}
	\For{{$k=0,1,\ldots$}}{
		$(\bm{\varphi}^{(k+1)}, \bm{z}^{(k+1)}) \gets \mathop{\argmin}\limits_{\bm{\varphi}, \bm{z}} \Big\{ \mathfrak{L}_{\sigma}(\bm{\varphi}, \bm{z}, \bm{q}^{(k)}; \bm{\alpha}^{(k)}, \bm{\beta}^{(k)})\mid \bm{\varphi}\in \mathrm{rge}(\A^*)\Big\}$;
		
		$\bm{q}^{(k+1)} \gets \mathop{\argmin}\limits_{\bm{q}} \Big\{ \mathfrak{L}_{\sigma}(\bm{\varphi}^{(k+1)}, \bm{z}^{(k+1)}, \bm{q}; \bm{\alpha}^{(k)}, \bm{\beta}^{(k)}) \Big\}$;
		
		$(\bm{\alpha}^{(k+1)}, \bm{\beta}^{(k+1)}) \gets (\bm{\alpha}^{(k)}; \bm{\beta}^{(k)}) + \tau \sigma \left( \mathcal{M} \begin{pmatrix}
			\bm{\varphi}^{(k+1)} \\ \bm{z}^{(k+1)}
		\end{pmatrix} + \mathcal{N} \bm{q}^{(k+1)} - \begin{pmatrix}
			0 \\ \bm{d}
		\end{pmatrix} \right)$.
	}
\end{algorithm}

The following result ensures the convergence of \Cref{alg:socinpalm} without requiring any additional assumptions or conditions.
\begin{theorem}
	\label{thmconv}
	The sequence $\{(\bm{\varphi}^{(k)},\z^{(k)},\q^{(k)},\bm{\alpha}^{(k)},\bm{\beta}^{(k)})\}$ generated by \Cref{alg:socinpalm} converges to a solution of the KKT system \eqref{eq:kkt2}, and hence \eqref{eq:kkt}.  
\end{theorem}

\begin{proof}
	\Cref{alg:socinpalm} is an instance of \cite[Algorithm sGS-inPADMM]{cl21} applied to \eqref{eq:opt-dualsocalm-ran} without error terms and proximal terms, so we obtain the convergence of the former by examining the convergence theorem \cite[Theorem 4.2]{cl21} for the latter. 
	According to \Cref{remarkkkt} and \Cref{prop:sockkt}, the solution set to the KKT system \eqref{eq:kkt2} is nonempty, so \cite[Assumption 1]{cl21} is satisfied.
	Note that the operator $\A^*\A$ restricted to the subspace $\mathrm{rge}(\A^*)$ is positive definite, so it is easy to see from 
	\eqref{alf} that all the subproblems in \Cref{alg:socinpalm} are strongly convex. 
	Moreover, since \Cref{alg:socinpalm} does not involve proximal terms and the subproblems are strongly convex, we know that \cite[Assumption 2]{cl21} is also satisfied. 
	To invoke \cite[Theorem 4.2~(e)]{cl21}, we identify the first primal block variable $y$ in \cite[Theorem 4.2]{cl21} with $(\bm\varphi, \bm z)$ in \eqref{eq:opt-dualsocalm-ran}, and the second primal block variable $z$ is identified with $\bm q$ in \eqref{eq:opt-dualsocalm-ran}. 
	Consequently, the linear operator $\mathcal{G}$ in \cite{cl21} corresponds to $\mathcal{N}^*$. Hence, the surjectivity condition on $\mathcal{G}$ required in \cite[Theorem 4.2~(e)]{cl21} is the surjectivity of $\mathcal{N}^*$.
	For the linear operator $\mathcal{N} \q := (-\q;-\B\F\q)$ in \eqref{eq:opt-dualsocalm-ran}, its adjoint is given by 
	$\mathcal{N}^* (\bm\alpha,\bm\beta) = -\bm\alpha - \mathcal{F}^*\mathcal{B}^*\bm\beta$. 
	For any $\bm w \in \mathds{R}^{|\G^s_0|+\cdots|\G^s_D|}$, taking $\hat{\bm \alpha}: = -\bm w$ and $\hat{\bm\beta}: = \bm 0$ yields $\mathcal{N}^* (\hat{\bm\alpha},\hat{\bm\beta}) = \bm w$. Thus, $\mathcal{N}^*$ is surjective.
	Then, it follows from the conclusion (e) of \cite[Theorem 4.2]{cl21}
	that \Cref{thmconv} holds.
\end{proof}

The following remark explains why we adopt \Cref{alg:socinpalm} to solve \eqref{eq:opt-dualsocalm-ran}.

\begin{remark}
	When directly applying the ALM \cite{rock-alm} to \eqref{eq:opt-dualsocalm-ran}, 
	one needs to solve the subproblem at the $k$-th iteration in the following form 
	\begin{equation}\label{eq:inalm-sub}
		(\bm{\varphi}^{(k+1)},\z^{(k+1)},\q^{(k+1)})\approx \mathop{\argmin}\limits_{\bm{\varphi},\z,\q} \Big\{\mathfrak{L}_{\sigma}(\bm{\varphi},\z,\q;\bm{\alpha}^{(k)},\bm{\beta}^{(k)})\mid \bm{\varphi}\in \mathrm{rge}(\A^*)\Big\}. 
	\end{equation}
	The coupling of variables makes it difficult to compute a solution to this subproblem. 
	In contrast, the $k$-th iteration of the ADMM \cite{glowinski75,gabay76,clcoap} minimizes the objective in \eqref{eq:inalm-sub} with respect to  $(\bm{\varphi},\z)$ and $\q$ successively for only one cycle. 
	This can result in significantly more iterations than ALM. 
	Thanks to the theoretical analysis in \cite{cl21}, we know that for solving \eqref{eq:opt-dualsocalm-ran}, the error in each iteration of the ADMM, when viewed as an inexact ALM, can be quantified by introducing a proximal term. 
	The resulting \Cref{alg:socinpalm} is similar to the classic ADMM and appears to be an instance of the proximal ADMM \cite{fazel13}. 
	Nevertheless,  the condition $\tau\in  (0,2)$ on the dual step length distinguishes it from ADMM-type algorithms. 
	In the latter, convergence generally requires $\tau\in(0,\frac{1+\sqrt{5}}{2})$. 
	Thus, the convergence analysis of \Cref{alg:socinpalm} does not follow directly from \cite{fazel13} or other literature, but relies on \cite{cl21}. 
	Based on \cite[Theorem 4.1]{cl21}, we know that Steps 1 and 2 in \Cref{alg:socinpalm} can be integrated as a single step
	$$
	%\label{eq:phizq}
	\begin{array}{ll}
		\displaystyle
		(\bm{\varphi}^{(k+1)},\z^{(k+1)},\q^{(k+1)}) &\approx \mathop{\argmin}\limits_{\bm{\varphi},\z,\q} \Big\{\mathfrak{L}_{\sigma}(\bm{\varphi},\z,\q;\bm{\alpha}^{(k)},\bm{\beta}^{(k)}) \\[1mm]
		&\displaystyle
		\quad+ \frac{1}{2}\big\|(\bm{\varphi};\z;\q)-(\bm{\varphi}^{(k)};\z^{(k)};\q^{(k)})\big\|^2_{\mathcal{S}}\mid \bm{\varphi}\in \mathrm{rge}(\A^*)\Big\},
	\end{array} 
	$$
	where the error can be controlled by a summable sequence of non-negative real numbers, and the self-adjoint positive semidefinite linear operator $\mathcal{S}$ is defined by
	\begin{equation*}
		\mathcal{S}(\bm{\varphi};\z;\q) :=  
		\Big((\sigma\mathcal{M}^*\mathcal{N} \big(\mathcal{N}^*\mathcal{N} \big)^{-1}\mathcal{N}^*\mathcal{M})(\bm{\varphi};\z); \bm{0}\Big).
	\end{equation*}
	This equivalence may partially clarify why ADMM can deliver robust computational performance for DOT problems in the existing literature \cite{bb00,bb15}.
\end{remark}

To conclude this section, we present the detailed implementation of \Cref{alg:socinpalm}. 
According to \eqref{alf}, the augmented Lagrangian function $\mathfrak{L}_\sigma$ is separable with respect to $\bm{\varphi}$ and $\z$. 
Thus, the simultaneous minimization of obtaining $(\bm\varphi^{(k+1)},\z^{(k+1)})$ in \Cref{alg:socinpalm} can be separated in parallel as 
$$
\begin{cases}
	\bm\varphi^{(k+1)} \gets \mathop{\argmin}\limits_{\bm\varphi \in \mathrm{rge}(\A^*)}  \big\{\langle \bm c,\bm\varphi\rangle+\langle \A\bm\varphi-\q^{(k)},\bm\alpha^{(k)}\rangle + \frac{\sigma}{2}\|\A\bm\varphi-\q^{(k)}\|^2\big\}, \\
	\z^{(k+1)} \gets \mathop{\argmin}\limits_{\z\in\mathds{Q}}\big\{\langle \z -\B\F\q^{(k)}-\d,\bm\beta^{(k)}\rangle + \frac{\sigma}{2}\|\z -\B\F\q^{(k)}-\d\|^2\big\}.
\end{cases}
$$
For the quadratic optimization to obtain $\bm\varphi^{(k+1)}$, the first-order optimality condition leads $\bm\varphi^{(k+1)}$ to the unique solution in $\mathrm{rge}(\A^*)$ of the linear system
$$
\A^*\A \bm\varphi = \b^{(k+1)}:= \A^* (\q^{(k)}-\bm\alpha^{(k)}/\sigma) - \bm c / \sigma.
$$
Note that the operator $\A^*\A$ is the standard discrete Laplacian on a staggered grid with a Neumann boundary condition. 
So we utilize the discrete cosine transform (DCT) to diagonalize this system (see \cite[Section 4.8]{matrixcomp} for details). 
In our implementation, the DCT is executed using the FFT, which reduces the computational complexity to $\mathcal{O}(N\log N)$, where $N$ is the total number of grid points. 

The subproblem of obtaining $\z^{(k+1)}$ is a projection onto the set $\mathds{Q}$, the Cartesian product of several second-order cones. 
Thus, using the explicit formula provided in \eqref{app-proj}, whose efficiency was demonstrated in \Cref{tab:compareProj}, 
this subproblem is solved with linear complexity $\mathcal{O}(N)$. 
Also, according to \eqref{alf}, the update for $\q^{(k+1)}$ in \Cref{alg:socinpalm} minimizes a convex quadratic function. 
The optimality condition yields $\q^{(k+1)}$ as the solution to the nonsingular linear system
$$
(\mathcal{I}_{\q}+ \mathcal{F}^*\B^*\B\mathcal{F}) \q^{(k+1)} = \A\bm{\varphi}^{(k+1)}+\bm{\alpha}^{(k)}/\sigma+\mathcal{F}^*\B^*(\z^{(k+1)}-\d+\bm{\beta}^{(k)}/\sigma),
$$
where $\I_{\q}$  denotes the identity operator on the underlying space of $\q$.
According to Proposition \ref{prop:fbbf}, the operator $\mathcal{I}_{\q}+ \mathcal{F}^*\B^*\B\mathcal{F}$ is diagonal. Thus, this system can be solved element-wise with a linear complexity $\mathcal{O}(N)$. 
To summarize, the detailed implementation of \Cref{alg:socinpalm} is given by \Cref{alg:inpalm-exp}.
\begin{algorithm}[ht]
	\caption{Explicit computational steps of \Cref{alg:socinpalm}}
	\label{alg:inpalm-exp}
	\KwIn{
		a penalty parameter $\sigma > 0$, 
		a dual step length $\tau \in (0, 2)$, and an initial point $(\bm{q}^{(0)}, \bm{\alpha}^{(0)}, \bm{\beta}^{(0)})$ in $ \mathds{R}^{|\mathcal{G}_0^s|+ \cdots + |\G_D^s|}\times\mathds{R}^{|\mathcal{G}_0^s|+ \cdots + |\G_D^s|}\times\mathds{Q}$.}
	\KwOut{the infinite sequence $\{(\bm{\varphi}^{(k)}, \z^{(k)}, \bm{q}^{(k)}, \bm{\alpha}^{(k)}, \bm{\beta}^{(k)})\}$.}
	\For{\textnormal{$k=0,1,\ldots$}}{
		
		$\b^{(k+1)}\gets\A^*\left(\q^{(k)}-\bm{\alpha}^{(k)}/\sigma\right)-\c/\sigma$;
		\vspace{0.3em}
		
		$\bm{\varphi}^{(k+1)}\gets$ \textnormal{ Solve }$\A^*\A\bm{\varphi} = \b^{(k+1)}$\ \text{via fast Fourier transform};
		\vspace{0.3em}
		
		$\z^{(k+1)}\gets\operatorname{\Pi}_{\mathds{Q}}\left(\B\mathcal{F}\q^{(k)}+\d-\bm{\beta}^{(k)}/\sigma\right)$;
		\vspace{0.3em}
		
		$\bm{q}^{(k+1)} \gets (\mathcal{I}_{\q}+ \mathcal{F}^*\B^*\B\mathcal{F})^{-1}(\A\bm{\varphi}^{(k+1)}+\bm{\alpha}^{(k)}/\sigma+\mathcal{F}^*\B^*(\z^{(k+1)}-\d+\bm{\beta}^{(k)}/\sigma))$;
		\vspace{0.3em}
		
		$
		\bm{\alpha}^{(k+1)} \gets \bm{\alpha}^{(k)}+\tau\sigma (\A\bm{\varphi}^{(k+1)}-\q^{(k+1)})
		$;
		\vspace{0.3em}
		
		$
		\bm{\beta}^{(k+1)} \gets \bm{\beta}^{(k)} +\tau\sigma (\z^{(k+1)}-\B\F\q^{(k+1)}-\d)
		$.}
\end{algorithm}

\section{Numerical experiments}
\label{sec:numerical}
In this section, we conduct extensive numerical experiments to test the performance of the algorithm proposed in the previous section and compare it with other promising approaches. 
We will briefly review two well-developed software packages that can be applied to discrete DOT problems. 
In addition to \Cref{alg:socinpalm}, we list a few augmented Lagrangian-based algorithms to solve the SOCP reformulation \eqref{eq:opt-dualsocalm-ran}. 
Then, we evaluate the performance of all these software packages and algorithms.
High-quality transportation trajectories for all the tested problems (computed by our software) can be found in Supplementary Material (Online Resource 1), along with the videos for visualizing the transportation (Online Resource 2). 

\subsection{Off-the-shelf software packages}
The software package\footnote{\url{https://github.com/gpeyre/2013-SIIMS-ot-splitting}} of \cite{papa14} is based on primal-dual methods, including the
Douglas-Rachford splitting and the primal-dual hybrid gradient method, referred to as DR and PD.
In \cite{yjj24}, a multilevel-FISTA is used to solve discrete DOT problems\footnote{\url{https://github.com/Jiajia-Yu/FISTA_MFP_euc}}. These algorithms are based on the primal discrete DOT formulation \eqref{eq:opt-ene-bb} on staggered finite-difference grids: 
\begin{equation*}
	% \label{eq:pdhg-form}
	\min_{\bm{u}\in  \mathds{R}^{|\mathcal{G}^s_0|+\cdots+ |\mathcal{G}^s_D|}} \mathcal{J}(\mathcal{E}\bm{u}) + \delta_{\mathds{C}}(\bm{u}),
\end{equation*}
where $\mathcal{J}(\cdot)$ is the discrete objective function defined in \eqref{eq:opt-ene-bb}, 
and $\delta_{\mathds{C}}(\cdot)$ is the indicator function of the discrete continuity equation (see \cite[Section 3.5]{papa14} for more details). The operator $\mathcal{E}$ serves as an interpolation matrix, aligning each column of $\bm{u}$ into a unified grid.

When implementing the software package provided in \cite{papa14}, we found that both the DR and PD methods are sensitive to their step sizes, and we can use a finer step size rule given in \cite{chambolle} without introducing a computational burden to adjust them. 
This can reduce the number of iterations of these methods.  
So, we denote the corresponding modified algorithms as DRc and PDc. 
Moreover, we also found that the code of the multilevel-FISTA released in \cite{yjj24} includes a slight modification compared to the algorithm proposed in their paper. 
Specifically, when updating the density function, the authors impose a positive lower bound to ensure that the density remains positive throughout the iteration.
We observed that this modification can enhance the practical performance of the multilevel-FISTA, although the convergence of the modified algorithm is unknown. 
However, since such a modification performs quite well, we implement the software package of \cite{yjj24} using this modification.

\subsection{Augmented Lagrangian-based methods for comparison}
% \label{subsec:numerical}
In addition to the inexact proximal ALM (inpALM) given by \Cref{alg:socinpalm},  several other numerical methods based on the augmented Lagrangian function \eqref{alf} can be used to solve the SOCP reformulation \eqref{eq:opt-dualsocalm}. 
Here, we tested some of them that have the potential to solve \eqref{eq:opt-dualsocalm} efficiently. 
The details of these algorithms are collected in Supplementary Material (Online Resource 1).
\begin{itemize}
	\item[-] {ALG2}: the classic ADMM with a unit dual step length. 
	\item[-] pALM: 
	A proximal ALM utilizing the symmetric Gauss-Seidel (sGS) decomposition technique \cite{lxd16,lxdsgs} to decouple variables efficiently.
	\item[-] sGS-inpALM:
	A variant of pALM that incorporates a reordering of the linear system with the coefficient matrix $\A^*\A$ in \Cref{alg:inpalm-exp} (based on the techniques in \cite[Section 12.4]{saad} and \cite{lyl21}) and the application of the sGS decomposition \cite{cl17,cl21,xiao2018}.
	\item[-] acc-pADMM:  An accelerated version of the preconditioned ADMM in \cite{accadmm}. 
	\item[-] acc-sGS-pADMM: an extension of acc-pADMM incorporating the sGS decomposition for the subproblems. 
\end{itemize}
Both \Cref{alg:socinpalm} and the pALM use a large (dual) step length $\tau = 1.9$. 
For the sGS-inpALM, we employ a red-black ordering technique (see \cite[Section 12.4]{saad}) and use $\tau = 1.8$. 
Based on the numerical observations in \cite[Section 4]{accadmm}, for the acc-pADMM, we fix
$\rho=2$ and explore two values of $\theta$ ($2$ and $15$). 
Additionally, for acc-sGS-pADMM, we set  $\rho=2$ and  $\theta=2$.

We make the following remark regarding the iteration complexity of the algorithms employed in the numerical experiments. 

\begin{remark}
	All algorithms here fall within the category of first-order methods and theoretically have a sublinear convergence rate. 
	Specifically, the inexact proximal ALM exhibits a non-ergodic iteration complexity of $o(1/\sqrt{k})$ for the KKT residual \cite{cl21}. 
	For the accelerated preconditioned ADMM, it was established in \cite{accadmm} that it achieves a non-ergodic rate of $\mathcal{O}(1/k)$ in terms of the KKT residual. 
	Regarding the standard ADMM with a unit dual step length, \cite{monteiro} first established an ergodic $\mathcal{O}(1/k)$ rate with respect to the $\varepsilon$-subdifferential, whereas \cite{davis} proved a non-ergodic iteration complexity of $o(1/\sqrt{k})$, focusing on primal feasibility violations and the gap in the primal objective function value. 
	Furthermore, \cite{chambolle} demonstrated that PDHG achieves an ergodic $\mathcal{O}(1/k)$ complexity in terms of the primal-dual gap. 
	Finally, FISTA possesses a non-ergodic $\mathcal{O}(1/k^2)$ complexity for the objective value when the differentiable component possesses a Lipschitz continuous gradient \cite{fista}, which is satisfied if the density admits a positive lower bound.
\end{remark}

\subsection{Implementation details}
% \label{sec:detail}
The initial data are first scaled consistently across all evaluated algorithms.
When the multilevel strategy is employed, we follow \cite{ljl21} to obtain an initial solution on a coarse grid, which is then successively refined through a sequence of increasingly finer grids. 
Since our primary goal is to solve the discrete DOT problem \eqref{eq:opt-disdualori}, the stopping criteria of our algorithms are based on the corresponding optimality condition \eqref{eq:kkt-dual}.
Specifically, we measure the accuracy of a computed solution to the discrete DOT and its dual via  
$$
\eta_{\mathrm{dot}} \eqdef \max\{\eta_P,\eta_D,\eta_C\} 
$$
\begin{equation*}
	\mbox{with}\ 
	\begin{cases} 
		\eta_P \eqdef \frac{\|\A\bm{\varphi}-\q\|_{L^2}}{1+\|\A\bm{\varphi}\|_{L^2}+\|\q\|_{L^2}}, \quad
		\eta_D \eqdef \frac{\|\A^*\bm{\alpha}+\c\|_{L^2}}{1+\|\c\|_{L^2}},\\[2mm]
		\eta_C \eqdef\max\Big\{ \frac{\|\bm{\alpha}_0-\max\{0,f(\q)+\bm{\alpha}_0\}\|_{L^2}}{1+\|\bm{\alpha}_0\|_{L^2}+\|f(\q)\|_{L^2}}, \frac{\|\bar{\bm{\alpha}}- g(\bm{\alpha}_0,\bar{\q})\|_{L^2}}{1 + \|\bar{\bm{\alpha}}\|_{L^2}+\|g(\bm{\alpha}_0,\bar{\q})\|_{L^2}}\Big\} ,
	\end{cases}
\end{equation*}
where 
$\bm{\alpha}\equiv(\bm{\alpha}_0;\bar{\bm{\alpha}})$, $f(\q)\eqdef\q_0+\frac{1}{2}\mathcal{L}_{\T}\mathcal{L}_{\X}^*(|\q_1|^2;\ldots;|\q_D|^2)$, $\bar{\q} \equiv (\q_1;\ldots;\q_D)$, and $g(\bm{\alpha}_0,\bar{\q})\eqdef (g_1(\bm{\alpha}_0,\q_1);\ldots; g_D(\bm{\alpha}_0,\q_D))$ with  $g_d(\bm{\alpha}_0,\q_d) \eqdef \left(\mathcal{L}^d_{\X}\mathcal{L}^*_{\T}\bm{\alpha}_0\right) \odot \q_d$.  
Here, $\|\cdot\|_{L^2}$ denotes the discrete $L^2$ norm.
To ensure a consistent and rigorous comparison\footnote{The original implementations of the {DR} and {PD} in \cite{papa14} terminate after a pre-specified number of iterations. In contrast, the multilevel-FISTA in \cite{yjj24} utilizes the discrete $L^2$ norm of the difference between consecutive iterates as a heuristic to quantify the solution error.}, all the evaluated algorithms terminate when $\eta_{\mathrm{dot}}$ falls below a given tolerance \texttt{Tol}.  

We also adaptively adjust the penalty parameter $\sigma$ on the primal-dual KKT system  \eqref{eq:kkt}, instead of \eqref{eq:kkt-dual}, because the numerical implementation is based on the augmented Lagrangian function \eqref{alf} of \eqref{eq:opt-dualsocalm}. 
By doing this, we restart the algorithms with a new penalty parameter, using the most recent iterate as the initial starting point. 
For this purpose, define 
\begin{equation*}
	\begin{array}{ll}
		\eta^{\mathrm{soc}}_P\eqdef \max\Big\{ \eta_P, \frac{\|\z-\B\mathcal{F}\q-\d\|_{L^2}}{1+\|\d\|_{L^2}} \Big\} \ \mbox{and}\  
		\eta^{\mathrm{soc}}_D \eqdef\max \Big\{\eta_D,\frac{\|\mathcal{F}^*\B^*\bm{\beta}+\bm{\alpha}\|_{L^2}}{1+\|\mathcal{F}^*\B^*\bm{\beta}\|_{L^2}+\|\bm{\alpha}\|_{L^2}}\Big\}. 
	\end{array}
\end{equation*}
For algorithms without acceleration, we follow the strategy used in \cite{lam,tang23} to update $\sigma$ at specific iterations based on the ratio $\eta_{P}^{\text{soc}} / \eta_{D}^{\text{soc}}$ with the scheme
\begin{equation*}
	\sigma\gets 
	\begin{cases}
		s_{\sigma}\sigma & \text{if } \eta_{P}^{\text{soc}} / \eta_{D}^{\text{soc}} > s_r , \\
		\sigma/s_{\sigma} & \text{if } \eta_{P}^{\text{soc}} / \eta_{D}^{\text{soc}} < s_r^{-1}, \\
		\sigma  & \text{otherwise}.
	\end{cases}
\end{equation*}
Here, $s_{\sigma} > 1$ and $s_r > 1$ are pre-selected parameters, and we set $s_{\sigma} = 1.25$ and $s_r = 2$.
For acc-pADMM and acc-sGS-pADMM, we directly implement the technique in \cite{accadmm} to restart the algorithm, mainly based on the number of iterations. % or the variation of the parameter $\sigma$.

%It is noteworthy that we use the KKT residual \eqref{kkt-soc} from the KKT system \eqref{eq:kkt} to adjust the parameter $\sigma$, and the KKT residual \eqref{eq:residue-ot} from the KKT system \eqref{eq:kkt-dual} (for the original problem \eqref{eq:opt-disdualori}) to measure the accuracy of an approximate solution.

%--------------------------------------------------------------------------------------------
%--------------------------------------------------------------------------------------------

\subsection{Numerical results}
\label{subsec:num}
In this part, we report the numerical results of our experiments for four categories of DOT instances. 
All numerical experiments were implemented in MATLAB (R2023b) on a desktop PC with an Intel Core i9-9900KF CPU (8 cores, 3.60 GHz) and 64 GB of memory. 
The numerical results for the method proposed in this paper can be reproduced using the code on GitHub\footnote{\url{https://github.com/chlhnu/DOT-SOCP}}, where a Python version of the software is also available.

\subsubsection{Continuous densities in a flat domain}
We begin with continuous densities in the flat domain $[0,1]^{2}$.
Each example contains two functions $\varrho_0$ and $\varrho_1$, and we construct $\rho_0={\rm Normalize}(\varrho_0)+\delta$ and $\rho_1={\rm Normalize}(\varrho_1)+\delta$. 
Here, “Normalize” refers to normalizing a density to ensure that the discrete density sums to one, and $\delta\geq0$ is a constant to specify the uniform lower bound.
The resulting non-unit mass does not hinder the computation, provided that $\rho_0$ and $\rho_1$ have equal total mass.
The details of the five examples we tested are given as follows

\begin{example}[Ex-1]
	\label{exam1}
	Set $\chi := \sqrt{0.05}$, $(\mu_1,\mu_2) := (0.25,0.75)$, and define
	$$
	\begin{array}{ll}
		\varrho_0(\x) :=\mathrm{exp}\big(\frac{-(x_1-\mu_1)^2-(x_2-\mu_2)^2}{2\chi^2}\big)\quad\mbox{and}\quad 
		\varrho_1(\x): = \mathrm{exp}\big(\frac{-(x_1-\mu_2)^2-(x_2-\mu_1)^2}{2\chi^2}\big).
	\end{array}
	$$ 
	
\end{example}

\begin{example}[Ex-2]
	Set $(\chi_1, \chi_2) := (0.1,0.05)$, $(\mu_1,\mu_2) := (0.25,0.75)$, and define
	\begin{equation*}
		\begin{array}{ll}
			\varrho_0(\x) :=& \mathrm{exp}\big(\frac{-(x_1-\mu_1)^2-(x_2-\mu_1)^2}{2\chi_1^2}\big),\\[1mm]
			\varrho_1(\x) :=&  \mathrm{exp}\big(\frac{-(x_1-\mu_1)^2-(x_2-\mu_1)^2}{2\chi_2^2}\big)+\mathrm{exp}\big(\frac{-(x_1-\mu_1)^2-(x_2-\mu_2)^2}{2\chi_2^2}\big) \\[1mm]
			&+\mathrm{exp}\big(\frac{-(x_1-\mu_2)^2-(x_2-\mu_1)^2}{2\chi_2^2}\big)+\mathrm{exp}\big(\frac{-(x_1-\mu_2)^2-(x_2-\mu_2)^2}{2\chi_2^2}\big).
		\end{array}
	\end{equation*} 
\end{example} 

\begin{example}[Ex-3]
	Set $(a_1,a_2) := (3, 5)$, $(\mu_1,\mu_2) := (0.25,0.75)$, $\chi := 0.05$ and define
	\begin{equation*}
		\begin{array}{ll}
			\varrho_0(\x) :=&   \mathrm{exp}\big(-a_1|x_1-\mu_1|-a_2|x_2-\mu_1|\big) ,\\[1mm]
			\varrho_1(\x) := &  \mathrm{exp}\big(\frac{-(x_1-\mu_1)^2-(x_2-\mu_1)^2}{2\chi^2}\big)+\mathrm{exp}\big(\frac{-(x_1-\mu_1)^2-(x_2-\mu_2)^2}{2\chi^2}\big) \\[1mm]
			&+\mathrm{exp}\big(\frac{-(x_1-\mu_2)^2-(x_2-\mu_1)^2}{2\chi^2}\big)+\mathrm{exp}\big(\frac{-(x_1-\mu_2)^2-(x_2-\mu_2)^2}{2\chi^2}\big).
		\end{array}
	\end{equation*}
	
\end{example}

\begin{example}[Ex-4]
	Set $(c_1,c_2) := (0.5,0.5)$, $(\mu_1,\mu_2) := (0.25,0.75)$, $\chi := 0.05$, and define
	\begin{equation*}
		\begin{array}{ll}
			\varrho_0(\x) :=& (x_1-c_1)^4+(x_2-c_2)^4,
			\\[1mm]
			\varrho_1(\x) :=& \mathrm{exp}\big(\frac{-(x_1-\mu_1)^2-(x_2-\mu_1)^2}{2\chi^2}\big)+\mathrm{exp}\big(\frac{-(x_1-\mu_1)^2-(x_2-\mu_2)^2}{2\chi^2}\big)
			\\[1mm]
			&+\mathrm{exp}\big(\frac{-(x_1-\mu_2)^2-(x_2-\mu_1)^2}{2\chi^2}\big)+\mathrm{exp}\big(\frac{-(x_1-\mu_2)^2-(x_2-\mu_2)^2}{2\chi^2}\big).
		\end{array}
	\end{equation*}
\end{example}

\begin{example}[Ex-5]
	\label{example5} Set $\rho_0$ and $\rho_1$ as the indicator functions of two sets illustrated in the following images.
	\begin{center}
		\setlength{\fboxrule}{0.25pt}
		\setlength{\fboxsep}{ 0 pt}
		%\raisebox{-0.5cm}{\includegraphics[width=0.521\linewidth]{images/Fig2_1.jpg}}
		\fcolorbox{black}{white}{\includegraphics[width = .15\textwidth]{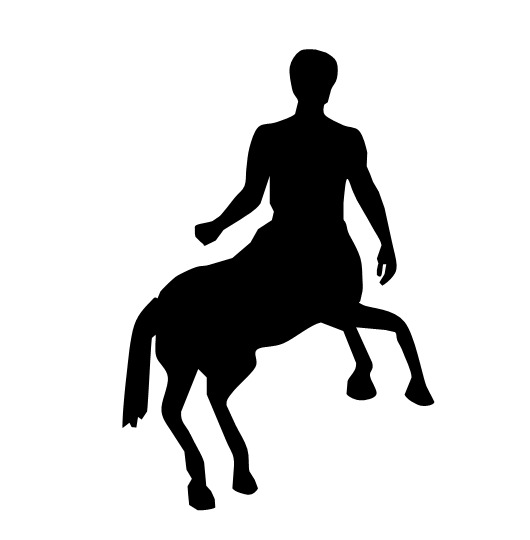}}
		\qquad 
		\fcolorbox{black}{white}{\includegraphics[width = .15\textwidth]{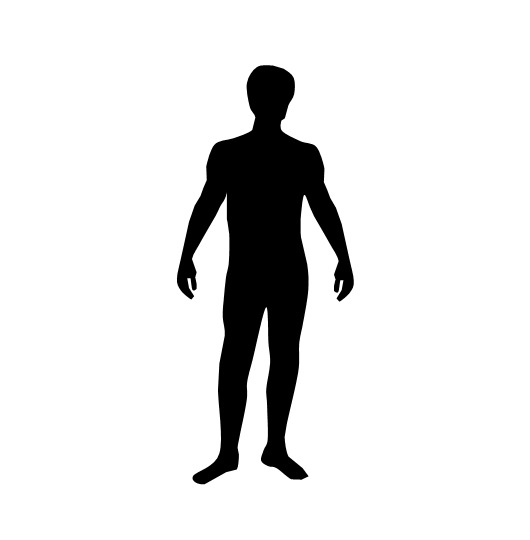}}
	\end{center}
\end{example}

Following the discrete framework in \Cref{subsec:dis}, we discretize all examples with grid sizes $(n_0,n_1,n_2) = (32,128,128)$ and $(n_0,n_1,n_2) = (64,256,256)$. We also select different values of $\delta$ to compare the robustness of different algorithms in handling non-negative probability densities. When employing the multilevel strategy, we first compute a solution with a precision of $\max\{\texttt{Tol}*10^{-v},10^{-6}\}$ on a certain grid and then increase the resolution. We apply the linear interpolation method proposed in \cite{ljl21} to the solutions on the coarse grids, using the interpolated results as initial points for computations on the subsequent finer grids. Specifically, for $(n_0,n_1,n_2) = (32,128,128)$, we use the grids $\left(\frac{n_0}{2^v},\frac{n_1}{2^v},\frac{n_2}{2^v}\right)$ for $v = 2,1$. For $(n_0,n_1,n_2) = (64,256,256)$, we use the grids $\left(\frac{n_0}{2^v},\frac{n_1}{2^v},\frac{n_2}{2^v}\right)$ for $v=3,2,1$. The maximum running time for each algorithm is set to 3,600 seconds.

We present our numerical results in Tables \ref{tab-gauss-tol3}, \ref{tab-gauss-tol4}, \ref{tab-admm}, and \ref{tab-admm-mul}. For a fair comparison, the reported computational times for all algorithms include the initialization phase via the multilevel strategy, where applicable (indicated by the iteration counts in these tables).

Tables \ref{tab-gauss-tol3} and \ref{tab-gauss-tol4} summarize the performance of {DR}, {PD}, {DRc}, {PDc}, multilevel-FISTA, and the proposed inpALM (\Cref{alg:socinpalm}) in solving \Cref{exam1} through \Cref{example5}. 
These evaluations were conducted across varying lower bounds $\delta$ and grid sizes, with the target precision set to \texttt{Tol}$=10^{-3}$ and $10^{-4}$, respectively.
As shown in Table \ref{tab-gauss-tol3}, {DR} consistently reaches the maximum time limit of 3,600 seconds across all test cases. While {DRc} and {PD} exhibit marginal improvements over {DR}, their achieved accuracy remains insufficient. In contrast, {PDc} successfully resolves more than half of the problem instances within the time limit. Regarding multilevel-FISTA, it reaches the time limit in all cases and occasionally diverges when $\delta$ is small (e.g., $0.05$ or $0$). Our observations indicate that while multilevel-FISTA converges rapidly in the initial stages, its progress significantly decelerates thereafter. This behavior is likely attributable to the safeguard mechanism within the multilevel-FISTA framework (\cite[Algorithm 1]{yjjmanifold}). Notably, inpALM emerges as the most efficient solver for the problems. 
Table \ref{tab-gauss-tol4} further highlights this advantage: when the target KKT-based relative error is tightened to $10^{-4}$, inpALM is the only algorithm capable of achieving the required precision within 1 hour.

\Cref{tab-admm} summarizes the performance of pALM, ALG2, acc-pADMM, acc-sGS-pADMM, and inpALM in solving \Cref{exam1} through \Cref{example5} based on the SOCP reformulation without using the multilevel strategy. The convergence tolerance was set to $\texttt{Tol} = 10^{-4}$. 

According to \Cref{tab-admm}, the inpALM consistently delivers superior performance, followed by pALM. We have omitted sGS-inpALM from \Cref{tab-admm} because its performance was found to be less competitive than that of acc-sGS-pADMM. Notably, while acc-sGS-pADMM performs effectively on grid sizes up to $(32, 128, 128)$, it encounters significant computational challenges when scaled to $(64, 256, 256)$.

In \Cref{tab-admm-mul}, we extend the evaluation by incorporating sGS-inpALM and applying the multilevel strategy across all solvers to generate initial points. 
Additionally, the iteration counts for each level are documented in the Supplementary Material (Online Resource  1). 
The results indicate that all algorithms benefit substantially from the multilevel initialization, achieving the $10^{-4}$ accuracy threshold more rapidly. Interestingly, for \Cref{exam1}, ALG2 slightly outperforms inpALM. However, for all the other more complicated instances, the inpALM maintains its position as the most robust and efficient solver.

\begin{landscape}
	\begin{table} 
		\scriptsize
		\setlength\tabcolsep{1pt}
		\centering
		\caption{Comparison of algorithms for DOT problems with tolerance $\texttt{Tol}=10^{-3}$.
			The algorithms are abbreviated as M (multilevel-FISTA) and I (inpALM). In the table,  ``Iter''  denotes the number of iterations for every level of the grids, ``-'' implies the algorithm diverges, and ``1h" stands for 3600s
		}
		\label{tab-gauss-tol3}
		\input{tables/Table2}
	\end{table}
\end{landscape}

\begin{landscape}
	\begin{table}
		\scriptsize
		\setlength\tabcolsep{1pt}
		\centering
		\caption{Comparison of algorithms for DOT problems with tolerance $\texttt{Tol}=10^{-4}$. The algorithms are abbreviated as M (multilevel-FISTA) and I (inpALM). In the table,  ``Iter''  denotes the number of iterations for every level of the grids, ``-'' implies the algorithm diverges, and ``1h'' stands for 3600s}
		\label{tab-gauss-tol4}
		\input{tables/Table3}
		
	\end{table}
\end{landscape}

\begin{landscape}
	\begin{table}
		\scriptsize
		\setlength\tabcolsep{1pt}
		\centering
		\caption{Comparison of algorithms (without multilevel strategy) for DOT problems with tolerance $\texttt{Tol}=10^{-4}$. The algorithms are abbreviated as P (pALM),  G (ALG2), $\text{A}_2$ (acc-pADMM with $\theta=2$), $\text{A}_{15}$ (acc-pADMM with $\theta=15$), $\text{S}_2$ (acc-sGS-pADMM with $\theta=2$), and I (inpALM). In the table,  ``Iter''  denotes the number of iterations, and ``1h'' stands for 3600s}
		\label{tab-admm}
		\input{tables/Table4}
	\end{table}
\end{landscape}

\begin{landscape}
	\begin{table}
		\scriptsize
		\setlength\tabcolsep{.5pt}
		\centering
		\caption{Comparison of algorithms for DOT problems with tolerance $\texttt{Tol}=10^{-4}$ and multilevel strategy. The algorithms are abbreviated as P (pALM), S (sGS-inpALM), G (ALG2), $\text{A}_2$ (acc-pADMM with $\theta=2$), $\text{A}_{15}$ (acc-pADMM with $\theta=15$), $\text{S}_2$ (acc-sGS-pADMM with $\theta=2$), and I (inpALM). In the table,  ``Iter''  denotes the total number of iterations for all levels of grids, and ``1h'' stands for 3600s}
		\label{tab-admm-mul}
		\input{tables/Table5}
	\end{table}
\end{landscape}

\subsubsection{Test of scalability}
Static discrete optimal transport algorithms are well-studied in optimization, with \cite{hd24,fast-comp,yl24,hot} representing the current state-of-the-art. 
In particular, the highly efficient software package in \cite{hot} was specialized for the $L^2$ OT problem. 
By leveraging an equivalent reduced model instead of the conventional linear programming discretization, it incorporates a Halpern acceleration scheme characterized by favorable low computational complexity. 
Consequently, it achieves superior execution speed and a significant memory advantage over many contemporary algorithms. For instance, it successfully computes the distance between two $512 \times 512$ gray images from the DOTmark dataset (e.g., \Cref{fig:dotmark}), a task where the well-known Sinkhorn method and the interior point method in Gurobi typically fail due to memory exhaustion, according to \cite[Table II]{hot}. 
Nevertheless, at a $1024 \times 1024$ resolution, this software also encounters memory overflow using our workstation with 64GB of RAM, as it theoretically necessitates approximately 144GB.
Motivated by these observations, we investigate the scalability of the proposed SOCP approach by addressing the $L^2$ OT problem from the DOT model at the challenging $1024 \times 1024$ scale. 
To this end, we test the following problem.

\begin{example}[Ex-DOTmark]
	\label{exdotmark}
	Select eight images and eight shapes ($512 \times 512$ resolution) from the DOTmark dataset \cite{dotmark}, as depicted in \Cref{fig:dotmark}. For each set, we randomly partition the images into source and target groups of equal size. Subsequently, these are synthesized into two $1024 \times 1024$ high-resolution images for our computational tests.
\end{example}

\begin{figure}[ht]
	\centering
	\includegraphics[width = 0.11\textwidth]{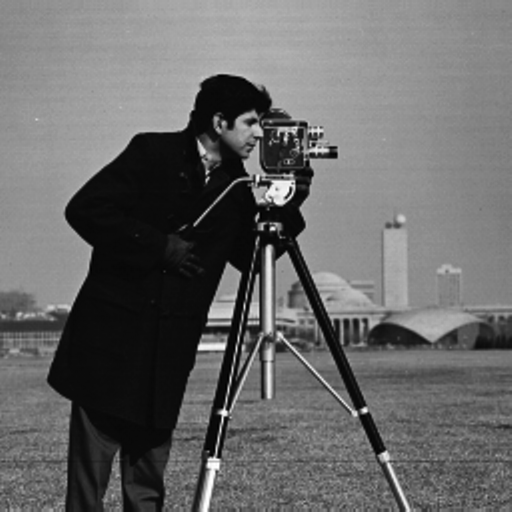}
	\includegraphics[width = 0.11\textwidth]{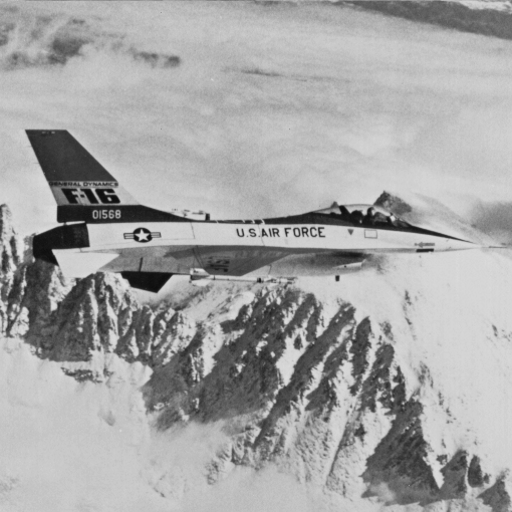}
	\includegraphics[width = 0.11\textwidth]{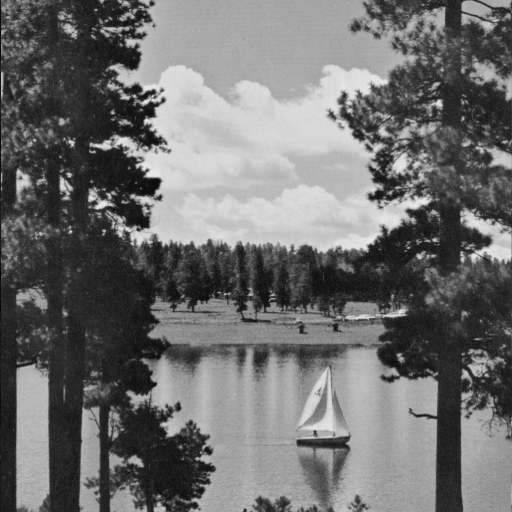}
	\includegraphics[width = 0.11\textwidth]{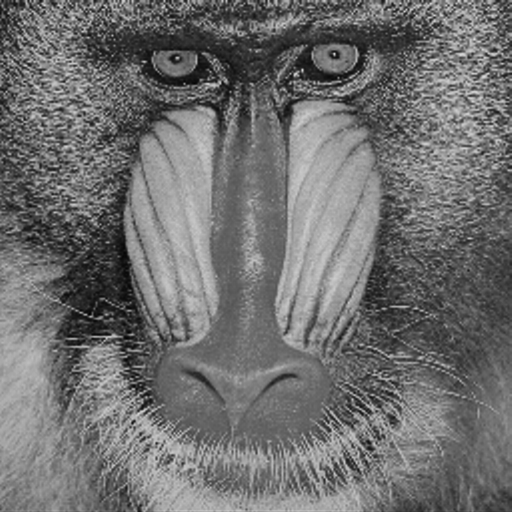}
	\includegraphics[width = 0.11\textwidth]{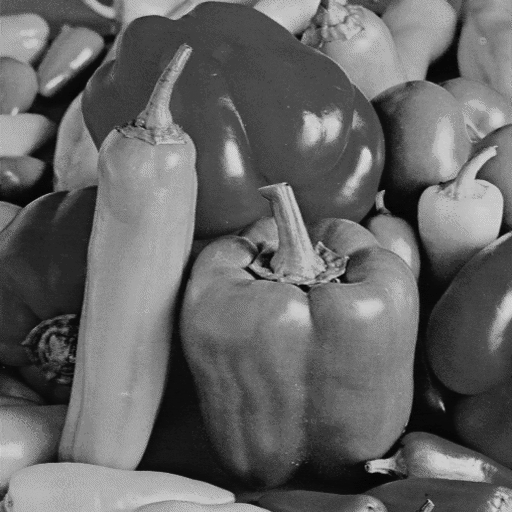}
	\includegraphics[width = 0.11\textwidth]{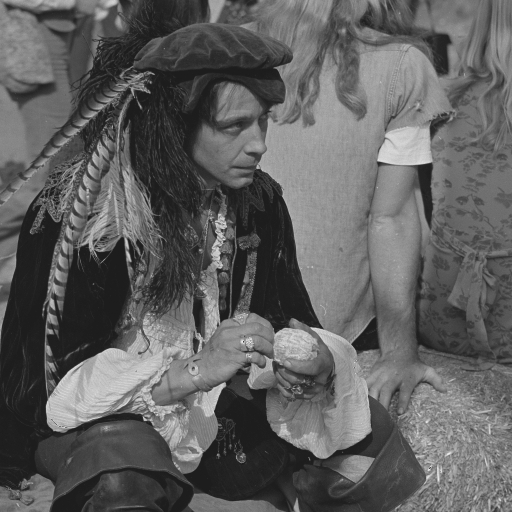}
	\includegraphics[width = 0.11\textwidth]{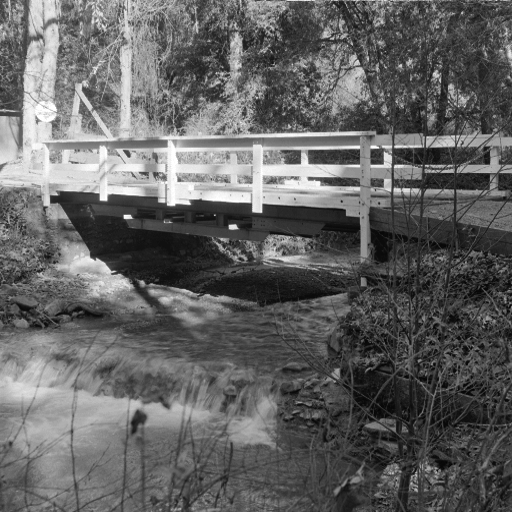}
	\includegraphics[width = 0.11\textwidth]{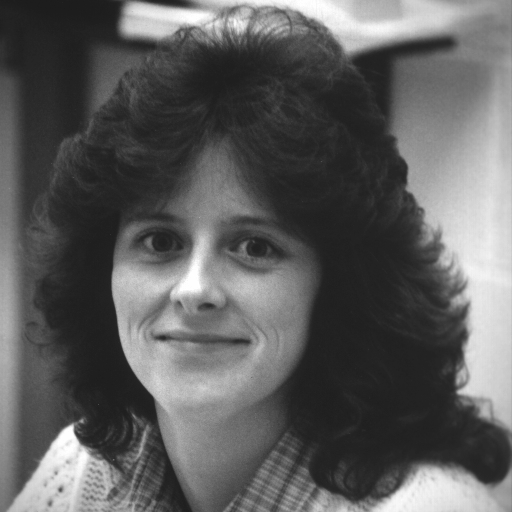}
	\\
	\vspace{2pt}
	\includegraphics[width = 0.11\textwidth]{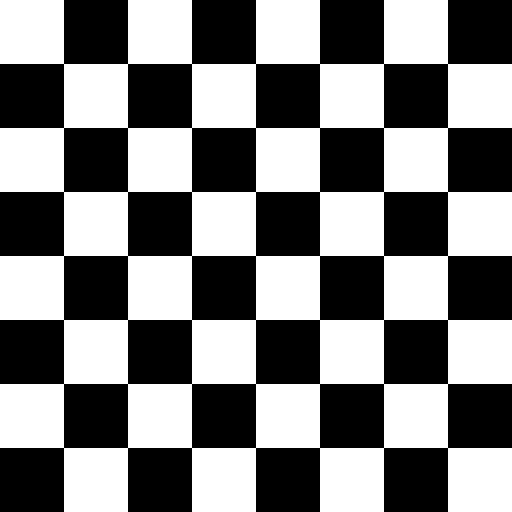}
	\includegraphics[width = 0.11\textwidth]{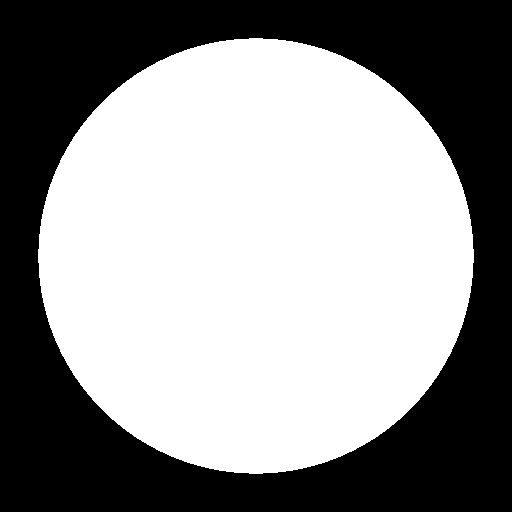}
	\includegraphics[width = 0.11\textwidth]{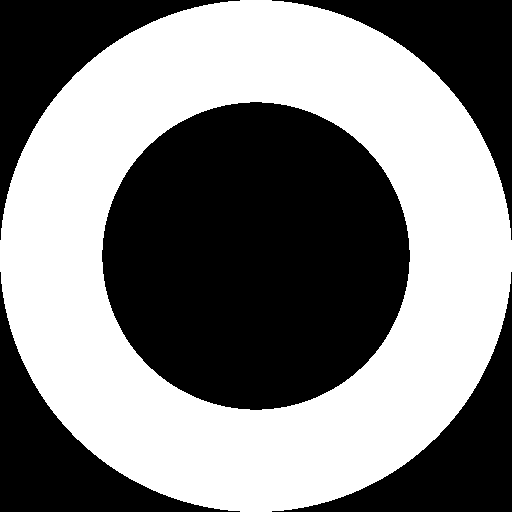}
	\includegraphics[width = 0.11\textwidth]{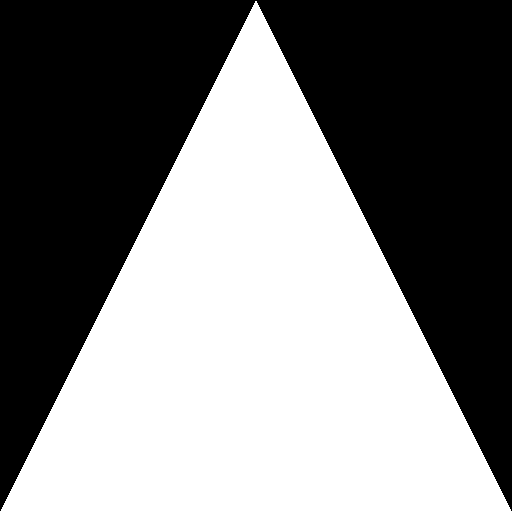}
	\includegraphics[width = 0.11\textwidth]{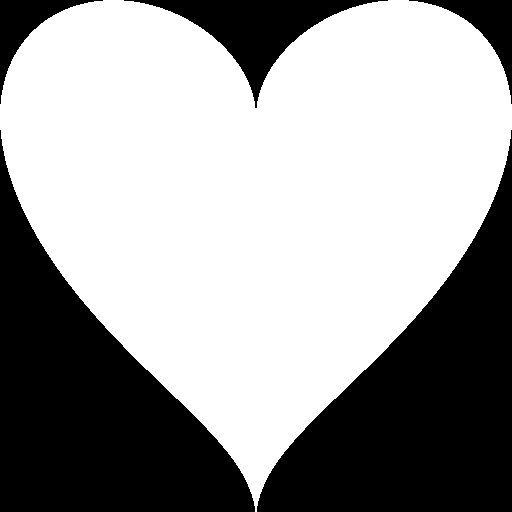}
	\includegraphics[width = 0.11\textwidth]{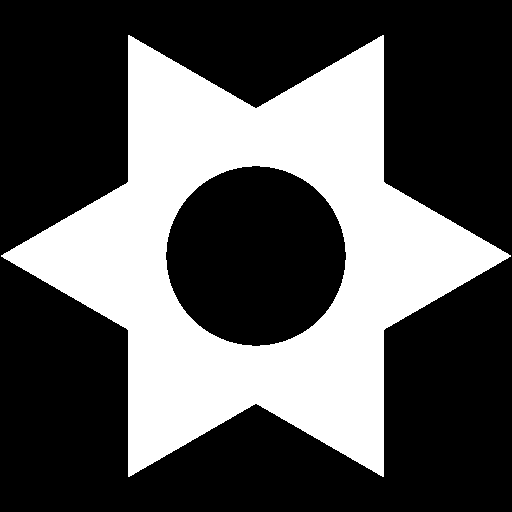}
	\includegraphics[width = 0.11\textwidth]{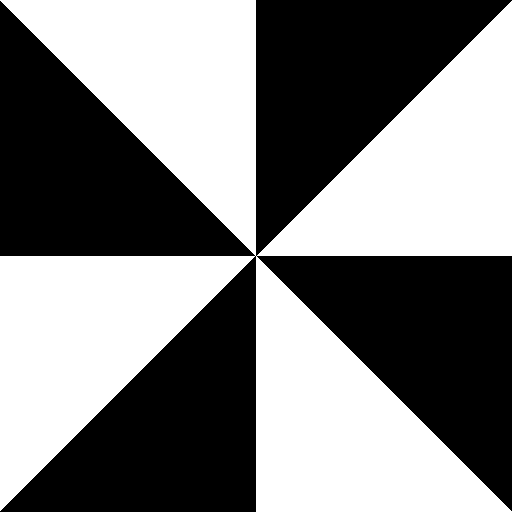}
	\includegraphics[width = 0.11\textwidth]{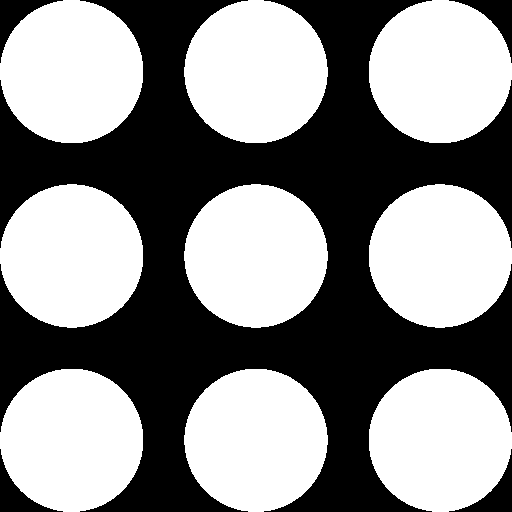}
	\caption{The upper row features images from the Classic Images category, and the bottom row features images from the Shapes category}
	\label{fig:dotmark}
\end{figure}

Utilizing the developed software, we computed both the $L^2$ Kantorovich-Wasserstein distances and the corresponding transportation trajectories on grid sizes up to $(32, 1024, 1024)$. 
The numerical results summarized in \Cref{tab:1024} demonstrate that our method successfully resolves these large-scale problems (using approximately 20 GB of memory during the computational process), thereby validating the superior scalability of the DOT model, the robustness of our SOCP reformulation approach, and the efficiency of Algorithm \ref{alg:socinpalm}.
Supplementary Material (Online Resource 1) contains extended numerical results for alternative classification types; notably, these outcomes align closely with the performance observed in \Cref{tab:1024}.

\begin{table}[ht]
	\centering
	\caption{Numerical results for \Cref{exdotmark} obtained by \Cref{alg:socinpalm}}
	\label{tab:1024}
	\input{tables/table_4_1024_short}

\end{table}

\subsubsection{Dirac measures in a flat domain}
In this experiment, we test the proposed approach for the following instance of the DOT problem from a continuous measure to a Dirac measure. 
\begin{example}[Ex-Dirac]
	\label{example7}
	Consider a continuous density $\rho_0$ and a discrete measure $\rho_1$ on $[0,1]^2$, as illustrated by \Cref{diracimage}, where $\rho_1$ is a sum of $30$ Dirac measures supported at randomly generated locations.
\end{example}

\begin{figure}
	\begin{subfigure}{0.49\textwidth}
		\centering
		\includegraphics[width = .45\textwidth]{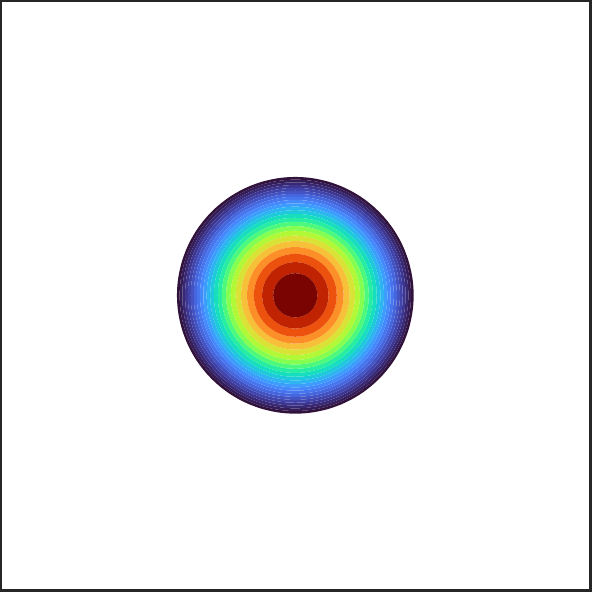}
		\includegraphics[width = .45\textwidth]{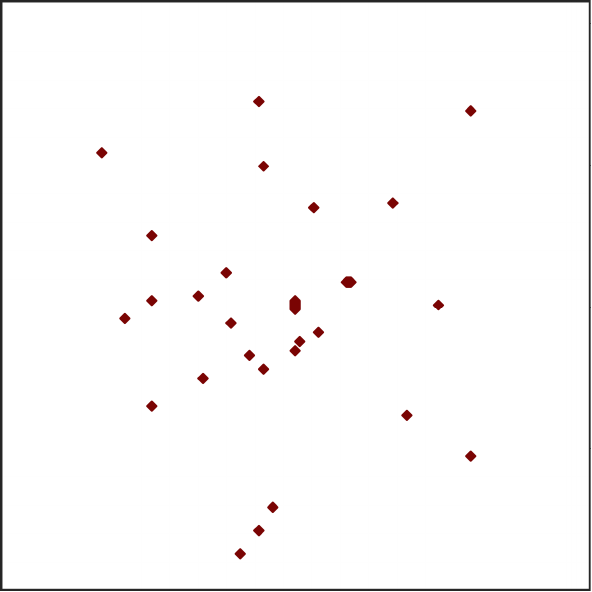}
		\caption{$\rho_0$ and $\rho_1$}
		\label{diracimage}
	\end{subfigure}
	\begin{subfigure}{0.49\textwidth}
		\centering
		\includegraphics[width = .9\textwidth]{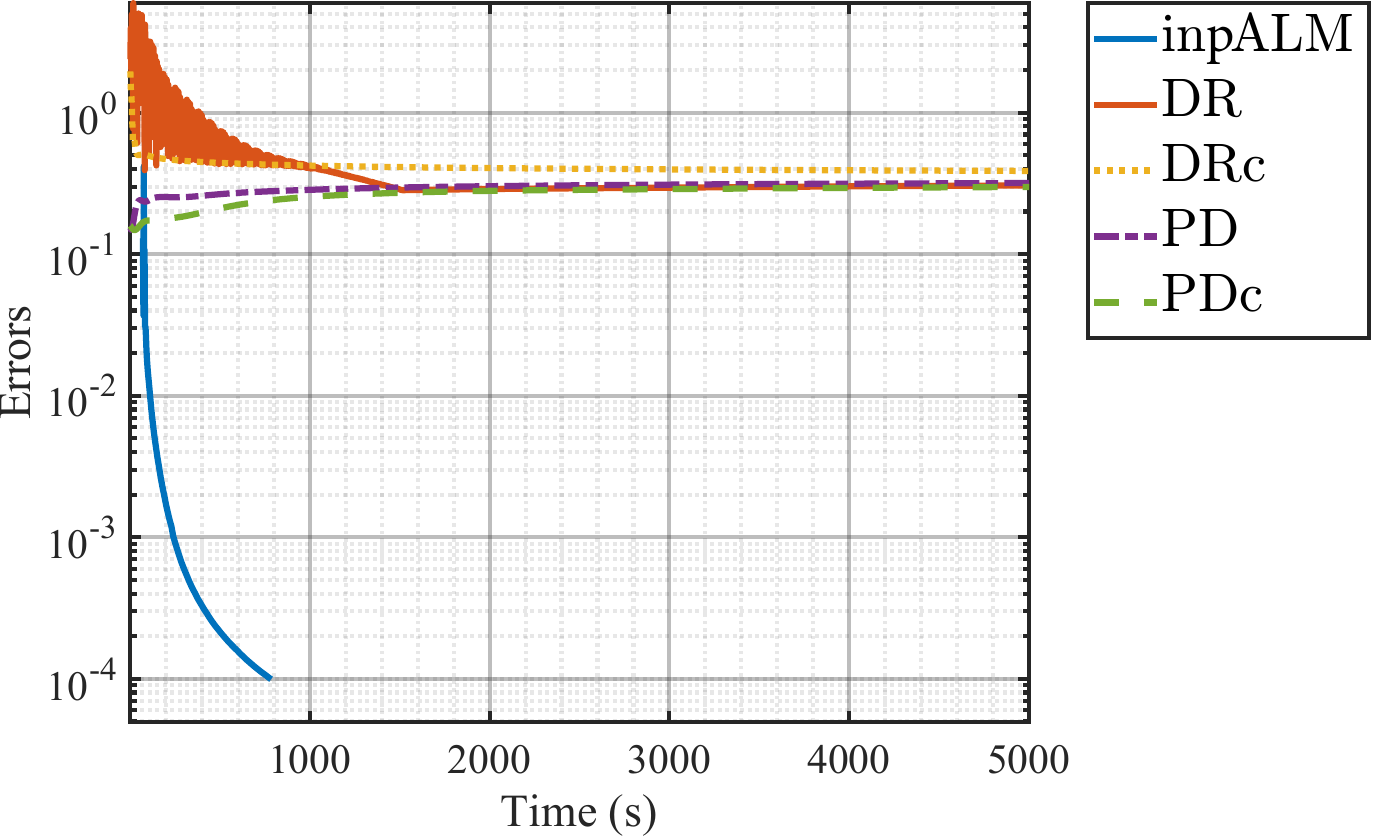}
		\caption{Comparison of efficiency}
		\label{performance}
	\end{subfigure}
	\caption{Measures in \Cref{example7} and numerical results}
\end{figure}

We discretize \Cref{example7} with the grid size $(n_0,n_1,n_2) = (128,128,128)$. The maximum running time is set to 5000 seconds and $\texttt{Tol} = 10^{-4}$.  
\Cref{performance} shows the decay of the relative error for stopping versus computational time, in which the starting time for each algorithm in \Cref{performance} indicates that the algorithm enters the finest grid.
We found that PDc (4294 steps, 5000.4s) performs the best among 
DR, DRc, PD, and PDc, while \Cref{alg:socinpalm} ($769|1169|1809$ steps, 779.90s) has a much more competitive performance.
Moreover, one can also observe from the snapshots of the evolutions computed by the two algorithms in Supplementary Material (Online Resource  1) that inpALM exhibits a relatively more uniform and stable evolution.

\subsubsection{Weighted DOT for irregular domains}
In this part, we extend our numerical investigation beyond rectangular grids to evaluate the performance of the inexact proximal ALM on the transport problem in irregular domains.  
We model these complex geometries by assigning weights to specific sub-regions, penalizing mass transport through the designated areas to force the optimal trajectory to bypass the obstacles.
Here, we introduce only the most essential parts of the procedure, and the routine details are provided in Supplementary Material (Online Resource 1).

Introducing a positive lower semicontinuous penalty function $\kappa(\x)$ defined on $X$ with a uniform positive lower bound, the weighted DOT problem reformulates \eqref{eq:opt-ene} to 
\begin{equation}
	\label{eq:opt-weight}
	\min_{\rho,\bm{v}}
	\left\{
	\int_0^1\int_{X} \frac{1}{2}\|\bm{v}(t,\cdot)\|^2 \kappa(\x) \dd\rho(t,\cdot) \dd t
	\ \Big\vert\ 
	\begin{array}{ll}
		\partial_t\rho+\mathrm{div}_{\x}(\rho\bm{v})=0,  \\
		\rho(0,\cdot)= \rho_0,  \rho(1,\cdot)= \rho_1 
	\end{array}
	\right\}.
\end{equation}
Transportation through points with large $\kappa(\x)$ is more expensive; thus, the optimal trajectory tends to avoid such points.
Following the analysis of \cite[Proposition 5.18]{Santamot}, one can prove that the objective function of \eqref{eq:opt-weight} admits an equivalent form given by
$$
\sup_{a, \bm{b}}\Big\{\int_\Omega a(t,\x)\dd\rho +\int_\Omega \bm{b}(t,\x)\cdot\dd\bm{m}\mid(a,\bm{b})\in  \mathcal{P}_\kappa \Big\},
$$
where  
$\mathcal{P}_\kappa: = \Big\{ (a, \bm b) \in C(\Omega) \times C(\Omega)^D \mid a + \frac{\|\bm b\|^2}{2\kappa(\x)} \le 0 \Big\}$.

Repeating the procedure for getting \eqref{eq:opt-dual}, the dual problem of the weighted DOT problem is given by
\begin{equation}
	\label{eq:weighted-conti}
	\max_{\phi,\p} \{G(\phi) - \delta_{\P_\kappa}(\bm p) \mid \nabla_{t,\x} \phi - \p = 0\},
\end{equation}
where $\bm p : = (a,\b) \in C(\Omega) \times C(\Omega)^D$. 
To discretize \eqref{eq:weighted-conti} using the procedure introduced in \Cref{subsec:dis}, we use a positive weighted vector  $\bm{\omega}\in \mathds{R}^{|\mathcal{G}^s_0|}_+$ to represent the discretization of $1/\kappa(x)$ at $x_{\i}$,
obtaining the discrete weighted DOT problem
\begin{equation}
	\label{eq:opt-weightdualori}
	\min\limits_{\bm{\varphi}, \q} 
	\big\{ 
	\langle \c,\bm{\varphi} \rangle \mid  
	\A \bm{\varphi} =\q\in  \mathds{P}_\omega
	\big\},
\end{equation}
where the weighted constraint set $\mathds{P}_{\omega}$ is defined by
\begin{equation*}
	\begin{array}{ll}
		\mathds{P}_{\omega}:=\Big\{  \q = (\q_0; \ldots; \q_D) \mid  
		\q_d \in  \mathds{R}^{|\G^s_d|}, 
		\\[1mm]
		\qquad
		~\qquad
		(\q_0)_{\i}+\frac{1}{2}\left( \bm{\omega}\odot\mathcal{L}_{\T}\mathcal{L}_{\X}^*(|\q_1|^2;\ldots ;|\q_D|^2)\right)_{\i}\leq 0\quad
		\forall\, \i \in \G^s_0 
		\Big\}.
	\end{array}
\end{equation*}
The discrete parabolic set $\mathds{P}_{\omega}$ preserves the same second-order cone structure as the set $\mathds{P}$. Hence, the SOCP reformulation approach also applies to \eqref{eq:opt-weightdualori}, and \Cref{alg:socinpalm} is still applicable, convergent, and admits the same implementation as \Cref{alg:inpalm-exp}.

In our numerical experiment, we construct the following examples to test the SOCP reformulation approach for the weighted DOT problem.  We define the obstacles by assigning a sufficiently small positive value to $\omega_{\i}$ in the infeasible regions and setting $ \omega_{\i} =1$ elsewhere.  
\begin{example}[Ex-heart]
	\label{example-heart}
	Transport between two normalized densities  $\rho_0$ and $\rho_1$
	in a heart-shaped domain with a heart-shaped obstacle, illustrated by \Cref{heartimage}. 
	Within the infeasible regions, the weights are set to $\omega_{i_0,i_1,i_2} = 10^{-6}$. 
\end{example}

\begin{example}[Ex-maze]
	\label{example6}
	Transport between two normalized densities  $\rho_0$ and $\rho_1$ in a domain with obstacles, illustrated by \Cref{mazeimage}. 
	Within the infeasible regions, the weights are set to $\omega_{i_0,i_1,i_2} = 10^{-6}$. 
\end{example}

\begin{figure} 
	\begin{subfigure}[b]{0.49\textwidth}
		\centering
		\includegraphics[width=0.45\linewidth]{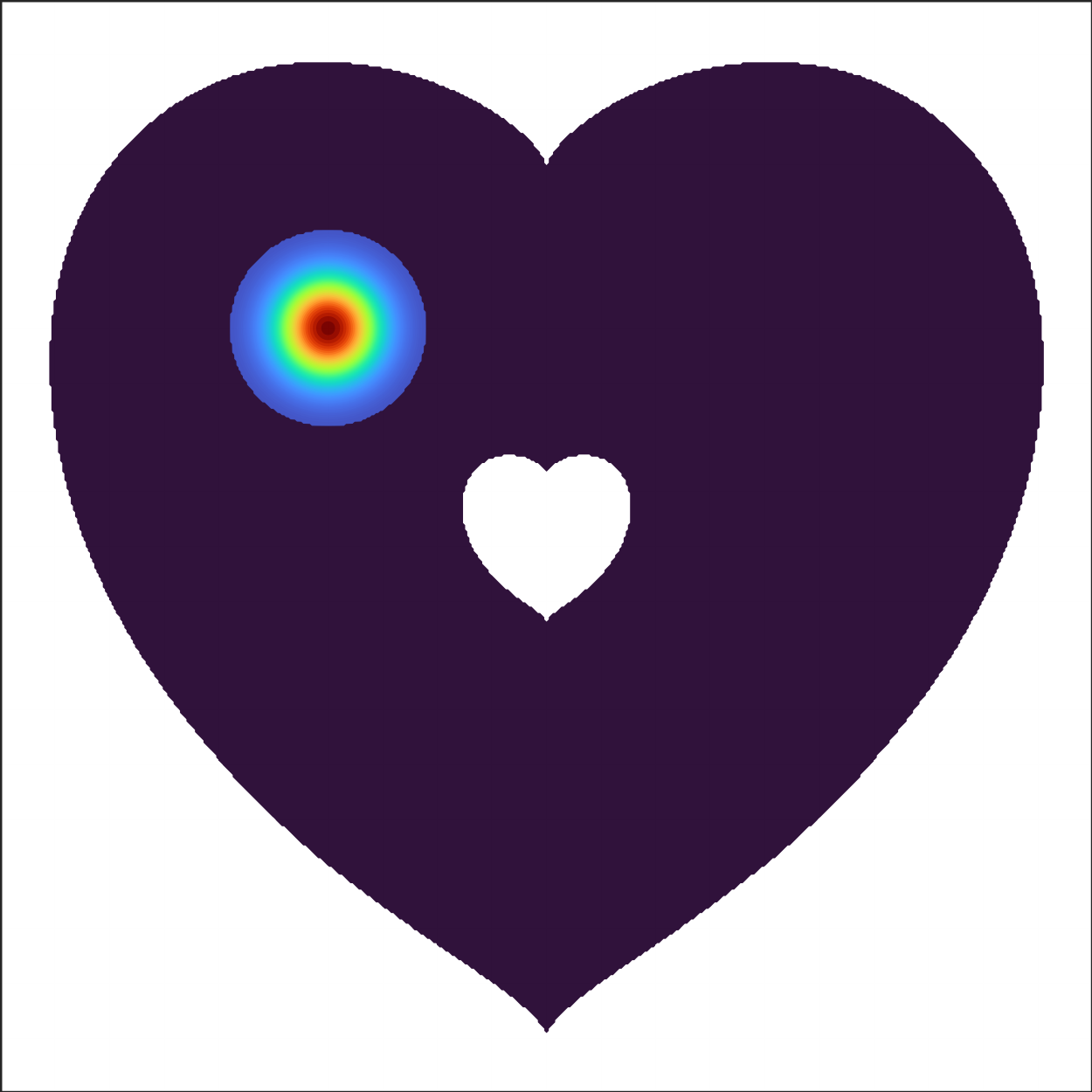}
		\includegraphics[width=0.45\linewidth]{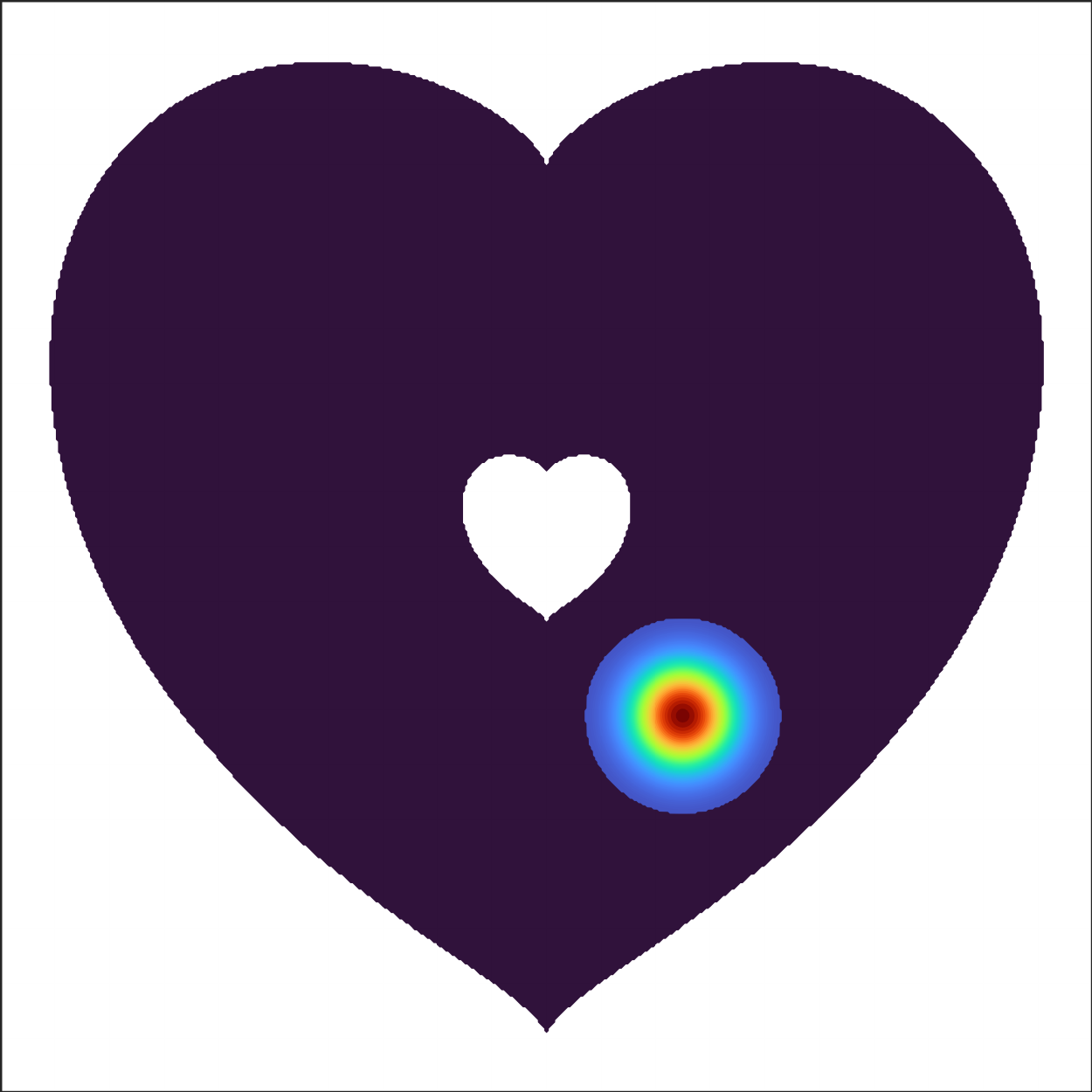}
		\caption{$\rho_0$ and $\rho_1$ in \Cref{example-heart}}
		\label{heartimage}
	\end{subfigure}
	\begin{subfigure}[b]{0.49\textwidth}
		\centering
		\includegraphics[width = .45\textwidth]{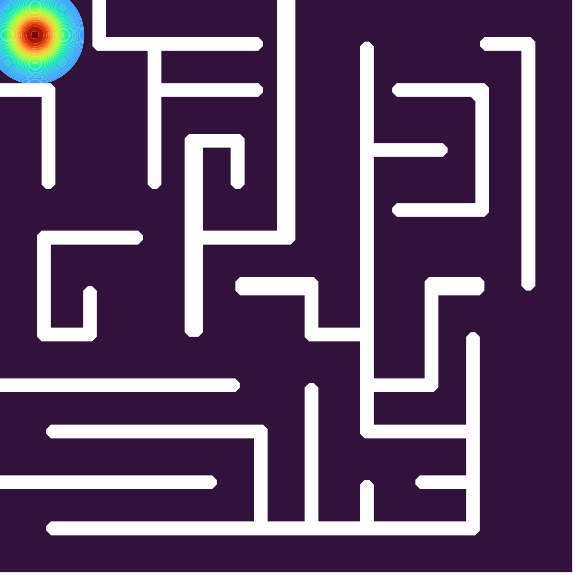}
		\includegraphics[width = .45\textwidth]{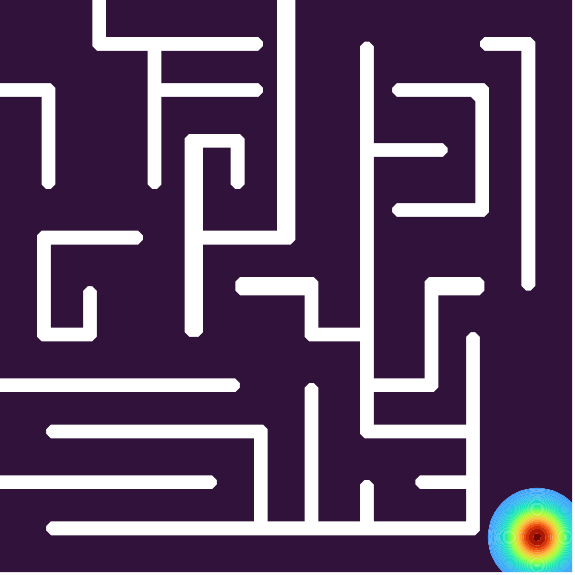}  
		\caption{$\rho_0$ and $\rho_1$ in \Cref{example6}}
		\label{mazeimage}
	\end{subfigure}
	\caption{Densities in \Cref{example-heart} and 
		\Cref{example6}}
\end{figure}

\begin{figure}
	\begin{subfigure}[b]{0.49\textwidth}
		\centering
		\includegraphics[width= .9\linewidth]{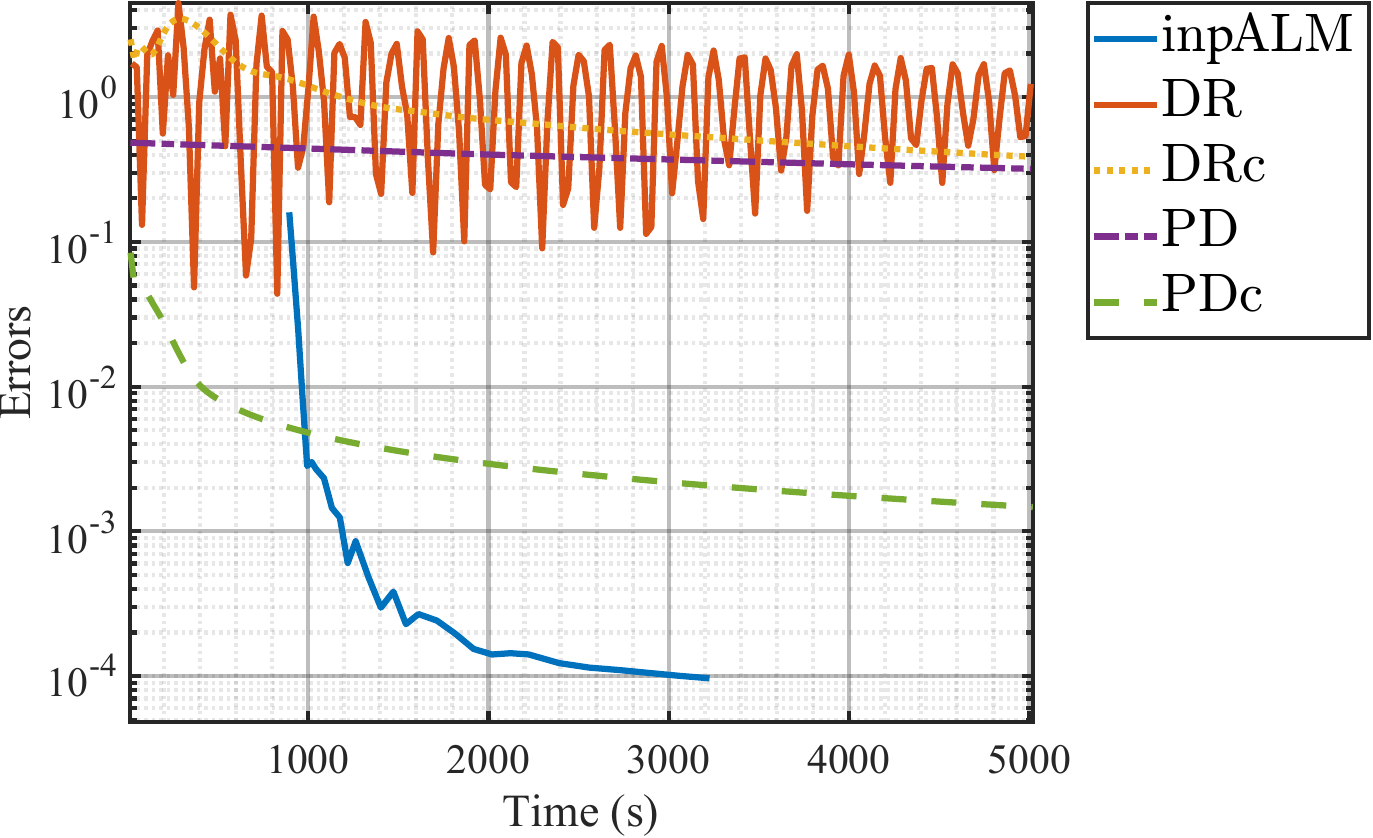}
		\caption{\Cref{example-heart}}
	\end{subfigure}
	\begin{subfigure}[b]{0.49\textwidth}
		\centering
		\includegraphics[width= .9\linewidth]{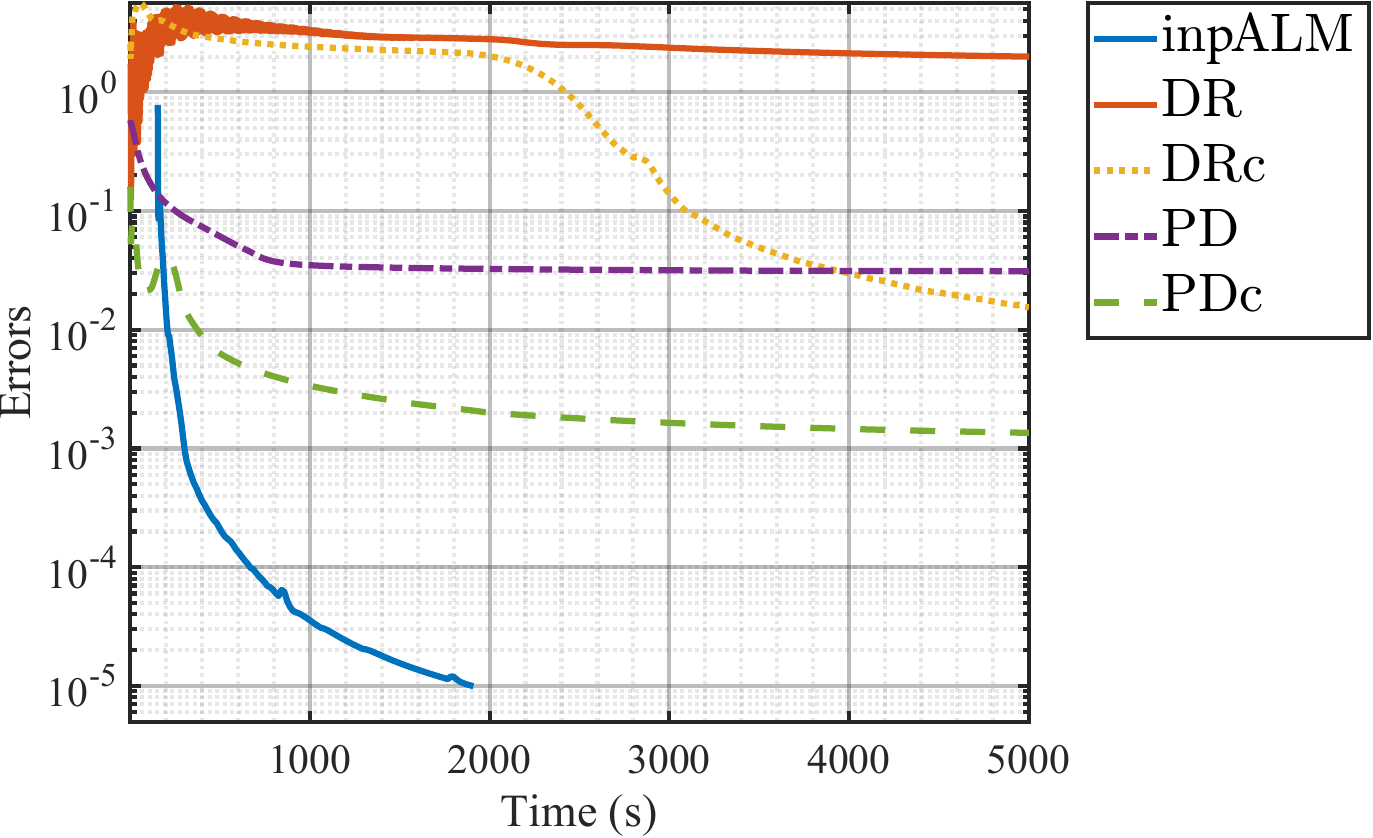}
		\caption{\Cref{example6}}
	\end{subfigure}
	\caption{Numerical results on comparing the efficiency for \Cref{example-heart} and \Cref{example6}}
	\label{fig:result89}
\end{figure}

We discretize \Cref{example-heart} on a grid with $(n_0,n_1,n_2) = (128,512,512)$ and use $\texttt{Tol} = 10^{-4}$.  
For \Cref{example6}, we set $(n_0,n_1,n_2) = (128,128,128)$ and a tighter tolerance  $\texttt{Tol} = 10^{-5}$. 
A uniform maximum runtime of 5000 seconds is imposed for both examples.
\Cref{tab:weighted} summarizes the numerical performance of \Cref{alg:socinpalm} in comparison with {PDc}, which was identified as the most competitive among the baseline algorithms ({DR}, {DRc}, {PD}, and {PDc}).
For all the tested algorithms, the decay of the KKT residual relative to computational time is illustrated in \Cref{fig:result89} (the starting time indicates entering the finest grid). 
Notably, the proposed inexact proximal ALM successfully achieves the required precision for both problems, while the other algorithms fail to reach the prescribed tolerance within the maximum time limit. 
Furthermore, as demonstrated by the density evolution snapshots in Supplementary Material (Online  Resource  1), the inexact proximal ALM appears to maintain a concentrated density distribution during the transportation process, which is more reasonable in the sense of OT.

\begin{table}[ht]
	\centering
	\caption{Numerical results for \Cref{example-heart} and \Cref{example6}}
	\label{tab:weighted}
	\input{tables/table_weighted}
\end{table}

In summary, the four categories of numerical evidence presented above suggest that the proposed SOCP reformulation provides an effective framework for addressing DOT problems on staggered grids. In particular, the inexact proximal ALM (Algorithm \ref{alg:socinpalm}) appears to be a competitive choice for solving these reformulated optimization problems. The proposed approach demonstrates the potential to handle large-scale discretized OT problems and DOT problems in irregular domains. Furthermore, it exhibits numerical robustness in the presence of Dirac measures.

\section{Conclusions}
\label{sec:conclu}
This paper presents an efficient numerical framework for solving DOT problems on staggered grids by leveraging an SOCP reformulation. 
Based on establishing the equivalence between the discretized DOT problem and a linear SOCP, we used an inexact decomposition-based proximal ALM to solve the SOCP reformulation. We developed an open-source software package that encourages further research and practical applications. 
Extensive numerical experiments demonstrate that our method consistently achieves high precision and superior computational efficiency across diverse and challenging scenarios. 
Beyond the scope of this work, the underlying SOCP methodology provides a versatile foundation for addressing a broader class of optimization problems in Wasserstein spaces, a direction we intend to explore in future research.
Moreover, it is also of interest to see whether the numerical techniques designed in this work can be integrated with grid-free frameworks, such as \cite{dl-geo,dl-ot}, for solving high-dimensional DOT problems with both efficiency and precision.

\bigskip
\noindent
\small \textbf{Funding} 
This work was supported by
the National Key R\&D Program of China (No. 2021YFA001300),
the National Natural Science Foundation of China (No. 12271150).

\medskip
\noindent
\textbf{Author Contributions}
All authors contributed equally to the study conception, design, writing, and numerical experiments. 
All authors read and approved the final manuscript.

\medskip
\noindent
\textbf{Data availability}
The datasets generated during and/or analyzed during the current study are available in the GitHub repository (\url{https://github.com/chlhnu/DOT-SOCP}). 
\section*{Declarations}

\noindent
\textbf{Competing Interests} 
The authors have no relevant financial or non-financial interests to disclose.

\end{document}

%% file: tables/Table2.tex
\begin{tabular}{lcccccccccccccccccccccccccccc}
\bottomrule[1pt]
\multirow{2}{*}{Problem} 
& \multirow{2}{*}{$\delta$} 
& & & 
\multicolumn{6}{c}{\(\eta_{\text{dot}}\)} & & & 
\multicolumn{6}{c}{Iter} & & & 
\multicolumn{6}{c}{Time (s)} \\
\cline{5-10} \cline{13-18} \cline{21-26}
          &            & &        & DR      & DRc      & PD      & PDc      & M & I                   & &           & DR      & DRc      & PD      & PDc      & M & I             & &           & DR      & DRc      & PD      & PDc      & M & I      \\
    \hline
% DR , DRc , PD , PDc , MFISTA , inPALM	
Ex-1(32-128-128) &    0 		 & & 		 & 2.26-2 & 9.99-4 & 1.00-3 & 9.93-4 & 4.30-1 & 8.70-4 		 & & 		 &  7474 &  3084 &  4121 &    81 &  5001,509,3368 &  144,37,16    & & 		 &      1h &  1027.2 & 1349.8 &   27.1 &      1h &     \bf3.9	 \\

Ex-1(64-256-256) &    0 		 & & 		 & 3.20-1 & 6.96-3 & 7.21-2 & 9.98-4 & 4.29-1 & 7.81-4 		 & & 		 &  1038 &  1451 &  1502 &    72 &  5001,778,26,525 & 214,59,31,10 & & 		 &      1h &      1h &      1h &  168.1 &      1h &    \bf18.4	 \\

Ex-1(32-128-128) & 0.05 		 & & 		 & 2.04-2 & 9.99-4 & 1.00-3 & 9.86-4 & 3.89-1 & 7.84-4 		 & & 		 &  1482 & 1696 &  3103 &    72 &   5001,5001,7553 &  114,37,13    & & 		 &      1h &  559.9 & 1009.2 &   24.2 &       1h &     \bf3.4	 \\

Ex-1(64-256-256) & 0.05 		 & & 		 & 3.04-1 & 2.44-3 & 7.30-2 & 9.99-4 &       - & 7.89-4 		 & & 		 &  1038 &  1318 &  1500 &    67 &    - & 159,49,25,7    & & 		 &      1h &      1h &      1h &  156.2 &       - &    \bf13.3	 \\

Ex-1(32-128-128) &  0.1 		 & & 		 & 1.74-2 & 9.97-4 & 9.99-4 & 9.86-4 & 4.02-1 & 9.80-4 		 & & 		 &  7504 &   904 &  2682 &    67 &  142,44,5590 &   99,31,10      & & 		 &      1h &  299.7 &  866.6 &   23.1 &      1h &     \bf2.8	 \\

Ex-1(64-256-256) &  0.1 		 & & 		 & 2.90-1 & 1.00-3 & 7.27-2 & 9.86-4 & 4.01-1 & 7.10-4 		 & & 		 &  1038 &   969 &  1511 &    64 &    141,44,18,886    & 129,43,19,7    & & 		 &      1h &  2402.7&      1h &  150.4 &      1h &    \bf13.4	 \\

Ex-2(32-128-128) &    0 		 & & 		 & 1.88-2 & 1.00-3 & 1.00-3 & 9.99-4 &       - & 9.47-4 		 & & 		 &  7466 &  6880 &  8401 &   539 &      - &  414,144,19    & & 		 &      1h & 2341.5 & 2715 &  184.9 &       - &    \bf6.8	 \\

Ex-2(64-256-256) &    0 		 & & 		 & 4.02-1 & 1.13-2 & 4.29-2 & 1.00-3 &       - & 8.76-4 		 & & 		 &  1036 &  1438 &  1511 &   490 &      - & 649,264,31,13  & & 		 &      1h &      1h &      1h & 1160 &       - &    \bf26.0	 \\

Ex-2(32-128-128) & 0.05 		 & & 		 & 1.87-2 & 1.00-3 & 1.00-3 & 9.99-4 &       - & 9.82-4 		 & & 		 &  7474 &  2376 &  3834 &   489 &      - &  364,144,16    & & 		 &      1h &  803.9 & 1239 &  166.6 &       - &    \bf6.3	 \\
Ex-2(64-256-256) & 0.05 		 & & 		 & 3.58-1 & 6.17-3 & 4.39-2 & 9.99-4 &       - & 8.86-4 		 & & 		 &  1038 &  1430 &  1512 &   453 &      - & 529,239,31,13  & & 		 &      1h &      1h &      1h & 1060.1 &       - &    \bf25.3	 \\

Ex-2(32-128-128) &  0.1 		 & & 		 & 1.92-2 & 9.98-4 & 1.00-3 & 9.99-4 & 4.58-1 & 9.15-4 		 & & 		 &  7476 &  1580 &  3218 &   449  &  564,347,5562 & 339,114,16    & & 		 &      1h &  523.4 & 1047 &  150.1 &      1h &    \bf5.5	 \\

Ex-2(64-256-256) &  0.1 		 & & 		 & 3.39-1 & 2.13-3 & 4.47-2 & 9.99-4 & 4.56-1 & 8.25-4 		 & & 		 &  1038 &  1441 &  1511 &   420 &    560,345,196,854 & 439,214,31,13  & & 		 &      1h & 1h &      1h & 982.8 &      1h &    \bf24.4	 \\
Ex-3(32-128-128) &    0 		 & & 		 & 1.58-2 & 1.19-3 & 1.23-3 & 1.13-3 & 1.62-1 & 8.18-4 		 & & 		 &  7486 & 10759 & 11140 & 10456 &  5001,5001,3469 &  189,69,19    & & 		 &      1h &      1h &      1h &      1h &      1h &     \bf4.9	 \\
Ex-3(64-256-256) &    0 		 & & 		 & 2.41-1 & 3.71-3 & 3.45-2 & 9.99-4 &       - & 4.86-4 		 & & 		 &  1041 & 1442  &  1513 &   138 &      - & 289,99,25,13  & & 		 &      1h &      1h &      1h &  323 &       - &   \bf21.5	 \\
Ex-3(32-128-128) & 0.05 		 & & 		 & 9.35-3 & 9.97-4 & 1.07-3 & 1.00-3 & 1.46-1 & 9.48-4 		 & & 		 &  7542 &  1000 & 11160 &   442 &  5001,5001,7774 &  159,59,16    & & 		 &      1h & 332.1 &      1h &  148.4 &      1h &     \bf4.1	 \\
Ex-3(64-256-256) & 0.05 		 & & 		 & 2.33-1 & 9.98-4 & 3.35-2 & 9.97-4 &       - & 4.76-4 		 & & 		 &  1041 &  1258 &  1515 &   113 &      - & 214,89,25,13  & & 		 &      1h & 3136.3 &      1h &  263.4 &       - &    \bf21.1	 \\

Ex-3(32-128-128) &  0.1 		 & & 		 & 6.48-3 & 9.94-4 & 1.00-3 & 9.99-4 & 1.36-1 & 7.74-4 		 & & 		 &  7483 &   792 &  4451 &   258 & 216,103,14463 &  159,49,16      & & 		 &      1h &  261.5 & 1431.7 &   88.3 &      1h &     \bf4.0	 \\
Ex-3(64-256-256) &  0.1 		 & & 		 & 2.26-1 & 9.95-4 & 3.26-2 & 9.97-4 & 1.35-1 & 4.84-4 		 & & 		 &  1040 &   806 &  1516 &    93 &  215,102,3,2275 & 214,69,25,13  & & 		 &      1h & 2005 &      1h &  216.8 &      1h &    \bf20.7	 \\

Ex-4(32-128-128) &    0 		 & & 		 & 1.13-2 & 2.66-3 & 2.66-3 & 2.64-3 & 2.29-1 & 9.39-4 		 & & 		 &  7508 & 10763 & 11133 & 10450 &  5001,5001,3469 &  264,59,16    & & 		 &      1h &      1h &      1h &      1h &      1h &     \bf4.4	 \\

Ex-4(64-256-256) &    0 		 & & 		 & 3.02-1 & 4.46-3 & 2.98-2 & 1.00-3 &       - & 5.94-4 		 & & 		 &  1037 &  1446 &  1510 &   323 &      - & 364,99,25,13  & & 		 &      1h &      1h &      1h &  769.4 &      - &    \bf22.0	 \\
Ex-4(32-128-128) & 0.05 		 & & 		 & 7.35-3 & 2.38-3 & 2.38-3 & 2.37-3 & 2.06-1 & 9.04-4 		 & & 		 &  7519 & 10772 &  11139 & 10493 &  5001,5001,7785 & 214,59,16    & & 		 &      1h &      1h &      1h &      1h &      1h &     \bf4.4	 \\
Ex-4(64-256-256) & 0.05 		 & & 		 & 2.92-1 & 1.46-3 & 2.98-2 & 1.00-3 &       - & 5.89-4 		 & & 		 &  1038 &  1442 &  1509 &   254 &      - & 264,89,25,13  & & 		 &      1h &      1h &      1h &  598.9 &       - &    \bf21.3	 \\
Ex-4(32-128-128) &  0.1 		 & & 		 & 4.50-3 & 2.15-3 & 2.15-3 & 2.14-3 & 1.92-1 & 7.75-4 		 & & 		 &  7519 & 10772 & 11143 & 10434 & 318,132,14806 & 189,49,16    & & 		 &      1h &      1h &      1h &      1h &      1h &     \bf4.2	 \\
Ex-4(64-256-256) &  0.1 		 & & 		 & 2.82-1 & 1.00-3 & 2.97-2 & 9.98-4 & 1.92-1 & 5.42-4 		 & & 		 &  1039 & 1200  &  1509 &   220 &  317,131,86,2213 & 264,69,25,13  & & 		 &      1h & 2993.2 &      1h &  516.5 &      1h &    \bf21.0	 \\

Ex-5(32-128-128) &    0 		 & & 		 & 1.00-2 & 4.52-3 & 4.38-3 & 4.46-3 &       - & 6.33-4 		 & & 		 &  7503 & 10734 & 11127 & 10448 &      - &  264,69,31    & & 		 &      1h &      1h &      1h &      1h &       - &    \bf6.4	 \\
Ex-5(64-256-256) &    0 		 & & 		 & 3.97-1 & 3.87-3 & 1.71-2 & 2.07-3 &       - & 7.81-4 		 & & 		 &  1038 &  1440 &  1515 &  1521 &      - & 389,129,37,16 & & 		 &      1h &      1h &      1h &      1h &       - &    \bf27.0	 \\
Ex-5(32-128-128) & 0.05 		 & & 		 & 1.25-2 & 3.51-3 & 3.54-3 & 3.35-3 &       - & 9.34-4 		 & & 		 &  7510 & 10764 & 11193 & 10457 &      - & 314,79,25    & & 		 &      1h &      1h &      1h &      1h &       - &    \bf6.0	 \\
Ex-5(64-256-256) & 0.05 		 & & 		 & 3.73-1 & 8.19-3 & 1.65-2 & 1.55-3 &       - & 8.92-4 		 & & 		 &  1039 &  1441 &  1515 &  1525 &      - & 439,144,49,16 & & 		 &      1h &      1h &      1h &      1h &       - &    \bf28.4	 \\
Ex-5(32-128-128) &  0.1 		 & & 		 & 1.41-2 & 3.04-3 & 3.07-3 & 2.87-3 & 1.42-1 & 8.91-4 		 & & 		 &  7508 & 10771 & 11129 & 10459 & 400,429,14848 & 289,69,25    & & 		 &      1h &      1h &      1h &      1h &      1h &    \bf5.7	 \\
Ex-5(64-256-256) &  0.1 		 & & 		 & 3.57-1 & 4.98-3 & 1.60-2 & 1.40-3 & 1.46-1 & 8.80-4 		 & & 		 &  1040 &  1447 &  1514 &  1524 &  405,430,248,1938 & 389,129,43,16 & & 		 &      1h &      1h &      1h &      1h &      1h &    \bf27.4	 \\
    \toprule[1pt]
\end{tabular}

%% file: tables/Table3.tex
\begin{tabular}{lcccccccccccccccccccccccccccc}
\bottomrule[1pt]
\multirow{2}{*}{Problem} 
& \multirow{2}{*}{$\delta$} 
& & & 
\multicolumn{6}{c}{\(\eta_{\text{dot}}\)} & & & 
\multicolumn{6}{c}{Iter} & & & 
\multicolumn{6}{c}{Time (s)} \\
\cline{5-10} \cline{13-18} \cline{21-26}
          &            & &        & DR      & DRc      & PD      & PDc      & M & I                   & &           & DR      & DRc      & PD      & PDc      & M & I             & &           & DR      & DRc      & PD      & PDc      & M & I      \\
    \hline
% DR , DRc , PD , PDc , MFISTA , inPALM	
Ex-1(32-128-128) &    0 		 & & 		 & 2.26-2 & 7.00-4 & 7.11-4 & 7.00-4 & 4.30-1 & 9.42-5 		 & & 		 &  7474 & 10779 & 10969 & 11624 &   5001,5001,3017 &  214,89,69 		 & & 		 &      1h &      1h &      1h &      1h &      1h & \bf 10.3  \\
Ex-1(64-256-256) &    0 		 & & 		 & 3.20-1 & 6.96-3 & 7.21-2 & 2.96-4 & 4.29-1 & 8.30-5 		 & & 		 &  1038 &  1451 &  1502 &  1559 &   5001,5001,5001,438 &  214,89,69,59 		 & & 		 &      1h &      1h &      1h &      1h &      1h & \bf 66.4  \\
Ex-1(32-128-128) & 0.05 		 & & 		 & 2.04-2 & 4.93-4 & 4.97-4 & 4.93-4 & 3.89-1 & 7.30-5 		 & & 		 &  7482 & 10795 & 10994 & 11630 &  5001,5001,7569 &  174,69,59 		 & & 		 &      1h &      1h &      1h &      1h &      1h & \bf 9.0   \\
Ex-1(64-256-256) & 0.05 		 & & 		 & 3.04-1 & 2.44-3 & 7.30-2 & 1.80-4 &       - & 9.39-5 		 & & 		 &  1038 &  1452 &  1500 &  1563 &       - &  159,69,59,43 		 & & 		 &      1h &      1h &      1h &      1h &       - & \bf 50.8  \\
Ex-1(32-128-128) &  0.1 		 & & 		 & 1.74-2 & 3.77-4 & 3.78-4 & 3.77-4 & 4.02-1 & 8.18-5 		 & & 		 &  7504 & 10793 & 11071 & 11637 &  5001,5001,5081 &  144,59,49 		 & & 		 &      1h &      1h &      1h &      1h &      1h & \bf 8.8   \\
Ex-1(64-256-256) &  0.1 		 & & 		 & 2.90-1 & 2.43-4 & 7.27-2 & 1.33-4 & 4.01-1 & 8.36-5 		 & & 		 &  1038 &  1453 &  1511 &  1563 &    5001,5001,3343,11&  144,59,49,37 		 & & 		 &      1h &      1h &      1h &      1h &      1h & \bf 44.5  \\
Ex-2(32-128-128) &    0 		 & & 		 & 1.88-2 & 9.35-4 & 9.70-4 & 9.25-4 &       - & 9.49-5 		 & & 		 &  7466 & 10570 & 11133 & 11529 &       - &  729,529,414 		 & & 		 &      1h &      1h &      1h &      1h &       - & \bf 52.5  \\
Ex-2(64-256-256) &    0 		 & & 		 & 4.02-1 & 1.13-2 & 4.29-2 & 3.65-4 &       - & 9.96-5 		 & & 		 &  1036 &  1438 &  1511 &  1539 &       - &  729,529,414,239 		 & & 		 &      1h &      1h &      1h &      1h &       - & \bf 255.4 \\
Ex-2(32-128-128) & 0.05 		 & & 		 & 1.87-2 & 8.06-4 & 8.09-4 & 8.05-4 &       - & 9.47-5 		 & & 		 &  7474 & 10562 & 11133 & 11622 &       - &  569,439,339 		 & & 		 &      1h &      1h &      1h &      1h &       - & \bf 43.1  \\
Ex-2(64-256-256) & 0.05 		 & & 		 & 3.58-1 & 6.17-3 & 4.98-2 & 3.26-4 &       - & 9.60-5 		 & & 		 &  1038 &  1430 &  1512 &  1560 &       - &  569,439,339,174 		 & & 		 &      1h &      1h &      1h &      1h &       - & \bf 192.9 \\
Ex-2(32-128-128) &  0.1 		 & & 		 & 1.92-2 & 7.99-4 & 8.00-4 & 7.95-4 & 4.58-1 & 9.27-5 		 & & 		 &  7476 & 10803 & 11026 & 11622 &   5001,5001,5092 &  489,364,264 		 & & 		 &      1h &      1h &      1h &      1h &      1h & \bf 35.2  \\
Ex-2(64-256-256) &  0.1 		 & & 		 & 3.39-1 & 2.13-3 & 4.47-2 & 2.95-4 & 4.56-1 & 9.35-5 		 & & 		 &  1038 &  1441 &  1511 &  1561 &   5001,5001,3130,21  &  489,364,264,144 		 & & 		 &      1h &      1h &      1h &      1h &      1h & \bf 161.2 \\
Ex-3(32-128-128) &    0 		 & & 		 & 1.58-2 & 1.19-3 & 1.23-3 & 1.13-3 & 1.62-1 & 9.63-5 		 & & 		 &  7486 & 10759 & 11140 & 10622 &  5001,5001,3469 &  314,174,89 		 & & 		 &      1h &      1h &      1h &      1h &      1h & \bf 14.0  \\
Ex-3(64-256-256) &    0 		 & & 		 & 2.41-1 & 3.71-3 & 3.45-2 & 5.04-4 &       -  & 8.60-5 		 & & 		 &  1041 &  1442 &  1515 &  1538 &       - &  314,174,89,59 		 & & 		 &      1h &      1h &      1h &      1h &       - & \bf 69.2  \\
Ex-3(32-128-128) & 0.05 		 & & 		 & 9.35-3 & 1.05-3 & 1.07-3 & 9.43-4 & 1.46-1 & 9.47-5 		 & & 		 &  7542 & 10777 & 11160 & 11091 & 5001,5001,7774 &  239,129,69 		 & & 		 &      1h &      1h &      1h &      1h &      1h & \bf 11.1  \\
Ex-3(64-256-256) & 0.05 		 & & 		 & 2.33-1 & 6.27-4 & 3.35-2 & 3.64-4 &       - & 9.61-5 		 & & 		 &  1041 &  1446 &  1515 &  1539 &       - &  239,129,69,37 		 & & 		 &      1h &      1h &      1h &      1h &       - & \bf 47.8  \\
Ex-3(32-128-128) &  0.1 		 & & 		 & 6.48-3 & 9.12-4 & 9.28-4 & 8.45-4 & 1.36-1 & 9.27-5 		 & & 		 &  7483 & 10787 & 11167 & 11634 & 867,730,6980 &  214,114,59 		 & & 		 &      1h &      1h &      1h &      1h &      1h & \bf 9.7   \\
Ex-3(64-256-256) &  0.1 		 & & 		 & 2.26-1 & 2.86-4 & 3.26-2 & 2.64-4 & 1.35-1 & 9.12-5 		 & & 		 &  1040 &  1446 &  1516 &  1563 &  854,732,433,408 &  214,114,59,31 		 & & 		 &      1h &      1h &      1h &      1h &      1h & \bf 41.1  \\
Ex-4(32-128-128) &    0 		 & & 		 & 1.13-2 & 2.66-3 & 2.66-3 & 2.64-3 & 2.29-1 & 9.98-5 		 & & 		 &  7508 & 10763 & 11229 & 10450 &   5001,5001,3469 &  414,239,144 		 & & 		 &      1h &      1h &      1h &      1h &      1h & \bf 20.6  \\
Ex-4(64-256-256) &    0 		 & & 		 & 3.02-1 & 4.46-3 & 2.98-2 & 1.21-3 &       - & 9.89-5 		 & & 		 &  1037 &  1446 &  1510 &  1517 &       - &  439,239,144,79 		 & & 		 &      1h &      1h &      1h &      1h &       - & \bf 92.8  \\
Ex-4(32-128-128) & 0.05 		 & & 		 & 7.35-3 & 2.38-3 & 2.38-3 & 2.37-3 & 2.06-1 & 9.99-5 		 & & 		 &  7519 & 10772 & 11133 & 10470 &  5001,5001,7785 &  314,189,114 		 & & 		 &      1h &      1h &      1h &      1h &      1h & \bf 16.9  \\
Ex-4(64-256-256) & 0.05 		 & & 		 & 2.92-1 & 1.46-3 & 2.98-2 & 1.04-3 &       - & 9.27-5 		 & & 		 &  1038 &  1442 &  1510 &  1529 &       - &  314,189,114,69 		 & & 		 &      1h &      1h &      1h &      1h &       - & \bf 79.9  \\
Ex-4(32-128-128) &  0.1 		 & & 		 & 4.50-3 & 2.15-3 & 2.15-3 & 2.14-3 & 1.92-1 & 9.40-5 		 & & 		 &  7519 & 10772 & 11143 & 10595 & 1270,931,6586 &  289,159,99 		 & & 		 &      1h &      1h &      1h &      1h &      1h & \bf 14.5  \\
Ex-4(64-256-256) &  0.1 		 & & 		 & 2.82-1 & 9.29-4 & 2.97-2 & 9.17-4 & 1.92-1 & 9.76-5 		 & & 		 &  1039 &  1444 &  1509 &  1532 &  1269,924,668,306 &  289,159,99,49 		 & & 		 &      1h &      1h &      1h &      1h &      1h & \bf 61.2  \\
Ex-5(32-128-128) &    0 		 & & 		 & 1.00-2 & 4.52-3 & 4.38-3 & 4.46-3 &       - & 9.86-5 		 & & 		 &  7503 & 10734 & 11127 & 10448 &       - &  464,339,264 		 & & 		 &      1h &      1h &      1h &      1h &       - & \bf 34.1  \\
Ex-5(64-256-256) &    0 		 & & 		 & 3.97-1 & 3.87-3 & 1.71-2 & 2.07-3 &       - & 9.66-5 		 & & 		 &  1038 &  1440 &  1515 &  1521 &       - &  464,339,264,174 		 & & 		 &      1h &      1h &      1h &      1h &       - & \bf 184.2 \\
Ex-5(32-128-128) & 0.05 		 & & 		 & 1.25-2 & 3.51-3 & 3.54-3 & 3.35-3 &       - & 9.72-5 		 & & 		 &  7510 & 10764 & 11193 & 10457 &       - &  489,389,314 		 & & 		 &      1h &      1h &      1h &      1h &       - & \bf 40.8  \\
Ex-5(64-256-256) & 0.05 		 & & 		 & 3.73-1 & 8.19-3 & 1.65-2 & 1.55-3 &       - & 9.81-5 		 & & 		 &  1039 &  1441 &  1515 &  1525 &       - &  489,389,314,189 		 & & 		 &      1h &      1h &      1h &      1h &       - & \bf 202.0 \\
Ex-5(32-128-128) &  0.1 		 & & 		 & 1.41-2 & 3.04-3 & 3.07-3 & 2.87-3 & 1.42-1 & 9.45-5 		 & & 		 &  7508 & 10771 & 11229 & 10459 & 1671,2163,4665 &  464,339,289 		 & & 		 &      1h &      1h &      1h &      1h &      1h & \bf 36.4  \\
Ex-5(64-256-256) &  0.1 		 & & 		 & 3.57-1 & 4.98-3 & 1.60-2 & 1.40-3 & 1.46-1 & 9.74-5 		 & & 		 &  1040 &  1447 &  1514 &  1524 &  1642,2158,2321,99 &  439,339,264,174 		 & & 		 &      1h &      1h &      1h &      1h &      1h & \bf 185.2 \\
\toprule[1pt]
\end{tabular}

%% file: tables/Table4.tex
\begin{tabular}{lcccccccccccccccccccccccccccc}
\bottomrule[1pt]
\multirow{2}{*}{Problem} & \multirow{2}{*}{$\delta$} & & & \multicolumn{6}{c}{\(\eta_{\text{dot}}\)} & & & \multicolumn{6}{c}{Iter} & & & \multicolumn{6}{c}{Time (s)} \\
\cline{5-10} \cline{13-18} \cline{21-26}
& & & & P & G & ${A}_{2}$ & ${A}_{15}$ & $S_{2}$ & I & & & P & G & ${A}_{2}$ & ${A}_{15}$ & $S_{2}$ & I & & & P & G & ${A}_{2}$ & ${A}_{15}$ & $S_{2}$ & I \\
\hline
% PALM | ALG2 | acc-pADMM(theta=2) | acc-pADMM(theta=15) | acc-sGS-ADMM(theta=2) | inPALM
Ex-1(32-128-128) & 0    & & & 8.61-5 & 9.99-5 & 8.45-5 & 9.05-5 & 9.71-5 & 9.47-5 & & & 264 & 264 & 339 & 314 & 7202 & 159 & & & 32.9 & 29.4 & 49.0 & 53.6 & 607.7 & \bf 18.4 \\
Ex-1(64-256-256) & 0    & & & 9.37-5 & 9.95-5 & 9.16-5 & 8.60-5 & 2.24-3 & 8.81-5 & & & 214 & 214 & 339 & 314 & 5374 & 144 & & & 210.9 & 186.6 & 383.9 & 426.9 & 1h & \bf 128.3 \\
Ex-1(32-128-128) & 0.05 & & & 8.26-5 & 7.99-5 & 1.06-4 & 1.00-4 & 6.73-5 & 8.97-5 & & & 189 & 214 & 214 & 214 & 5459 & 129 & & & 24.1 & 24.3 & 32.2 & 37.4 & 460.7 & \bf 15.2 \\
Ex-1(64-256-256) & 0.05 & & & 8.14-5 & 8.89-5 & 7.53-5 & 7.34-5 & 2.52-3 & 9.37-5 & & & 174 & 174 & 239 & 239 & 5415 & 114 & & & 173.1 & 153.3 & 275.3 & 319.8 & 1h & \bf 103.2 \\
Ex-1(32-128-128) & 0.1  & & & 8.61-5 & 9.03-5 & 1.03-4 & 1.00-4 & 7.81-5 & 8.14-5 & & & 159 & 174 & 189 & 189 & 4463 & 114 & & & 20.6 & 19.9 & 29.1 & 33.2 & 376.6 & \bf 13.7 \\
Ex-1(64-256-256) & 0.1  & & & 8.83-5 & 7.16-5 & 8.87-5 & 9.39-5 & 2.79-3 & 9.51-5 & & & 144 & 159 & 214 & 189 & 5423 & 99  & & & 144.6 & 140.7 & 249.1 & 256.2 & 1h & \bf 91.2 \\
Ex-2(32-128-128) & 0    & & & 9.59-5 & 9.88-5 & 9.89-5 & 9.63-5 & 9.81-5 & 9.37-5 & & & 1049 & 1369 & 1209 & 1289 & 14423 & 929 & & & 124.5 & 144.1 & 173.9 & 213.3 & 1236.4 & \bf 97.6 \\
Ex-2(64-256-256) & 0    & & & 9.93-5 & 9.62-5 & 9.70-5 & 9.75-5 & 4.04-4 & 9.34-5 & & & 1009 & 1289 & 1089 & 1209 & 5333 & 929 & & & 949.3 & 1061.6 & 1246.9 & 1547.9 & 1h & \bf 769.8 \\
Ex-2(32-128-128) & 0.05 & & & 9.94-5 & 9.68-5 & 9.58-5 & 9.53-5 & 8.15-5 & 9.73-5 & & & 729 & 969 & 889 & 929 & 16415 & 689 & & & 87.0 & 101.5 & 123.6 & 152.0 & 1400.2 & \bf 73.0 \\
Ex-2(64-256-256) & 0.05 & & & 9.45-5 & 9.97-5 & 9.95-5 & 9.64-5 & 7.31-4 & 8.89-5 & & & 729 & 889 & 809 & 849 & 5347 & 689 & & & 688.3 & 734.5 & 893.1 & 1090.5 & 1h & \bf 572.0 \\
Ex-2(32-128-128) & 0.1  & & & 8.65-5 & 9.88-5 & 9.73-5 & 9.89-5 & 9.22-5 & 8.61-5 & & & 649 & 809 & 729 & 769 & 12680 & 609 & & & 77.7 & 84.8 & 102.3 & 127.0 & 1081.4 & \bf 64.7 \\
Ex-2(64-256-256) & 0.1  & & & 8.73-5 & 9.10-5 & 9.99-5 & 8.96-5 & 9.46-4 & 8.96-5 & & & 609 & 769 & 649 & 729 & 5355 & 569 & & & 577.7 & 636.9 & 713.9 & 941.1 & 1h & \bf 475.7 \\
Ex-3(32-128-128) & 0    & & & 9.64-5 & 9.11-5 & 9.58-5 & 9.57-5 & 9.47-5 & 9.86-5 & & & 339 & 464 & 364 & 414 & 5210 & 314 & & & 41.8 & 50.2 & 53.2 & 69.4 & 447.8 & \bf 34.8 \\
Ex-3(64-256-256) & 0    & & & 9.08-5 & 9.92-5 & 9.89-5 & 9.64-5 & 2.31-3 & 9.46-5 & & & 314 & 364 & 339 & 364 & 5340 & 264 & & & 304.4 & 309.4 & 386.3 & 477.1 & 1h & \bf 227.5 \\
Ex-3(32-128-128) & 0.05 & & & 8.90-5 & 9.01-5 & 9.77-5 & 9.33-5 & 6.41-5 & 9.56-5 & & & 264 & 339 & 239 & 314 & 4712 & 239 & & & 33.1 & 37.1 & 38.2 & 53.1 & 403.4 & \bf 26.8 \\
Ex-3(64-256-256) & 0.05 & & & 8.82-5 & 8.09-5 & 9.12-5 & 8.36-5 & 2.47-3 & 8.36-5 & & & 239 & 289 & 264 & 289 & 5358 & 214 & & & 234.2 & 247.6 & 308.3 & 381.5 & 1h & \bf 186.1 \\
Ex-3(32-128-128) & 0.1  & & & 7.49-5 & 9.62-5 & 7.88-5 & 7.99-5 & 8.55-5 & 9.25-5 & & & 239 & 289 & 239 & 289 & 3716 & 214 & & & 30.1 & 31.7 & 38.0 & 49.0 & 318.7 & \bf 24.3 \\
Ex-3(64-256-256) & 0.1  & & & 7.61-5 & 8.59-5 & 8.53-5 & 7.83-5 & 2.58-3 & 8.54-5 & & & 214 & 264 & 239 & 264 & 5361 & 189 & & & 210.8 & 227.2 & 290.7 & 350.0 & 1h & \bf 165.7 \\
Ex-4(32-128-128) & 0    & & & 9.35-5 & 9.21-5 & 9.88-5 & 9.93-5 & 9.68-5 & 9.42-5 & & & 489 & 649 & 529 & 609 & 7202 & 439 & & & 59.6 & 69.2 & 76.5 & 100.9 & 629.6 & \bf 47.4 \\
Ex-4(64-256-256) & 0    & & & 9.46-5 & 9.26-5 & 9.53-5 & 8.79-5 & 2.03-3 & 9.96-5 & & & 439 & 569 & 414 & 529 & 5232 & 389 & & & 421.3 & 477.1 & 469.2 & 686.9 & 1h & \bf 330.1 \\
Ex-4(32-128-128) & 0.05 & & & 9.98-5 & 9.37-5 & 8.98-5 & 9.07-5 & 9.78-5 & 8.86-5 & & & 339 & 464 & 339 & 439 & 6455 & 339 & & & 41.8 & 50.1 & 50.4 & 73.4 & 578.5 & \bf 37.0 \\
Ex-4(64-256-256) & 0.05 & & & 9.57-5 & 9.91-5 & 8.21-5 & 9.61-5 & 2.31-3 & 9.29-5 & & & 339 & 414 & 314 & 389 & 5100 & 314 & & & 327.3 & 350.4 & 361.3 & 508.6 & 1h & \bf 267.9 \\
Ex-4(32-128-128) & 0.1  & & & 8.64-5 & 8.87-5 & 8.52-5 & 9.03-5 & 9.01-5 & 9.52-5 & & & 314 & 414 & 314 & 389 & 6206 & 289 & & & 38.9 & 44.9 & 46.4 & 65.3 & 531.7 & \bf 32.0 \\
Ex-4(64-256-256) & 0.1  & & & 9.87-5 & 8.90-5 & 9.60-5 & 9.45-5 & 2.34-3 & 8.28-5 & & & 289 & 389 & 264 & 339 & 5348 & 289 & & & 280.4 & 328.8 & 305.3 & 445.1 & 1h & \bf 247.4 \\
Ex-5(32-128-128) & 0    & & & 9.82-5 & 9.82-5 & 9.81-5 & 9.90-5 & 1.32-4 & 9.96-5 & & & 529 & 729 & 649 & 689 & 6206 & 529 & & & 64.3 & 77.1 & 99.4 & 113.7 & 545.3 & \bf 57.1 \\
Ex-5(64-256-256) & 0    & & & 9.28-5 & 9.62-5 & 9.25-5 & 9.97-5 & 1.59-3 & 9.47-5 & & & 569 & 769 & 609 & 689 & 5189 & 569 & & & 541.6 & 638.1 & 699.3 & 887.7 & 1h & \bf 476.1 \\
Ex-5(32-128-128) & 0.05 & & & 9.71-5 & 9.68-5 & 9.75-5 & 9.84-5 & 9.71-5 & 9.47-5 & & & 689 & 889 & 729 & 809 & 12680 & 649 & & & 82.4 & 94.0 & 104.1 & 132.5 & 1109.8 & \bf 68.8 \\
Ex-5(64-256-256) & 0.05 & & & 9.32-5 & 9.87-5 & 9.61-5 & 9.50-5 & 1.92-3 & 9.72-5 & & & 729 & 889 & 689 & 809 & 5272 & 649 & & & 688.4 & 733.8 & 774.3 & 1037.6 & 1h & \bf 540.0 \\
Ex-5(32-128-128) & 0.1  & & & 9.97-5 & 9.93-5 & 9.87-5 & 9.40-5 & 9.22-5 & 9.18-5 & & & 609 & 809 & 649 & 769 & 10937 & 609 & & & 73.5 & 85.4 & 93.4 & 126.0 & 976.3 & \bf 65.3 \\
Ex-5(64-256-256) & 0.1  & & & 9.26-5 & 9.87-5 & 9.98-5 & 9.58-5 & 2.33-3 & 9.36-5 & & & 649 & 809 & 569 & 729 & 4995 & 609 & & & 614.6 & 668.8 & 637.4 & 937.1 & 1h & \bf 508.6 \\
\toprule[1pt]
\end{tabular}

%% file: tables/Table5.tex
\begin{tabular}{lcccccccccccccccccccccccccccc}
    \bottomrule[1pt]
\multirow{2}{*}{Problem} & \multirow{2}{*}{$\delta$} & & & \multicolumn{7}{c}{\(\eta_{\text{dot}}\)} & & & \multicolumn{7}{c}{Iter} & & & \multicolumn{7}{c}{Time (s)} \\
\cline{5-11} \cline{14-20} \cline{23-29}
& & & & P & S & G & ${A}_{2}$ & ${A}_{15}$ & $S_{2}$ & I & & & P & S & G & ${A}_{2}$ & ${A}_{15}$ & $S_{2}$ & I & & & P & S & G & ${A}_{2}$ & ${A}_{15}$ & $S_{2}$ & I \\
    \hline
% PALM | sGS-inPALM | ADMM2 | ALG2 | acc-pADMM(theta=2) | acc-pADMM(theta=15) | acc-sGS-ADMM(theta=2) | inPALM
Ex-1(32-128-128) & 0 &  &  & 9.89-5 & 9.42-5 & 9.27-5 & 8.68-5 & 7.97-5 & 6.97-5 & 9.42-5 &  &  & 447 & 4194 & 547 & 767 & 647 & 1704 & 372 &  &  & 12.6 & 212.5 & 11.9 & 40.8 & 29.3 & 115.0 & \bf 10.3  \\

Ex-1(64-256-256) & 0 &  &  & 8.18-5 & 9.89-5 & 9.53-5 & 9.79-5 & 9.37-5 & 9.73-5 & 8.30-5 &  &  & 506 & 2231 & 590 & 981 & 776 & 1490 & 431 &  &  & 76.3 & 768.4 & \bf 54.2 & 302.4 & 210.7 & 678.2 & 66.4  \\

Ex-1(32-128-128) & 0.05 &  &  & 7.54-5 & 8.99-5 & 9.59-5 & 8.03-5 & 8.36-5 & 8.13-5 & 7.30-5 &  &  & 342 & 1982 & 371 & 602 & 462 & 1733 & 302 &  &  & 11.0 & 100.1 & \bf 8.2 & 34.8 & 23.1 & 133.6 & 9.0   \\

Ex-1(64-256-256) & 0.05 &  &  & 9.14-5 & 7.49-5 & 8.49-5 & 8.88-5 & 9.50-5 & 8.65-5 & 9.39-5 &  &  & 385 & 1754 & 402 & 791 & 561 & 1507 & 330 &  &  & 58.8 & 660.8 & \bf 39.9 & 261.7 & 165.6 & 841.8 & 50.8  \\

Ex-1(32-128-128) & 0.1 &  &  & 8.14-5 & 9.66-5 & 8.27-5 & 7.08-5 & 9.52-5 & 9.68-5 & 8.18-5 &  &  & 282 & 1942 & 330 & 547 & 397 & 1693 & 252 &  &  & 9.7 & 100.2 & \bf 7.2 & 33.4 & 19.3 & 133.2 & 8.8   \\

Ex-1(64-256-256) & 0.1 &  &  & 8.44-5 & 8.30-5 & 7.47-5 & 8.66-5 & 9.68-5 & 8.25-5 & 8.36-5 &  &  & 319 & 1704 & 355 & 736 & 486 & 1457 & 289 &  &  & 50.8 & 659.4 & \bf 33.8 & 261.0 & 148.6 & 839.3 & 44.5  \\

Ex-2(32-128-128) & 0 &  &  & 9.64-5 & 6.94-5 & 9.37-5 & 9.67-5 & 9.62-5 & 9.55-5 & 9.49-5 &  &  & 1792 & 3342 & 2387 & 2022 & 2347 & 3840 & 1672 &  &  & 63.5 & 111.1 & 74.3 & 72.4 & 120.0 & 208.4 & \bf 52.5  \\

Ex-2(64-256-256) & 0 &  &  & 9.92-5 & 7.35-5 & 9.90-5 & 9.43-5 & 9.88-5 & 9.43-5 & 9.96-5 &  &  & 1991 & 3494 & 2701 & 2411 & 2761 & 3000 & 1911 &  &  & 298.5 & 711.1 & 341.6 & 519.8 & 678.9 & 723.4 & \bf 255.4 \\

Ex-2(32-128-128) & 0.05 &  &  & 9.41-5 & 8.67-5 & 9.96-5 & 9.95-5 & 9.50-5 & 9.75-5 & 9.47-5 &  &  & 1347 & 3021 & 1792 & 1392 & 1737 & 2024 & 1347 &  &  & 51.5 & 120.3 & 55.9 & 56.2 & 93.0 & 98.7 & \bf 43.1  \\

Ex-2(64-256-256) & 0.05 &  &  & 9.97-5 & 8.95-5 & 9.38-5 & 9.61-5 & 9.25-5 & 9.84-5 & 9.60-5 &  &  & 1546 & 3046 & 2031 & 1681 & 2076 & 2552 & 1521 &  &  & 210.9 & 804.3 & 260.7 & 395.5 & 549.2 & 884.6 & \bf 192.9 \\

Ex-2(32-128-128) & 0.1 &  &  & 9.24-5 & 7.69-5 & 9.38-5 & 9.53-5 & 9.61-5 & 8.34-5 & 9.27-5 &  &  & 1157 & 2866 & 1557 & 1277 & 1542 & 1869 & 1117 &  &  & 41.6 & 119.5 & 46.9 & 51.7 & 78.2 & 97.9 & \bf 35.2  \\

Ex-2(64-256-256) & 0.1 &  &  & 9.46-5 & 9.85-5 & 9.86-5 & 9.30-5 & 9.70-5 & 8.15-5 & 9.35-5 &  &  & 1261 & 2841 & 1731 & 1501 & 1831 & 2594 & 1261 &  &  & 186.1 & 797.8 & 198.8 & 362.8 & 468.7 & 1052.2 & \bf 161.2 \\

Ex-3(32-128-128) & 0 &  &  & 9.73-5 & 8.96-5 & 9.76-5 & 8.19-5 & 9.80-5 & 7.48-5 & 9.63-5 &  &  & 627 & 2287 & 857 & 942 & 847 & 1164 & 577 &  &  & 16.8 & 103.4 & 20.8 & 47.6 & 34.6 & 61.9 & \bf 14.0  \\

Ex-3(64-256-256) & 0 &  &  & 8.61-5 & 6.68-5 & 9.52-5 & 7.81-5 & 9.83-5 & 6.60-5 & 8.60-5 &  &  & 686 & 2089 & 936 & 1256 & 991 & 1422 & 636 &  &  & 80.7 & 674.2 & 94.1 & 395.0 & 237.2 & 555.5 & \bf 69.2  \\

Ex-3(32-128-128) & 0.05 &  &  & 9.93-5 & 6.47-5 & 7.91-5 & 6.49-5 & 8.18-5 & 9.05-5 & 9.47-5 &  &  & 437 & 2615 & 642 & 717 & 587 & 994 & 437 &  &  & 13.2 & 129.1 & 17.8 & 42.7 & 27.9 & 58.9 & \bf 11.1  \\

Ex-3(64-256-256) & 0.05 &  &  & 9.69-5 & 9.09-5 & 8.20-5 & 7.30-5 & 8.43-5 & 4.45-5 & 9.61-5 &  &  & 474 & 1899 & 696 & 981 & 701 & 1652 & 474 &  &  & 55.0 & 667.9 & 81.0 & 354.6 & 189.8 & 828.3 & \bf 47.8  \\

Ex-3(32-128-128) & 0.1 &  &  & 9.58-5 & 9.87-5 & 8.78-5 & 6.58-5 & 9.20-5 & 8.38-5 & 9.27-5 &  &  & 387 & 1828 & 562 & 662 & 532 & 954 & 387 &  &  & 11.6 & 87.6 & 14.5 & 38.0 & 22.4 & 59.0 & \bf 9.7   \\

Ex-3(64-256-256) & 0.1 &  &  & 9.31-5 & 9.82-5 & 8.02-5 & 7.57-5 & 7.89-5 & 4.89-5 & 9.12-5 &  &  & 418 & 1849 & 621 & 901 & 631 & 1602 & 418 &  &  & 47.3 & 661.7 & 70.9 & 320.6 & 164.8 & 828.4 & \bf 41.1  \\

Ex-4(32-128-128) & 0 &  &  & 9.99-5 & 7.09-5 & 9.39-5 & 8.32-5 & 8.78-5 & 9.00-5 & 9.98-5 &  &  & 847 & 2243 & 1137 & 1072 & 1112 & 1369 & 797 &  &  & 24.9 & 90.9 & 31.0 & 49.0 & 49.1 & 62.2 & \bf 20.6  \\

Ex-4(64-256-256) & 0 &  &  & 9.47-5 & 8.80-5 & 9.90-5 & 8.85-5 & 8.89-5 & 8.24-5 & 9.89-5 &  &  & 936 & 2364 & 1251 & 1271 & 1286 & 2117 & 901 &  &  & 118.0 & 679.2 & 133.9 & 390.3 & 292.2 & 851.9 & \bf 92.8  \\

Ex-4(32-128-128) & 0.05 &  &  & 8.97-5 & 9.49-5 & 9.77-5 & 6.63-5 & 9.31-5 & 9.80-5 & 9.99-5 &  &  & 632 & 2501 & 837 & 792 & 772 & 1629 & 617 &  &  & 21.8 & 116.5 & 23.9 & 42.6 & 34.8 & 102.3 & \bf 16.9  \\

Ex-4(64-256-256) & 0.05 &  &  & 9.50-5 & 9.95-5 & 9.77-5 & 9.98-5 & 9.99-5 & 8.82-5 & 9.27-5 &  &  & 691 & 2588 & 926 & 916 & 901 & 2341 & 686 &  &  & 85.2 & 889.4 & 105.5 & 355.4 & 217.7 & 1171.9 & \bf 79.9  \\

Ex-4(32-128-128) & 0.1 &  &  & 9.52-5 & 5.15-5 & 9.84-5 & 8.86-5 & 8.85-5 & 9.32-5 & 9.40-5 &  &  & 547 & 2695 & 732 & 687 & 707 & 1574 & 547 &  &  & 17.6 & 129.6 & 19.8 & 36.0 & 31.4 & 101.5 & \bf 14.5  \\

Ex-4(64-256-256) & 0.1 &  &  & 9.79-5 & 7.38-5 & 9.60-5 & 7.94-5 & 9.46-5 & 9.10-5 & 9.76-5 &  &  & 596 & 2997 & 821 & 806 & 806 & 2503 & 596 &  &  & 71.4 & 1104.0 & 101.5 & 318.8 & 191.7 & 1324.4 & \bf 61.2  \\

Ex-5(32-128-128) & 0 &  &  & 9.54-5 & 9.99-5 & 9.72-5 & 8.98-5 & 9.69-5 & 1.01-4 & 9.86-5 &  &  & 1067 & 2726 & 1477 & 1227 & 1502 & 1479 & 1067 &  &  & 40.7 & 121.0 & 48.6 & 50.3 & 82.6 & 78.4 & \bf 34.1  \\

Ex-5(64-256-256) & 0 &  &  & 9.64-5 & 8.63-5 & 9.58-5 & 9.39-5 & 9.76-5 & 6.61-5 & 9.66-5 &  &  & 1226 & 1935 & 1716 & 1741 & 1741 & 1762 & 1241 &  &  & 200.0 & 465.8 & 254.4 & 385.2 & 411.2 & 623.8 & \bf 184.2 \\

Ex-5(32-128-128) & 0.05 &  &  & 9.73-5 & 5.89-5 & 9.95-5 & 8.04-5 & 9.59-5 & 8.10-5 & 9.72-5 &  &  & 1232 & 3180 & 1632 & 1137 & 1567 & 1684 & 1192 &  &  & 47.6 & 134.9 & 54.4 & 48.9 & 83.6 & 80.8 & \bf 40.8  \\

Ex-5(64-256-256) & 0.05 &  &  & 9.80-5 & 5.93-5 & 9.78-5 & 9.64-5 & 9.64-5 & 7.24-5 & 9.81-5 &  &  & 1421 & 3203 & 1896 & 1766 & 1781 & 2709 & 1381 &  &  & 235.6 & 923.2 & 280.1 & 386.8 & 407.7 & 1065.6 & \bf 202.0 \\

Ex-5(32-128-128) & 0.1 &  &  & 9.89-5 & 6.00-5 & 9.57-5 & 7.42-5 & 9.54-5 & 7.33-5 & 9.45-5 &  &  & 1092 & 3075 & 1502 & 1047 & 1437 & 1579 & 1092 &  &  & 41.1 & 133.9 & 50.7 & 45.2 & 78.1 & 80.5 & \bf 36.4  \\

Ex-5(64-256-256) & 0.1 &  &  & 9.73-5 & 5.17-5 & 9.65-5 & 7.24-5 & 9.64-5 & 6.89-5 & 9.74-5 &  &  & 1241 & 3048 & 1716 & 1601 & 1626 & 2554 & 1216 &  &  & 213.7 & 907.3 & 253.3 & 383.2 & 369.1 & 1061.1 & \bf 185.2 \\
    \bottomrule[1pt]
\end{tabular}

%% file: tables/table_4_1024_short.tex
\begin{tabular}{ccccccc}
\noalign{\smallskip}\hline\noalign{\smallskip}
Category & Problem  & $\eta_{\rm dot}$ & Duality gap  & Iter & Time (s) 
\\
\hline
Images&[1,2,3,4] to [5,6,7,8]  & 5.91-05 & 1.58-06 &  214,79 & 790.2 
\\ 
&[4,3,7,8] to [5,6,1,2] & 8.36-05 & 1.10-06 &  214,69 & 713.7 \\   \hline
Shapes &[2,6,5,3] to [1,4,8,7]& 9.84-05 & 4.01-05  & 1129,314 & 3153.0 
\\ 
&[7,4,6,8] to [3,2,5,1]& 9.82-05 & 3.59-05  & 1129,314 & 3156.2\\ 
\hline
\end{tabular}

%% file: tables/table_weighted.tex
\begin{tabular}{cccccc}
\noalign{\smallskip}\hline\noalign{\smallskip}
 Problem    &Algorithm & $\eta_{\text{dot}}$ &  Iter & Time (s) \\ \hline
 Ex-heart(128-512-512)  & inpALM & 9.66-05 & 1169,339 & \bf 3347.0\\
   & PDc  & 1.57-03 & 249 & 5000.7 \\ \hline 
  Ex-maze(128-128-128)   & inpALM & 9.94-06 & 2849,4449 & \bf 1912.0\\
 &  PDc  & 1.34-03& 4311  & 5000.9 \\ 
\hline\noalign{\smallskip}
\end{tabular}

%% file: DOT-Arxiv.bbl
\begin{thebibliography}{}
	
	\bibitem{soc}
	Alizadeh, F., Goldfarb, D.:
	Second order cone programming. 
	Math. Program. {\bf 95}, 3--51 (2003)
	
	
	\bibitem{cfd}
	Aref, H., Balachandar, S.:
	A First Course in Computational Fluid Dynamics. 
	Cambridge University Press, Cambridge (2017)
	
	
	\bibitem{fista}
	Beck, A., Teboulle, M.:
	A fast iterative shrinkage-thresholding algorithm for linear inverse problems. 
	SIAM J. Imaging Sci. {\bf 2}(1), 183--202 (2009)
	
	
	\bibitem{bb00}
	Benamou, J.-D., Brenier, Y.:
	A computational fluid mechanics solution to the Monge-Kantorovich mass transfer problem. 
	Numer. Math. {\bf 84}(3), 375--393 (2000)
	
	
	
	\bibitem{bb15}
	Benamou, J.-D., Carlier, G.:
	Augmented Lagrangian methods for transport optimization, mean field games and degenerate elliptic equations. 
	J. Optim. Theory Appl. {\bf 167}, 1--26 (2015)
	
	
	\bibitem{bonnans23}
	Bonnans, J.F., Lavigne, P., Pfeiffer, L.:
	Discrete potential mean field games: duality and numerical resolution.
	Math. Program. {\bf 202}, 241--278 (2023)
	
	
	\bibitem{jose22}
	Carrillo, J.A., Craig, K., Wang, L., Wei, C.Z.:
	Primal dual methods for Wasserstein gradient flows.
	Found. Comput. Math. {\bf 22}, 389--443 (2022)
	
	
	\bibitem{chambolle}
	Chambolle, A., Pock, T.:
	A first-order primal-dual algorithm for convex problems with applications to imaging.
	J. Math. Imaging Vis. {\bf 40}, 120--145 (2011)
	
	
	\bibitem{clcoap}
	Chen, L., Sun, D.F., Toh, K.-C.:
	A note on the convergence of ADMM for linearly constrained convex optimization problems.
	Comput. Optim. Appl. {\bf 66}, 327--343 (2017)
	
	\bibitem{cl17}
	Chen, L., Sun, D.F., Toh, K.-C.:
	An efficient inexact symmetric Gauss-Seidel based majorized ADMM for high-dimensional convex composite conic programming. 
	Math. Program. {\bf 161}(1), 237--270 (2017)
	
	\bibitem{cl21}
	Chen, L., Li, X.D., Sun, D.F., Toh, K.-C.:
	On the equivalence of inexact proximal ALM and ADMM for a class of convex composite programming.
	Math. Program. {\bf 185}, 111--161 (2021)
	
	\bibitem{chensocpreview2025}
	Chen, L., Yang, L.X., Zhu, J.Y.:
	A survey on some recent advances in linear and nonlinear second-order cone programming.
	J. Oper. Res. Soc. China (2025), \url{https://doi.org/10.1007/s40305-025-00643-7}
	
	
	
	\bibitem{davis}
	Davis, D., Yin, W.:
	Convergence rate analysis of several splitting schemes.
	In Splitting Methods in Communication, Imaging, Science, and Engineering,
	Glowinski, R., Osher, S.J., Yin, W., eds., Springer, pp. 115--163 (2016)
	
	
	\bibitem{fazel13}
	Fazel, M., Pong, T.K., Sun, D.F., Tseng, P.:
	Hankel matrix rank minimization with applications to system identification and realization.
	SIAM J. Matrix Anal. Appl. {\bf 34}(3), 946--977 (2013)
	
	\bibitem{flamary16}
	Flamary, R., Févotte, C., Courty, N., Emiya, V.:
	Optimal spectral transportation with application to music transcription.
	In: Conference on Neural Information Processing Systems, pp. 703--711 (2016)
	
	\bibitem{fukushima}
	Fukushima, M., Luo, Z.Q., Tseng, P.:
	Smoothing functions for second-order cone complementarity problems.
	SIAM J. Optim. {\bf 12}(2), 436--460 (2002)
	
	\bibitem{gabay76}
	Gabay, D., Mercier, B.:
	A dual algorithm for the solution of nonlinear variational problems via finite element approximation.
	Comput. Math. Appl. {\bf 2}(1), 17--40 (1976)
	
	\bibitem{galichon16}
	Galichon, A.:
	Optimal Transport Methods in Economics.
	Princeton University Press, Princeton (2016)
	
	\bibitem{glowinski75}
	Glowinski, R., Marroco, A.:
	Sur l'approximation, par \'el\'ements finis d'ordre un, et la r\'esolution, par p\'enalisation-dualit\'e d'une classe de probl\`emes de Dirichlet non lin\'eaires.
	Revue fran\cedilla{c}aise d'automatique, Informatique, Recherche  Op\'erationnelle. Analyse Num\'erique {\bf 9}(R2), 41--76 (1975)
	
	
	\bibitem{matrixcomp}
	Golub, G., Van Loan, C.:
	Matrix Computations, 4th edn.
	Johns Hopkins University Press,
	Baltimore (2013)
	
	
	
	\bibitem{hestenes}
	Hestenes, M.:
	Multiplier and gradient methods.
	J. Optim. Theory Appl. {\bf 4}(5), 303--320 (1969)
	
	
	\bibitem{hd24}
	Hou, D., Liang, L., Toh, K.-C.:
	A sparse smoothing Newton method for solving discrete optimal transport problems.
	ACM Trans. Math. Softw. {\bf 50}(3) (2024)
	
	
	\bibitem{Kantorovich42}
	Kantorovich, L.V.:
	On the transfer of masses (in Russian).
	Dokl. Akad. Nauk {\bf 37}(2), 227--229 (1942)
	
	
	
	\bibitem{lam}
	Lam, X.Y., Marron, J.S., Sun, D.F., Toh, K.-C.:
	Fast algorithms for large-scale generalized distance weighted discrimination.
	J. Comput. Graph. Statist. {\bf 27}(2), 368--379 (2018)
	
	
	
	\bibitem{hugo18}
	Lavenant, H., Claici, S., Chien, E., Solomon, J.:
	Dynamical optimal transport on discrete surfaces.
	ACM Trans. Graph. {\bf 37}(6), 1--16 (2018)
	
	\bibitem{hugo20}
	Lavenant, H.:
	Unconditional convergence for discretizations of dynamical optimal transport.
	Math. Comput. {\bf 90}, 739--786 (2021)
	
	
	\bibitem{fast-comp}
	Li, G., Chen, Y.X., Huang, Y., Chi, Y.J., Poor, H.V., Chen, Y.X.:
	Fast computation of optimal transport via entropy-regularized extragradient methods.
	SIAM J. Optim. {\bf 35}(2), 1330--1363 (2025)
	
	
	
	
	\bibitem{lxd16}
	Li, X.D., Sun, D.F., Toh, K.-C.:
	A Schur complement based semi-proximal ADMM for convex quadratic conic programming and extensions.
	Math. Program. {\bf 155}, 333--373 (2016)
	
	\bibitem{lxdsgs}
	Li, X.D., Sun, D.F., Toh, K.-C.:
	A block symmetric Gauss-Seidel decomposition theorem for convex composite quadratic programming and its application.
	Math. Program. {\bf 175}, 395--418 (2019)
	
	\bibitem{ll21}
	Liang, L., Sun, D.F., Toh, K.-C.:
	An inexact augmented Lagrangian method for second-order cone programming with applications.
	SIAM J. Optim. {\bf 31}(3), 1748--1773 (2021)
	
	
	\bibitem{ljl21}
	Liu, J.L., Yin, W.T., Li, W.C., Chow, Y.T.:
	Multilevel optimal transport: a fast approximation of Wasserstein-1 distances.
	SIAM J. Sci. Comput. {\bf 43}(1), A193--A220 (2021)
	
	
	
	\bibitem{dl-geo}
	Liu, S., Ma, S.J., Chen, Y.X., Zha, H.Y., Zhou, H.M.:
	Learning high dimensional Wasserstein geodesics. (2021),
	\href{https://arxiv.org/pdf/2102.02992}{arXiv: 2102.02992}
	
	
	\bibitem{lyl21}
	Liu, Y.L., Xu, Y.B., Yin, W.T.:
	Acceleration of primal-dual methods by preconditioning and simple subproblem procedures.
	J. Sci. Comput. {\bf 86}(2), 1--34 (2021)
	
	
	
	
	
	\bibitem{monge}
	Monge, G.:
	M\'emoire sur la th\'eorie des d\'eblais et des remblais.
	M\'emoires de l'Acad\'emie Royale des Sciences, Paris (1781)
	
	\bibitem{monteiro00}
	Monteiro, R.D.C., Tsuchiya, T.:
	Polynomial convergence of primal-dual algorithms for the second-order cone program based on the MZ-family of directions.
	Math. Program. {\bf 88}, 61--83 (2000)
	
	\bibitem{monteiro}
	Monteiro, R.D.C, Svaiter, B.F.:
	Iteration-complexity of block-decomposition algorithms and the alternating direction method of multipliers.
	SIAM J. Optim. {\bf 23}(1), 475--507 (2013)
	
	
	\bibitem{nata21}
	Natale, A., Todeschi, G.:
	Computation of optimal transport with finite volumes.
	ESAIM-Math. Model. Num. {\bf 55}, 1847--1871 (2021)
	
	\bibitem{papa14}
	Papadakis, N., Peyr\'{e}, G., Oudet, E.:
	Optimal transport with proximal splitting.
	SIAM J. Imaging Sci. {\bf 7}(1), 212--238 (2014)
	
	
	\bibitem{peyre}
	Peyr\'{e}, G.,  Cuturi, M.:
	Computational optimal transport: with applications to data science.
	Found. Trends Mach. Learn. {\bf 11}, 355--607 (2019)
	
	\bibitem{powell}
	Powell, M.:
	A method for nonlinear constraints in minimization problems.
	In: Fletcher, R.(ed.) Optimization, pp. 283--298. 
	Academic Press, New York (1969)
	
	\bibitem{rock-1}
	Rockafellar, R.T.:
	Convex Analysis. Princeton University Press, Princeton (1970)
	
	\bibitem{rock-alm}
	Rockafellar, R.T.:
	Augmented Lagrangians and applications of the proximal point algorithm in convex programming.
	Math. Oper. Res. {\bf 1}, 97--116 (1976)
	
	
	\bibitem{saad}
	Saad, Y.: Iterative Methods for Sparse Linear Systems, 2nd edn.
	Society for Industrial and Applied Mathematics, Philadelphia (2003)
	
	\bibitem{Santamot}
	Santambrogio, F.: Optimal Transport for Applied Mathematicians. Calculus of Variations, PDEs, and Modeling. Birkh\"{a}user, Cham (2015)
	
	\bibitem{dotmark}
	Schrieber, J., Schuhmacher, D., Gottschlich, C.:
	DOTmark--a benchmark for discrete optimal transport.
	IEEE Access {\bf 5}, 271--282 (2017)
	
	\bibitem{accadmm}
	Sun, D.F., Yuan, Y.C., Zhang, G.J., Zhao, X.Y.: Accelerating preconditioned ADMM via degenerate proximal point mappings.
	SIAM J. Optim. {\bf 35}(2), 1165--1193 (2025)
	
	
	\bibitem{tang23}
	Tang, T.Y., Toh, K.-C.:
	Self-adaptive ADMM for semi-strongly convex problems.
	Math. Program. Comput. {\bf 16}, 113--150 (2024)
	
	\bibitem{tart16}
	Tartavel, G., Peyré, G., Gousseau, Y.:
	Wasserstein loss for image synthesis and restoration.
	SIAM J. Imaging Sci. {\bf 9}(4), 1726--1755 (2016)
	
	\bibitem{sdpt3}
	Tutuncu, R.H., Toh, K.-C., Todd, M.J.:
	Solving semidefinite-quadratic-linear programs using SDPT3.
	Math. Program. {\bf 95}, 189--217 (2003)
	
	\bibitem{villani}
	Villani, C.:
	Topics in Optimal Transportation.
	Graduate studies in mathematics, American Mathematical Society, Providence (2003)
	
	
	\bibitem{dl-ot}
	Wan, W., Zhang, Y.J., Bao, C.L., Dong, B., Shi, Z.Q.:
	A scalable deep learning approach for solving high-dimensional dynamic optimal transport.
	SIAM J. Sci. Comput. {\bf 45}(4), B544--B563 (2023)
	
	
	\bibitem{wrd25}
	Wang, R., Zhang, Z.:
	Quadratic-form optimal transport.
	Math. Program. (2025), \url{https://doi.org/10.1007/s10107-025-02282-5}
	
	\bibitem{xiao2018}
	Xiao, Y.H., Chen, L., Li, D.H.: 
	A generalized alternating direction method of multipliers with semi-proximal terms for convex composite conic programming.  
	Math. Program.  Comput. {\bf 10}(4), 533--555 (2018)
	
	
	\bibitem{yl24}
	Yang, L., Liang, L., Chu, T., Toh, K.-C.:
	A corrected inexact proximal augmented Lagrangian method with a relative error criterion for a class of group-quadratic regularized optimal transport problems.
	J. Sci. Comput. {\bf 99}(3), 1--36 (2024)
	
	
	\bibitem{yjjmanifold}
	Yu, J.J., Lai, R.J., Li, W.C., Osher, S.:
	Computational mean-field games on manifolds.
	J. Comput. Phys. {\bf 484}, 112070 (2023)
	
	\bibitem{yjj24}
	Yu, J.J., Lai, R.J., Li, W.C., Osher, S.:
	A fast proximal gradient method and convergence analysis for dynamic mean field planning.
	Math. Comput. {\bf 93}, 603--642 (2024)
	
	\bibitem{hot}
	Zhang, G., Gu, Z., Yuan, Y., Sun, D.F.:
	HOT: an efficient Halpern accelerating algorithm for optimal transport problems.
	IEEE Trans. Pattern Anal. Mach. Intell. {\bf 47}(8), 6703--6714 (2025)
	
	
	
	
\end{thebibliography}
